\documentclass[11pt,a4paper]{amsart}
\usepackage{amssymb,amsthm}
\usepackage{bm}
\usepackage{cite}

\usepackage[truedimen,margin=30truemm]{geometry}

\usepackage[pdftex,colorlinks=true,bookmarks=true,citecolor=blue,linkcolor=blue,pdfauthor={},pdftitle={}]{hyperref}

\theoremstyle{plain}
\newtheorem{theorem}{Theorem}[section]
\newtheorem{lemma}[theorem]{Lemma}
\newtheorem{proposition}[theorem]{Proposition}

\theoremstyle{definition}
\newtheorem{definition}[theorem]{Definition}
\newtheorem{assumption}[theorem]{Assumption}

\theoremstyle{remark}
\newtheorem{remark}[theorem]{Remark}

\numberwithin{equation}{section}

\DeclareMathOperator*{\esup}{ess\,sup}

\begin{document}

\title[Parabolic $p$-Laplace equation in a moving thin domain]{Thin-film limit of the parabolic $p$-Laplace equation in a moving thin domain}

\author[T.-H. Miura]{Tatsu-Hiko Miura}
\address{Graduate School of Science and Technology, Hirosaki University, 3, Bunkyo-cho, Hirosaki-shi, Aomori, 036-8561, Japan}
\email{thmiura623@hirosaki-u.ac.jp}



\begin{abstract}
  We consider the parabolic $p$-Laplace equation with $p>2$ in a moving thin domain under a Neumann type boundary condition corresponding to the total mass conservation.
  When the moving thin domain shrinks to a given closed moving hypersurface as its thickness tends to zero, we rigorously derive a limit problem by showing the weak convergence of the weighted average of a weak solution to the thin-domain problem and characterizing the limit function as a unique weak solution to the limit problem.
  The limit problem obtained in this paper is a system of a nonlinear partial differential equation and an algebraic equation on the moving hypersurface.
  This seems to be somewhat strange, but we also find that the limit problem can be seen as a new kind of local mass conservation law on the moving hypersurface with a normal flux.
\end{abstract}

\maketitle

%%% Section 1 %%%
\section{Introduction} \label{S:Intro}

\subsection{Problem settings and main result} \label{SS:Int_Main}
Let $T\in(0,\infty)$.
For $t\in[0,T]$, let $\Gamma_t$ be a given closed moving hypersurface in $\mathbb{R}^n$ with $n\geq2$, and let $\bm{\nu}(\cdot,t)$ be the unit outward normal vector field of $\Gamma_t$.
We set
\begin{align*}
  S_T := \bigcup_{t\in(0,T)}\Gamma_t\times\{t\}, \quad \overline{S_T} := \bigcup_{t\in[0,T]}\Gamma_t\times\{t\}.
\end{align*}
Let $g_0$ and $g_1$ be smooth functions on $\overline{S_T}$ such that $g:=g_1-g_0\geq c$ on $\overline{S_T}$ with some constant $c>0$ (note that we do not make any assumptions on the signs of $g_0$ and $g_1$).
For $t\in[0,T]$ and a sufficiently small $\varepsilon>0$, we define
\begin{align} \label{E:Def_MTD}
  \Omega_t^\varepsilon := \{y+r\bm{\nu}(y,t) \mid y\in\Gamma_t, \, \varepsilon g_0(y,t)<r<\varepsilon g_1(y,t)\}
\end{align}
and call $\Omega_t^\varepsilon$ a moving thin domain.
Moreover, we set
\begin{align*}
  Q_T^\varepsilon := \bigcup_{t\in(0,T)}\Omega_t^\varepsilon\times\{t\}, \quad \partial_\ell Q_T^\varepsilon := \bigcup_{t\in(0,T)}\partial\Omega_t^\varepsilon\times\{t\},
\end{align*}
and denote by $\partial_{\nu^\varepsilon}$ the outer normal derivative on $\partial\Omega_t^\varepsilon$ and by $V^\varepsilon$ the scalar outer normal velocity of $\partial\Omega_t^\varepsilon$ (see Section \ref{S:Prelim} for the details of the above notations).

Let $p\in(2,\infty)$.
We consider the following parabolic $p$-Laplace equation
\begin{align} \label{E:pLap_MTD}
  \left\{
  \begin{alignedat}{3}
    \partial_tu^\varepsilon-\mathrm{div}(|\nabla u^\varepsilon|^{p-2}\nabla u^\varepsilon) &= f^\varepsilon &\quad &\text{in} &\quad &Q_T^\varepsilon, \\
    |\nabla u^\varepsilon|^{p-2}\partial_{\nu^\varepsilon}u^\varepsilon+V^\varepsilon u^\varepsilon &= 0 &\quad &\text{on} &\quad &\partial_\ell Q_T^\varepsilon, \\
    u^\varepsilon|_{t=0} &= u_0^\varepsilon &\quad &\text{in} &\quad &\Omega_0^\varepsilon.
  \end{alignedat}
  \right.
\end{align}
Here, the Neumann type boundary condition implies the conservation of total mass
\begin{align*}
  \frac{d}{dt}\int_{\Omega_t^\varepsilon}u^\varepsilon\,dx = \int_{\Omega_t^\varepsilon}f^\varepsilon\,dx, \quad t\in(0,T)
\end{align*}
by the Reynolds transport theorem and the divergence theorem.
In other words, the rate of mass on $\partial\Omega_t^\varepsilon$ moving in the normal direction is balanced with the normal flux so that a substance does not go out of or come into $\Omega_t^\varepsilon$ through $\partial\Omega_t^\varepsilon$.

We are concerned with weak solutions to \eqref{E:pLap_MTD} in the framework of evolving Bochner spaces and weak material derivatives introduced in \cite{AlCaDjEl23,AlElSt15_PM}.
Roughly speaking, for a function space $\mathcal{X}(\Omega_t^\varepsilon)$ on the moving domain $\Omega_t^\varepsilon$, we define the evolving Bochner space
\begin{align*}
  L_{\mathcal{X}}^r(Q_T^\varepsilon) = L^r(0,T;\mathcal{X}(\Omega_{(\cdot)}^\varepsilon)), \quad r\in[1,\infty].
\end{align*}
Also, for $p\in(2,\infty)$ and $p'\in(1,2)$ satisfying $1/p+1/p'=1$, we set
\begin{align*}
  \mathbb{W}^{p,p'}(Q_T^\varepsilon) = \{u\in L_{W^{1,p}}^p(Q_T^\varepsilon) \mid \partial_\varepsilon^\bullet u\in L_{[W^{1,p}]^\ast}^{p'}(Q_T^\varepsilon)\},
\end{align*}
where $\partial_\varepsilon^\bullet$ stands for the weak material derivative (i.e., the weak time derivative along a velocity of $\Omega_t^\varepsilon$).
For the precise definition, we refer to Section \ref{S:Boch}.
In what follows, we also use the above notations with $\Omega_t^\varepsilon$ and $Q_T^\varepsilon$ replaced by $\Gamma_t$ and $S_T$, respectively.

In our previous work \cite{Miu25pre_pLap}, we studied the problem \eqref{E:pLap_MTD} in a general moving domain and established the existence and uniqueness of a weak solution $u^\varepsilon\in\mathbb{W}^{p,p'}(Q_T^\varepsilon)$ to \eqref{E:pLap_MTD} for any given data $u_0^\varepsilon\in L^2(\Omega_0^\varepsilon)$ and $f^\varepsilon\in L_{[W^{1,p}]^\ast}^{p'}(Q_T^\varepsilon)$ (see Definition \ref{D:WS_MTD} for the definition of a weak solution to \eqref{E:pLap_MTD}).
The purpose of this paper is to study the thin-film limit of \eqref{E:pLap_MTD}.
More precisely, when $\Omega_t^\varepsilon$ shrinks to $\Gamma_t$ as $\varepsilon\to0$, we intend to show the convergence of the weak solution $u^\varepsilon$ to \eqref{E:pLap_MTD} in an appropriate sense and to derive a limit problem on $\Gamma_t$ satisfied by the limit of $u^\varepsilon$.
Such a thin-film limit problem is important in view of applications, since it is closely related to reduction of dimension as well as modelling of various phenomena in thin domains and their limit sets.

When $p=2$ (the heat equation), we studied the thin-film limit of \eqref{E:pLap_MTD} in \cite{Miu17} and derived a limit problem on $\Gamma_t$ rigorously by means of the weak convergence of the average of a solution in the thin direction and characterization of the weak limit.
The limit problem obtained there is a linear diffusion equation on $\Gamma_t$ which involves the velocity, the mean curvature, and the Laplace--Beltrami operator of $\Gamma_t$.
We also estimated the difference of a solution on $\Omega_t^\varepsilon$ and the one on $\Gamma_t$ in terms of the $L^2$-norm in \cite{Miu17} and of the sup-norm in \cite{Miu23}.
In this paper, we consider the case $p>2$ and focus on derivation of a limit problem of \eqref{E:pLap_MTD}.
Surprisingly, it turns out that the limit problem is a system of a nonlinear partial differential equation and an algebraic equation on $\Gamma_t$.
We will study the difference estimate for solutions to the thin-domain and limit problems in a future work.

To state the main result, let us fix some notations (see Sections \ref{S:Prelim} for details).
Let $\nabla_\Gamma$ and $\mathrm{div}_\Gamma$ be the tangential gradient and the surface divergence on $\Gamma_t$.
We write $\mathbf{v}_\Gamma$ for the total velocity of $\Gamma_t$ and denote its normal and tangential components by
\begin{align*}
  V_\Gamma := \mathbf{v}_\Gamma\cdot\bm{\nu}, \quad \mathbf{v}_\Gamma^\tau := \mathbf{v}_\Gamma-V_\Gamma\bm{\nu} \quad\text{on}\quad \Gamma_t.
\end{align*}
Note that $V_\Gamma$ is scalar valued.
We write $\partial_\Gamma^\bullet$ for the material derivative on $\Gamma_t$.
It is the time derivative along $\mathbf{v}_\Gamma$ and formally written as $\partial_\Gamma^\bullet=\partial_t+\mathbf{v}_\Gamma\cdot\nabla$.
For a function $\eta$ on $\Gamma_t$, let $\bar{\eta}$ be the constant extension of $\eta$ in the normal direction of $\Gamma_t$.
By duality, we also define the ``constant extension'' $\bar{f}$ of $f\in[W^{1,p}(\Gamma_t)]^\ast$ (see \eqref{E:Def_CEFu} for the precise definition).
When $\varphi$ is a function on $\Omega_t^\varepsilon$, we define the weighted average of $\varphi$ in the thin direction by
\begin{align*}
  \mathcal{M}_\varepsilon \varphi(y) := \frac{1}{\varepsilon g(y,t)}\int_{\varepsilon g_0(y,t)}^{\varepsilon g_1(y,t)}\varphi\bigl(y+r\bm{\nu}(y,t)\bigr)J(y,t,r)\,dr, \quad y\in\Gamma_t.
\end{align*}
Here, $J$ is the Jacobian that appears in the change of variables from an integral over $\Omega_t^\varepsilon$ to integrals over $\Gamma_t$ and its normal direction (see \eqref{E:Def_J} and \eqref{E:CoV_MTD}).
Also, let $\partial_\nu\varphi=\bar{\bm{\nu}}\cdot\nabla\varphi$ be the derivative of $\varphi$ in the direction $\bar{\bm{\nu}}$ (i.e., the normal direction of $\Gamma_t$).
As we mentioned before, we use the evolving Bochner spaces
\begin{align*}
  L_{\mathcal{X}}^r(Q_T^\varepsilon), \quad L_{\mathcal{X}}^r(S_T), \quad \mathbb{W}^{p,p'}(Q_T^\varepsilon), \quad \mathbb{W}^{p,p'}(S_T), \quad r\in[1,\infty], \quad p \in (2,\infty)
\end{align*}
with $1/p+1/p'=1$, whose precise definitions are given in Section \ref{S:Boch}.

Now, we are ready to present the main result of this paper.

\begin{theorem} \label{T:TFL}
  For $u_0^\varepsilon\in L^2(\Omega_0^\varepsilon)$ and $f^\varepsilon\in L_{[W^{1,p}]^\ast}^{p'}(Q_T^\varepsilon)$, let $u^\varepsilon\in\mathbb{W}^{p,p'}(Q_T^\varepsilon)$ be a unique weak solution to \eqref{E:pLap_MTD}.
  Suppose that the following conditions are satisfied:
  \begin{itemize}
    \item[(a)] there exists a constant $c>0$ such that
    \begin{align} \label{E:Data_MTD}
      \varepsilon^{-1/2}\|u_0^\varepsilon\|_{L^2(\Omega_0^\varepsilon)} \leq c, \quad \varepsilon^{-1/p'}\|f^\varepsilon\|_{L_{[W^{1,p}]^\ast}^{p'}(Q_T^\varepsilon)} \leq c
    \end{align}
    for all sufficiently small $\varepsilon>0$,
    \item[(b)] there exists a function $v_0\in L^2(\Gamma_0)$ such that
    \begin{align} \label{E:u0Av_WeCo}
      \lim_{\varepsilon\to0}\mathcal{M}_\varepsilon u_0^\varepsilon = v_0 \quad\text{weakly in $L^2(\Gamma_0)$},
    \end{align}
    \item[(c)] there exists a functional $f\in L_{[W^{1,p}]^\ast}^{p'}(S_T)$ such that
    \begin{align} \label{E:ExF_StCo}
      \lim_{\varepsilon\to0}\varepsilon^{-1/p'}\|f^\varepsilon-\bar{f}\|_{L_{[W^{1,p}]^\ast}^{p'}(Q_T^\varepsilon)} = 0.
    \end{align}
  \end{itemize}
  Then, there exist functions $v\in\mathbb{W}^{p,p'}(S_T)$ and $\zeta\in L_{L^p}^p(S_T)$ such that
  \begin{alignat*}{2}
    \lim_{\varepsilon\to0}\mathcal{M}_\varepsilon u^\varepsilon &= v &\quad &\text{weakly in $\mathbb{W}^{p,p'}(S_T)$}, \\
    \lim_{\varepsilon\to0}\mathcal{M}_\varepsilon(\partial_\nu u^\varepsilon) &= \zeta &\quad &\text{weakly in $L_{L^p}^p(S_T)$},
  \end{alignat*}
  and $(v,\zeta)$ is a unique weak solution to the limit problem
  \begin{align} \label{E:pLap_Lim}
    \left\{
    \begin{aligned}
      &\partial_\Gamma^\bullet(gv)+gv\,\mathrm{div}_\Gamma\mathbf{v}_\Gamma \\
      &\qquad -\mathrm{div}_\Gamma\Bigl(g\bigl[(|\nabla_\Gamma v|^2+\zeta^2)^{(p-2)/2}\nabla_\Gamma v+v\mathbf{v}_\Gamma^\tau\bigr]\Bigr) = gf \quad\text{on}\quad S_T, \\
      &(|\nabla_\Gamma v|^2+\zeta^2)^{(p-2)/2}\zeta+V_\Gamma v = 0 \quad\text{on}\quad S_T, \\
      &v|_{t=0} = v_0 \quad\text{on}\quad \Gamma_0.
    \end{aligned}
    \right.
  \end{align}
\end{theorem}

For the definition of a weak solution to \eqref{E:pLap_Lim}, see Definition \ref{D:WS_Lim}.
In Section \ref{S:Outline}, we show the outline of the proof of Theorem \ref{T:TFL}.
The precise proof is given in Section \ref{S:TFLPr}.

By Theorem \ref{T:TFL}, we can also get the following existence (and uniqueness) result.
For the proof, we again refer to Section \ref{S:TFLPr}.

\begin{theorem} \label{T:Lim_UnEx}
  For all $v_0\in L^2(\Gamma_0)$ and $f\in L_{[W^{1,p}]^\ast}^{p'}(S_T)$, there exists a unique weak solution $(v,\zeta)$ to \eqref{E:pLap_Lim}.
\end{theorem}

\subsection{Physical interpretation of the limit problem} \label{SS:Int_PhIn}
Although the original thin-domain problem \eqref{E:pLap_MTD} describes diffusion (or local conservation of mass), the limit problem \eqref{E:pLap_Lim} seems to be less related to physical phenomena due to its complexity.
However, it actually expresses local conservation of mass on $\Gamma_t$.
Let us see it in a simplified setting.

When $g\equiv1$ and $f\equiv0$, the limit problem \eqref{E:pLap_Lim} reduces to
  \begin{align} \label{E:pLim_Redu}
  \left\{
  \begin{aligned}
    &\partial_\Gamma^\bullet v+v\,\mathrm{div}_\Gamma\mathbf{v}_\Gamma-\mathrm{div}_\Gamma\Bigl((|\nabla_\Gamma v|^2+\zeta^2)^{(p-2)/2}\nabla_\Gamma v+v\mathbf{v}_\Gamma^\tau\Bigr) = 0 \quad\text{on}\quad S_T, \\
    &(|\nabla_\Gamma v|^2+\zeta^2)^{(p-2)/2}\zeta+V_\Gamma v = 0 \quad\text{on}\quad S_T.
  \end{aligned}
  \right.
\end{align}
Moreover, since $\mathbf{v}_\Gamma=V_\Gamma\bm{\nu}+\mathbf{v}_\Gamma^\tau$, and since $\mathbf{v}_\Gamma^\tau$ and $\nabla_\Gamma V_\Gamma$ are tangential on $\Gamma_t$,
\begin{align*}
  \partial_\Gamma^\bullet v &= \partial_tv+\mathbf{v}_\Gamma\cdot\nabla v = \partial_t v+(V_\Gamma\bm{\nu})\cdot\nabla v+\mathbf{v}_\Gamma^\tau\cdot\nabla_\Gamma v, \\
  \mathrm{div}_\Gamma\mathbf{v}_\Gamma &= \nabla_\Gamma V_\Gamma\cdot\bm{\nu}+V_\Gamma\,\mathrm{div}_\Gamma\bm{\nu}+\mathrm{div}_\Gamma\mathbf{v}_\Gamma^\tau = -V_\Gamma H+\mathrm{div}_\Gamma\mathbf{v}_\Gamma^\tau, \\
  \mathrm{div}_\Gamma(v\mathbf{v}_\Gamma^\tau) &= \nabla_\Gamma v\cdot\mathbf{v}_\Gamma^\tau+v\,\mathrm{div}_\Gamma\mathbf{v}_\Gamma^\tau,
\end{align*}
where $H=-\mathrm{div}_\Gamma\bm{\nu}$ is the mean curvature of $\Gamma_t$.
Thus, we can write \eqref{E:pLim_Redu} as
\begin{align} \label{E:pLim_Nor}
  \left\{
  \begin{aligned}
    &\partial_\Gamma^\circ v-V_\Gamma Hv = \mathrm{div}_\Gamma\Bigl((|\nabla_\Gamma v|^2+\zeta^2)^{(p-2)/2}\nabla_\Gamma v\Bigr) \quad\text{on}\quad S_T, \\
    &(|\nabla_\Gamma v|^2+\zeta^2)^{(p-2)/2}\zeta+V_\Gamma v = 0 \quad\text{on}\quad S_T,
  \end{aligned}
  \right.
\end{align}
where $\partial_\Gamma^\circ=\partial_t+(V_\Gamma\bm{\nu})\cdot\nabla$ is the normal time derivative (i.e., the time derivative along the normal velocity field $V_\Gamma\bm{\nu}$, see Section \ref{S:Prelim} for details).

First, let us formally set $p=2$.
In this case, the system \eqref{E:pLim_Nor} decouples to
\begin{align} \label{E:Heat_Lim}
  \partial_\Gamma^\circ v-V_\Gamma Hv = \Delta_\Gamma v, \quad \zeta+V_\Gamma v = 0 \quad\text{on}\quad S_T,
\end{align}
where $\Delta_\Gamma$ is the Laplace--Beltrami operator on $\Gamma_t$.
In \cite{Miu17}, the first equation of \eqref{E:Heat_Lim} was derived from the heat equation in $\Omega_t^\varepsilon$ by the thin-film limit, but the second one was not obtained.
As mentioned in \cite{DziEll07} (see also \cite{DziEll13_AN}), the first equation of \eqref{E:Heat_Lim} is a linear diffusion equation on $\Gamma_t$ that corresponds to the local mass conservation of the form
\begin{align*}
  \frac{d}{dt}\int_{\mathcal{D}_t}v\,d\mathcal{H}^{n-1} = -\int_{\partial\mathcal{D}_t}\mathbf{q}\cdot\bm{\mu}\,d\mathcal{H}^{n-2}, \quad \mathbf{q} = -\nabla_\Gamma v
\end{align*}
for any portion $\mathcal{D}_t$ of $\Gamma_t$ moving with normal velocity field $V_\Gamma\bm{\nu}$, where $\mathbf{q}$ is a tangential flux, $\mathcal{H}^k$ is the $k$-dimensional Hausdorff measure, and $\bm{\mu}$ is the unit outer conormal of $\partial\mathcal{D}_t$ (i.e., a unit tangential vector field on $\Gamma_t$ that is outer normal to $\partial\mathcal{D}_t$).
Note that
\begin{align*}
  \frac{d}{dt}\int_{\mathcal{D}_t}v\,d\mathcal{H}^{n-1} = \int_{\mathcal{D}_t}\{\partial_\Gamma^\circ v+\mathrm{div}_\Gamma(vV_\Gamma\bm{\nu})\}\,d\mathcal{H}^{n-1} = \int_{\mathcal{D}_t}(\partial_\Gamma^\circ v-V_\Gamma Hv)\,d\mathcal{H}^{n-1}
\end{align*}
by the Leibniz formula when $\mathcal{D}_t$ moves with velocity $V_\Gamma\bm{\nu}$ (see \cite[Lemma 2.2]{DziEll07}).
However, this description does not explain why the second equation of \eqref{E:Heat_Lim} appears.
In particular, the physical meaning of the function $\zeta$ is totally unclear.

Let us return to the case $p>2$.
To give a physical interpretation of \eqref{E:pLim_Nor}, we consider a surface flux $\mathbf{j}$ on $\Gamma_t$ which may have both the normal and tangential components, and propose the following local mass conservation law:
\begin{align} \label{E:LCM_full}
  \frac{d}{dt}\int_{\mathcal{D}_t}v\,d\mathcal{H}^{n-1} = -\int_{\partial\mathcal{D}_t}\mathbf{j}\cdot\bm{\mu}\,d\mathcal{H}^{n-2}-\int_{\mathcal{D}_t}\mathbf{j}\cdot\bm{\nu}\,d\mathcal{H}^{n-1}+\int_{\mathcal{D}_t}V_\Gamma v\,d\mathcal{H}^{n-1}.
\end{align}
Again, $\mathcal{D}_t$ is any portion of $\Gamma_t$ moving with velocity $V_\Gamma\bm{\nu}$ and $\bm{\mu}$ is the unit outer conormal of $\partial\mathcal{D}_t$.
In the right-hand side, the first term is the tangential mass flow out of $\mathcal{D}_t$ through $\partial\mathcal{D}_t$ by the flux $\mathbf{j}$, which represents the exchange of mass within $\Gamma_t$.
On the other hand, the second term is the normal mass flow leaving $\mathcal{D}_t$ by the flux $\mathbf{j}$, which can be seen as the exchange of mass between $\Gamma_t$ and its outer environment.
The last term stands for the rate of mass on $\mathcal{D}_t$ moving in the normal direction of $\Gamma_t$ with scalar velocity $V_\Gamma$.
Compared to the description in \cite{DziEll07,DziEll13_AN}, the second and third terms are new terms that take into account the normal flux and velocity.
Since $\bm{\mu}$ is the unit outer conormal of $\partial\mathcal{D}_t$ that is tangential on $\Gamma_t$, we see by the surface divergence theorem that
\begin{align*}
  \int_{\partial\mathcal{D}_t}\mathbf{j}\cdot\bm{\mu}\,d\mathcal{H}^{n-2} = \int_{\partial\mathcal{D}_t}\mathbf{j}_\tau\cdot\bm{\mu}\,d\mathcal{H}^{n-2} = \int_{\mathcal{D}_t}\mathrm{div}_\Gamma\mathbf{j}_\tau\,d\mathcal{H}^{n-1},
\end{align*}
where $\mathbf{j}_\tau=\mathbf{j}-(\mathbf{j}\cdot\bm{\nu})\bm{\nu}$ is the tangential component of $\mathbf{j}$.
Using this to \eqref{E:LCM_full}, we get
\begin{align} \label{E:DLCM_gen}
  \partial_\Gamma^\circ v-V_\Gamma Hv = -\mathrm{div}_\Gamma\mathbf{j}_\tau-\mathbf{j}\cdot\bm{\nu}+V_\Gamma v \quad\text{on}\quad S_T.
\end{align}
This equation is the differential form of a local mass conservation law when there is the exchange of mass between $\Gamma_t$ and its outer environment caused by the flux $\mathbf{j}$.

Now, we assume that the rate of mass on $\Gamma_t$ moving in the normal direction is balanced with the normal flux, i.e., $V_\Gamma v=\mathbf{j}\cdot\bm{\nu}$ on $\Gamma_t$, so that a substance on $\Gamma_t$ does not go out of or come from the outer environment.
This situation is similar to the boundary condition of the thin-domain problem \eqref{E:pLap_MTD}.
Under this assumption, the equation \eqref{E:DLCM_gen} reduces to
\begin{align} \label{E:DLCM_spe}
  \partial_\Gamma^\circ v-V_\Gamma Hv = -\mathrm{div}_\Gamma\mathbf{j}_\tau, \quad V_\Gamma v = \mathbf{j}\cdot\bm{\nu} \quad\text{on}\quad S_T.
\end{align}
In particular, setting the flux by $\mathbf{j}=-|\mathbf{b}_{v,\zeta}|^{p-2}\mathbf{b}_{v,\zeta}$ with
\begin{align*}
  \mathbf{b}_{v,\zeta} = \nabla_\Gamma v+\zeta\bm{\nu}, \quad |\mathbf{b}_{v,\zeta}| = (|\nabla_\Gamma v|^2+\zeta^2)^{1/2},
\end{align*}
we get \eqref{E:pLim_Nor} when $p>2$ and \eqref{E:Heat_Lim} when $p=2$ from \eqref{E:DLCM_spe}, and find that $\zeta$ is related to the normal flux.
It seems that the second equation of \eqref{E:DLCM_spe} is implicitly assumed in \cite{DziEll07,DziEll13_AN} as a natural condition keeping a substance on $\Gamma_t$.
Here, we make it an explicit assumption.

\subsection{Literature overview} \label{SS:Int_Liter}
Thin domains arise in various problems like elasticity, lubrication, and geophysical fluid dynamics (see e.g. \cite{Cia97,Cia00,OckOck95,Ped87}).
Since the works by Hale and Raugel \cite{HalRau92_DH,HalRau92_RD}, many authors have studied partial differential equations (PDEs) in flat thin domains around lower dimensional domains.
There are also several works on other shapes of thin domains such as a thin L-shaped domain \cite{HalRau95}, a flat thin domain with holes \cite{PriRyb01}, thin tubes and networks \cite{Kos00,Kos02,Yan90}, and a thin spherical shell \cite{BrDhLe_21,Miu25pre_NSTSS,TemZia97} (see also \cite{Rau95} for examples of thin domains).
Curved thin domains around general (hyper)surfaces and manifolds were considered in the study of the eigenvalue problem of the Laplacian \cite{JimKur16,Kre14,Sch96,Yac18}, a reaction-diffusion equation \cite{PrRiRy02,PriRyb03_Cur}, the Navier--Stokes equations \cite{Miu20_03,Miu21_02,Miu22_01,Miu24_SNS}, the Ginzburg--Landau heat flow \cite{Miu25_GL}, and the Cahn--Hilliard equation \cite{Miu25pre_CH}.
The $p$-Laplace equation in flat thin domains was studied in the elliptic case \cite{ArNaPe21,NakPaz24,NakPer21,NakPer22_NON,NakPer24,NogNak20,PerSil15,Sil13} and the parabolic case \cite{LiFre25,LiFrYu23,LiLiFr24,PuLi25}.
Recently, the authors of \cite{GiLaRy25} introduced an abstract framework of singular limits of gradient flows in Hilbert spaces varying with parameters and applied it to the parabolic $p$-Laplace equation in a flat thin domain.

The above cited papers deals with fixed thin domains that do not move in time.
PDEs in moving thin domains have been studied in a few works.
Elliott and Stinner \cite{EllSti09} considered a moving thin domain around a moving surface to propose a diffuse interface approximation of a given surface advection-diffusion equation (see also \cite{ElStStWe11} for a numerical computation of the diffuse interface model).
Pereira and Silva \cite{PerSil13} studied a reaction-diffusion equation in a flat thin domain with moving top boundary.
In \cite{Miu17}, the present author rigorously derived a limit problem from the heat equation in a moving thin domain around a moving surface.
Moreover, the difference estimates of solutions to the heat equation and the limit problem were established in the $L^2$-norm \cite{Miu17} and the sup-norm \cite{Miu23}.
Fluid and porous medium equations in moving thin domains around moving surfaces were studied in \cite{Miu18,MiGiLi18}.
These works carried out formal calculations of the thin-film limit, but a mathematical justification has not been done due to the difficulty arising from the nonlinearlity.
This paper gives the first rigorous result on the thin-film limit of nonlinear equations in moving thin domains around moving surfaces.

Recently, PDEs on moving surfaces have been widely studied both mathematically and numerically in view of applications.
There are many works on linear diffusion equations on a moving surface (see e.g. \cite{AlElSt15_IFB,DziEll07,DziEll13_AN,DziEll13_MC,EllFri15,OlReXu14,Vie14} and the references cited therein).
Nonlinear equations on moving surfaces were also studied such as a Stefan problem \cite{AlpEll15}, a porous medium equation \cite{AlpEll16}, and the Hamilton--Jacobi equation \cite{DeElMiSt19}.
Fluid equations on moving surfaces were proposed and analyzed in \cite{ArrDeS09,JaOlRe18,KoLiGi17,OlReZh22}.
The Cahn--Hilliard equation on a moving surface was used to model phase separation in binary alloys and cell membranes which evolve in time (see e.g. \cite{EilEll08,YuQuOl20}), and it has been analyzed under various settings of potential and mobility (see e.g. \cite{CaeEll21,CaElGrPo23,EllRan15,EllSal25_RP,EllSal_26_DM,OCoSti16}).
Moreover, the recent paper \cite{EllSal25_NSCH} studied the system of the Navier--Stokes and Cahn--Hilliard equations on a moving surface.
Liquid crystals on moving surfaces were modeled and computed in \cite{NiReVo19,NiReVo20}.
The authors of \cite{AbBuGa23_QP,AbBuGa23_ST} studied a system of the mean curvature flow and a diffusion equation on a moving surface.
Also, the numerical schemes of that system were developed in \cite{DecSty22,ElGaKo22}.
In \cite{YaOzSa16}, the authors introduced a model of nonlinear elasticity in time-dependent ambient spaces like moving surfaces.
As explained in Section \ref{SS:Int_PhIn}, this paper proposes a new kind of diffusion equation on a moving surface by the thin-film limit of the parabolic $p$-Laplace equation.

\subsection{Organization of the paper} \label{SS:Int_Org}
The rest of this paper is organized as follows.
We give the outline of the proof of Theorem \ref{T:TFL} in Section \ref{S:Outline}.
In Section \ref{S:Prelim}, we prepare notations and basic results on a moving hypersurface and thin domain.
Section \ref{S:Boch} presents settings and some facts of function spaces on time-dependent domains.
In Section \ref{S:WSMTD}, we give the definition of a weak solution to the thin-domain problem \eqref{E:pLap_MTD} and derive some estimates for the weak solution.
We study the weighted average operator $\mathcal{M}_\varepsilon$ in Section \ref{S:WeAve}.
Section \ref{S:TFLPr} is devoted to the proof of Theorems \ref{T:TFL} and \ref{T:Lim_UnEx}.
The proofs of some auxiliary results are provided in Section \ref{S:Pf_Aux}.
A short conclusion is made in Section \ref{S:Concl}.

%%% Section 2 %%%
\section{Outline of Proof} \label{S:Outline}
Let us explain the outline, difficulty, and idea of the proof of Theorem \ref{T:TFL}.
We refer to Section \ref{S:TFLPr} for the precise proof.

\subsection{Basic outline} \label{SS:OuLi_Bas}
For the weak solution $u^\varepsilon$ to the thin-domain problem \eqref{E:pLap_MTD}, let
\begin{align*}
  v^\varepsilon := \mathcal{M}_\varepsilon u^\varepsilon, \quad \mathbf{w}^\varepsilon := \mathcal{M}_\varepsilon(|\nabla u^\varepsilon|^{p-2}\nabla u^\varepsilon), \quad \zeta^\varepsilon := \mathcal{M}_\varepsilon(\partial_\nu u^\varepsilon) \quad\text{on}\quad S_T.
\end{align*}
First, we estimate $u^\varepsilon$, $\nabla u^\varepsilon$, and $\partial_\varepsilon^\bullet u^\varepsilon$ explicitly in terms of $\varepsilon$ by using an energy method and the assumption \eqref{E:Data_MTD} on the given data (see Section \ref{S:WSMTD}).
Here, we can apply an energy method as in the case of fixed domains with the aid of the abstract framework of evolving Bochner spaces developed in \cite{AlCaDjEl23,AlElSt15_PM}.
It follows from the resulting estimates and properties of the weighed average operator $\mathcal{M}_\varepsilon$ given in Section \ref{S:WeAve} that
\begin{align*}
  v^\varepsilon\in\mathbb{W}^{p,p'}(S_T), \quad \mathbf{w}^\varepsilon\in [L_{L^{p'}}^{p'}(S_T)]^n, \quad \zeta^\varepsilon\in L_{L^p}^p(S_T)
\end{align*}
and they are bounded in these spaces uniformly in $\varepsilon$.
Thus, up to a subsequence,
\begin{align*}
  v^\varepsilon \to v, \quad \mathbf{w}^\varepsilon \to \mathbf{w}, \quad \zeta^\varepsilon \to \zeta \quad\text{as}\quad \varepsilon\to0 \quad\text{weakly.}
\end{align*}
Next, for a test function $\eta$ on $S_T$ and $t\in(0,T)$, let $\bar{\eta}(\cdot,t)$ be the constant extension of $\eta(\cdot,t)$ in the normal direction of $\Gamma_t$.
We substitute $\bar{\eta}$ for a weak form of the thin-domain problem \eqref{E:pLap_MTD} and take the average in the thin direction by using the relation
\begin{align} \label{E:OuLi_Ave}
  \frac{1}{\varepsilon}\int_{\Omega_t^\varepsilon}u^\varepsilon(x,t)\bar{\eta}(x,t)\,dx = \int_{\Gamma_t}g(y,t)\mathcal{M}_\varepsilon u^\varepsilon(y,t)\eta(y,t)\,d\mathcal{H}^{n-1}(y),
\end{align}
which follows from the change of variables formula \eqref{E:CoV_MTD} and the definition of $\mathcal{M}_\varepsilon$.
Then, we send $\varepsilon\to0$ to get a weak form on $\Gamma_t$ satisfied by the limit functions $v$ and $\mathbf{w}$ (but $\zeta$ does not appear).
After that, we determine $\mathbf{w}$ in terms of $v$ and $\zeta$ to confirm that $(v,\zeta)$ is a weak solution to the limit problem \eqref{E:pLap_Lim}.
Moreover, we show the uniqueness of a weak solution to \eqref{E:pLap_Lim} to complete the proof of Theorem \ref{T:TFL}.
Here, the characterization of $\mathbf{w}$ and the proof of the uniqueness are the difficult parts.
In particular, the characterization of $\mathbf{w}$ requires the most effort and new ideas.

\subsection{Difference from the linear case} \label{SS:OuLi_Lin}
Before showing the main difficulty, let us explain the linear case $p=2$ studied in \cite{Miu17}.
Let $u_2^\varepsilon$ be a weak solution to the heat equation in $\Omega_t^\varepsilon$ (i.e., the problem \eqref{E:pLap_MTD} with $p=2$) and let $v_2^\varepsilon:=\mathcal{M}_\varepsilon u_2^\varepsilon$ on $S_T$.
Also, let $\bar{\eta}$ be the constant extension of a test function $\eta$ on $S_T$.
When $p=2$, a weak form of the heat equation in $\Omega_t^\varepsilon$ involves the standard Dirichlet form on $\Omega_t^\varepsilon$, and we can get
\begin{align*}
  \frac{1}{\varepsilon}\int_0^T\left(\int_{\Omega_t^\varepsilon}\nabla u_2^\varepsilon\cdot\nabla\bar{\eta}\,dx\right)\,dt = \int_0^T\left(\int_{\Gamma_t}g\nabla_\Gamma v_2^\varepsilon\cdot\nabla_\Gamma\eta\,d\mathcal{H}^{n-1}\right)\,dt+R_\varepsilon
\end{align*}
and estimate the error term $R_\varepsilon$ explicitly in terms of $\varepsilon$ to show that $R_\varepsilon\to0$ as $\varepsilon\to0$ by using \eqref{E:OuLi_Ave} and computing the commutator of $\mathcal{M}_\varepsilon$ and $\nabla$ directly.
Thus, we can derive a weak form on $\Gamma_t$ involving only $v_2^\varepsilon$ from that of the heat equation in $\Omega_t^\varepsilon$, and we easily find that the weak limit of $v_2^\varepsilon$ satisfies a weak form of a limit problem just by sending $\varepsilon\to0$.

In the present case $p>2$, however, it is too hard to compute the commutator of $\mathcal{M}_\varepsilon$ and the nonlinear gradient operator $|\nabla|^{p-2}\nabla$.
Moreover, it is not clear how the function $\zeta$ in \eqref{E:pLap_Lim} appears in the averaging process.
Thus, instead of computing the commutator, we characterize the limit vector field $\mathbf{w}$ in terms of $v$ and $\zeta$ after sending $\varepsilon\to0$.

\subsection{Characterization of $\mathbf{w}$} \label{SS:OuLi_Cha}
We first determine the normal component of $\mathbf{w}$.
To this end, we return to the weak form of the thin-domain problem \eqref{E:pLap_MTD} which involves the terms
\begin{align*}
  \int_0^T\left(\int_{\Omega_t^\varepsilon}|\nabla u^\varepsilon|^{p-2}\nabla u^\varepsilon\cdot\nabla\psi\,dx\right)\,dt, \quad \int_0^T\left(\int_{\Omega_t^\varepsilon}u^\varepsilon(\mathbf{v}^\varepsilon\cdot\nabla\psi)\,dx\right)\,dt,
\end{align*}
where $\psi$ is a test function on $Q_T^\varepsilon$ and $\mathbf{v}^\varepsilon$ is a velocity of $\Omega_t^\varepsilon$ which is close to the velocity $\mathbf{v}_\Gamma$ of $\Gamma_t$.
In the weak form of \eqref{E:pLap_MTD} divided by $\varepsilon$, we take the test function
\begin{align*}
  \psi(x,t) = [d\bar{\eta}](x,t) = d(x,t)\bar{\eta}(x,t), \quad (x,t) \in Q_T^\varepsilon,
\end{align*}
where $d(\cdot,t)$ is the signed distance function from $\Gamma_t$ and $\bar{\eta}$ is the constant extension of a test function $\eta$ on $S_T$.
Then, using $\nabla d=\bar{\bm{\nu}}$ in $Q_T^\varepsilon$, $\mathbf{v}_\Gamma\cdot\bm{\nu}=V_\Gamma$ on $S_T$, \eqref{E:OuLi_Ave}, and the estimates for $u^\varepsilon$, we carry out some calculations to find that
\begin{align*}
  \frac{1}{\varepsilon}\int_0^T\left(\int_{\Omega_t^\varepsilon}|\nabla u^\varepsilon|^{p-2}\nabla u^\varepsilon\cdot\nabla[d\bar{\eta}]\,dx\right)\,dt &= \int_0^T\left(\int_{\Gamma_t}g(\mathbf{w}^\varepsilon\cdot\bm{\nu})\eta\,d\mathcal{H}^{n-1}\right)\,dt+R_\varepsilon^1, \\
  \frac{1}{\varepsilon}\int_0^T\left(\int_{\Omega_t^\varepsilon}u^\varepsilon(\mathbf{v}^\varepsilon\cdot\nabla[d\bar{\eta}])\,dx\right)\,dt &= \int_0^T\left(\int_{\Gamma_t}gV_\Gamma v^\varepsilon\eta\,d\mathcal{H}^{n-1}\right)\,dt+R_\varepsilon^2,
\end{align*}
and that the error terms $R_\varepsilon^1$, $R_\varepsilon^2$ and all other terms in the weak form of \eqref{E:pLap_MTD} divided by $\varepsilon$ with test function $d\bar{\eta}$ vanish as $\varepsilon\to0$.
Thus, letting $\varepsilon\to0$ in the weak form, we get
\begin{align*}
  \int_0^T\left(\int_{\Gamma_t}g(\mathbf{w}\cdot\bm{\nu}+V_\Gamma v)\eta\,d\mathcal{H}^{n-1}\right)\,dt = 0
\end{align*}
for any test function $\eta$, which implies $\mathbf{w}\cdot\bm{\nu}+V_\Gamma v=0$ on $S_T$ by $g>0$.
For details of the above discussions, we refer to Proposition \ref{P:Chw_Nor}.

Next, we determine the whole part of $\mathbf{w}$.
We proceed a monotonicity argument as usual, but start from the inequality for the integral on $Q_T^\varepsilon$ of the form
\begin{align} \label{E:OuLi_Mono}
  \frac{1}{\varepsilon}\int_0^T\theta(t)\left(\int_{\Omega_t^\varepsilon}\Bigl(|\nabla u^\varepsilon|^{p-2}\nabla u^\varepsilon-|\bar{\mathbf{z}}|^{p-2}\bar{\mathbf{z}}\Bigr)\cdot(\nabla u^\varepsilon-\bar{\mathbf{z}})\,dx\right)\,dt \geq 0.
\end{align}
Here, $\bar{\mathbf{z}}$ is the constant extension of any vector field $\mathbf{z}$ on $S_T$ which may have both the normal and tangential components.
Also, $\theta$ is any nonnegative function in $C_c^\infty(0,T)$ which depends only on time.
We send $\varepsilon\to0$ in the above inequality.
Then, we easily obtain the convergence of the integrals of
\begin{align*}
  \frac{\theta(t)}{\varepsilon}|\nabla u^\varepsilon|^{p-2}\nabla u^\varepsilon\cdot\bar{\mathbf{z}}, \quad \frac{\theta(t)}{\varepsilon}|\bar{\mathbf{z}}|^{p-2}\bar{\mathbf{z}}\cdot\nabla u^\varepsilon, \quad \frac{\theta(t)}{\varepsilon}|\bar{\mathbf{z}}|^p
\end{align*}
by using the change of variables formula \eqref{E:CoV_MTD}, the relation \eqref{E:OuLi_Ave} and some other properties of $\mathcal{M}_\varepsilon$, and the weak convergence of $v^\varepsilon$, $\mathbf{w}^\varepsilon$, and $\zeta^\varepsilon$.
Here, we emphasize that the function $\zeta$ in \eqref{E:pLap_Lim} appears at this point since
\begin{align*}
  \lim_{\varepsilon\to0}\mathcal{M}_\varepsilon(\nabla u^\varepsilon) = \mathbf{b}_{v,\zeta} = \nabla_\Gamma v+\zeta\bm{\nu} \quad\text{weakly in}\quad [L_{L^p}^p(S_T)]^n
\end{align*}
by properties of $\mathcal{M}_\varepsilon$ and the weak convergence of $v^\varepsilon$ and $\zeta^\varepsilon$.

The most difficult task is the proof of the convergence of
\begin{align} \label{E:OuLi_pGrad}
  \frac{1}{\varepsilon}\int_0^T\theta(t)\left(\int_{\Omega_t^\varepsilon}|\nabla u^\varepsilon|^p\,dx\right)\,dt.
\end{align}
Testing $\theta u^\varepsilon$ in the weak form of the thin-domain problem \eqref{E:pLap_MTD} and carrying out integration by parts with respect to time (see \eqref{E:Tr_Mult} for the precise integration by parts formula), we rewrite the above integral as a linear combination of
\begin{align} \label{E:OuLi_Quad}
  \begin{aligned}
    &\frac{1}{\varepsilon}\int_0^T\frac{d\theta}{dt}(t)\left(\int_{\Omega_t^\varepsilon}|u^\varepsilon|^2\,dx\right)\,dt, \quad \frac{1}{\varepsilon}\int_0^T\theta(t)\left(\int_{\Omega_t^\varepsilon}u^\varepsilon(\mathbf{v}^\varepsilon\cdot\nabla u^\varepsilon)\,dx\right)\,dt, \\
    &\frac{1}{\varepsilon}\int_0^T\theta(t)\left(\int_{\Omega_t^\varepsilon}|u^\varepsilon|^2\,\mathrm{div}\,\mathbf{v}^\varepsilon\,dx\right)\,dt, \quad \frac{1}{\varepsilon}\int_0^T\theta(t)\langle f^\varepsilon,u^\varepsilon\rangle_{W^{1,p}(\Omega_t^\varepsilon)}\,dt.
  \end{aligned}
\end{align}
Here, $\langle\cdot,\cdot\rangle_{W^{1,p}(\Omega_t^\varepsilon)}$ is the duality product between $[W^{1,p}(\Omega_t^\varepsilon)]^\ast$ and $W^{1,p}(\Omega_t^\varepsilon)$.
Note that, if we consider \eqref{E:OuLi_Mono} without the function $\theta$, then we will encounter the time traces
\begin{align*}
  \frac{1}{\varepsilon}\int_{\Omega_0^\varepsilon}|u_0^\varepsilon|^2\,dx, \quad \frac{1}{\varepsilon}\int_{\Omega_T^\varepsilon}|u^\varepsilon(T)|^2\,dx
\end{align*}
when we carry out integration by parts with respect to time.
They are usually handled by the strong convergence of $u_0^\varepsilon$ and the weak convergence of $u^\varepsilon(T)$ (see e.g. \cite{AlCaDjEl23,LaSoUr68,Lio69,Zei90_2B}), but we cannot get such results in the present case since $\mathcal{M}_\varepsilon$ and the square do not commute with each other.
Thus, we include $\theta$ in \eqref{E:OuLi_Mono} and apply $\theta(0)=\theta(T)=0$ to eliminate the time traces.
Moreover, since the integrals \eqref{E:OuLi_Quad} are quadratic with respect to $u^\varepsilon$ and $f^\varepsilon$, the weak convergence of $v^\varepsilon$ and $\mathcal{M}_\varepsilon(\nabla u^\varepsilon)$ is not enough to obtain the convergence of \eqref{E:OuLi_Quad} after taking the average in the thin direction.
In the present case, however, $v^\varepsilon$ is bounded in $\mathbb{W}^{p,p'}(S_T)$ and thus we can apply the Aubin--Lions type lemma (see Lemma \ref{L:AuLi}) shown in \cite{AlCaDjEl23} to observe that $v^\varepsilon$ converges to $v$ strongly in $L_{L^2}^2(S_T)$.
By this fact, the assumption \eqref{E:ExF_StCo} on $f^\varepsilon$, and some other results on $\mathcal{M}_\varepsilon$ and $\mathbf{v}^\varepsilon$, we can prove the convergence of \eqref{E:OuLi_Quad} that implies the convergence of \eqref{E:OuLi_pGrad} as $\varepsilon\to0$.

After the above discussions, we can send $\varepsilon\to0$ in \eqref{E:OuLi_Mono}.
We combine the resulting inequality and the weak form satisfied by $v$ and $\mathbf{w}$ to obtain
\begin{multline*}
  \int_0^T\theta(t)\left(\int_{\Gamma_t}g(\mathbf{w}-|\mathbf{z}|^{p-2}\mathbf{z})\cdot(\mathbf{b}_{v,\zeta}-\mathbf{z})\,d\mathcal{H}^{n-1}\right)\,dt \\
  -\int_0^T\theta(t)\left(\int_{\Gamma_t}g(\mathbf{w}\cdot\bm{\nu}+V_\Gamma v)\zeta\,d\mathcal{H}^{n-1}\right)\,dt \geq 0
\end{multline*}
with $\mathbf{b}_{v,\zeta}=\nabla_\Gamma v+\zeta\bm{\nu}$.
Moreover, the second integral vanishes by $\mathbf{w}\cdot\bm{\nu}+V_\Gamma v=0$ (this is the reason why we first determine the normal component of $\mathbf{w}$).
Hence, we get
\begin{align*}
  \int_0^T\theta(t)\left(\int_{\Gamma_t}g(\mathbf{w}-|\mathbf{z}|^{p-2}\mathbf{z})\cdot(\mathbf{b}_{v,\zeta}-\mathbf{z})\,d\mathcal{H}^{n-1}\right)\,dt \geq 0
\end{align*}
for any vector field $\mathbf{z}$ on $S_T$ and any nonnegative $\theta\in C_c^\infty(0,T)$.
From this inequality, we can derive $\mathbf{w}=|\mathbf{b}_{v,\zeta}|^{p-2}\mathbf{b}_{v,\zeta}$ on $S_T$ by choosing $\mathbf{z}$ and $\theta$ suitably and noting that $g>0$ on $S_T$.
This shows that $(v,\zeta)$ is indeed a weak solution to the limit problem \eqref{E:pLap_Lim}.
We refer to Proposition \ref{P:Chw_All} for details of the above discussions.

\subsection{Uniqueness of a weak solution} \label{SS:OuLi_Uni}
To prove the uniqueness of a weak solution to the limit problem \eqref{E:pLap_Lim}, we consider two weak solutions $(v_1,\zeta_1)$ and $(v_2,\zeta_2)$ to \eqref{E:pLap_Lim} with same data.
If $v_1=v_2$ on $S_T$, then we can deduce that $\zeta_1=\zeta_2$ on $S_T$ by the second equation of \eqref{E:pLap_Lim}, since the mapping
\begin{align*}
  \varphi\colon\mathbb{R}\to\mathbb{R}, \quad \varphi(s) = (a+s^2)^{(p-2)/2}s+b, \quad s\in\mathbb{R}
\end{align*}
is one-to-one (and surjective) for any fixed $a\geq0$ and $b\in\mathbb{R}$.
Thus, we set
\begin{align*}
  v_{\mathrm{d}} = v_1-v_2, \quad \zeta_{\mathrm{d}} = \zeta_1-\zeta_2, \quad \mathbf{b}_i = \nabla_\Gamma v_i+\zeta_i\bm{\nu} \quad\text{on}\quad S_T, \quad i=1,2.
\end{align*}
and prove $v_{\mathrm{d}}=0$ by using the weak forms of $(v_1,\zeta_1)$ and $(v_2,\zeta_2)$.
However, the weak form to \eqref{E:pLap_Lim} only corresponds to the first equation of \eqref{E:pLap_Lim} and the second equation is assumed to be satisfied in the strong form (see Definition \ref{D:WS_Lim}).
Thus, if we take the difference of the weak forms of $(v_1,\zeta_1)$ and $(v_2,\zeta_2)$, then the resulting equality involves the term
\begin{align} \label{E:OuLi_UGT}
  \int_0^T\left(\int_{\Gamma_t}g(|\mathbf{b}_1|^{p-2}\nabla_\Gamma v_1-|\mathbf{b}_2|^{p-2}\nabla_\Gamma v_2)\cdot\nabla_\Gamma\eta\,d\mathcal{H}^{n-1}\right)\,dt,
\end{align}
and we cannot say that this is nonnegative when $\eta=v_{\mathrm{d}}$, since $\mathbf{b}_1$ and $\mathbf{b}_2$ have the normal component.
To overcome this difficulty, we rewrite the second equation of \eqref{E:pLap_Lim} as
\begin{align*}
  \int_0^T\left(\int_{\Gamma_t}g(|\mathbf{b}_i|^{p-2}\zeta_i\bm{\nu}+v_i\mathbf{v}_\Gamma)\cdot(\xi\bm{\nu})\,d\mathcal{H}^{n-1}\right)\,dt = 0
\end{align*}
with a test function $\xi$ on $S_T$ by using $V_\Gamma=\mathbf{v}_\Gamma\cdot\bm{\nu}$, and add this equality to the weak form of the first equation of \eqref{E:pLap_Lim}.
Then, taking the difference of the new weak forms of $(v_1,\zeta_1)$ and $(v_2,\zeta_2)$, we get a weak form of $(v_{\mathrm{d}},\zeta_{\mathrm{d}})$ which involves the term
\begin{align*}
  \int_0^T\left(\int_{\Gamma_t}g(|\mathbf{b}_1|^{p-2}\mathbf{b}_1-|\mathbf{b}_2|^{p-2}\mathbf{b}_2)\cdot(\nabla_\Gamma\eta+\xi\bm{\nu})\,d\mathcal{H}^{n-1}\right)\,dt
\end{align*}
instead of \eqref{E:OuLi_UGT}.
This term is nonnegative when $\eta=v_{\mathrm{d}}$ and $\xi=\zeta_{\mathrm{d}}$.
Thus, it is negligible when we intend to estimate the $L^2(\Gamma_t)$-norm of $v_{\mathrm{d}}$ by an energy method.
However, this choice of test functions generates another problematic term
\begin{align*}
  \int_0^T\left(\int_{\Gamma_t}gv_{\mathrm{d}}(\mathbf{v}_\Gamma\cdot\mathbf{b}_{\mathrm{d}})\,d\mathcal{H}^{n-1}\right)\,dt,
\end{align*}
which cannot be properly absorbed into the integrals of $|v_{\mathrm{d}}|^2$ and $|\mathbf{b}_{\mathrm{d}}|^p$ due to $p\neq2$.
Note that this term appears since $\Gamma_t$ moves in time.
To circumvent this difficulty, we choose test functions appropriately to compute the time derivative of the (mollified) $L^1(\Gamma_t)$-norm of $v_{\mathrm{d}}$ by following the idea of the proof of \cite[Theorem 2.9]{CaNoOr17} which deals with nonlinear parabolic equations in time-dependent domains (see also \cite[Proposition 3.4]{Miu25pre_pLap}).
For details of the above discussions, we refer to Theorem \ref{T:Lim_Uni}.

\subsection{Average of the material derivative} \label{SS:OuLi_AvMat}
Let us also give a comment on the average of the material derivative.
When $\Gamma_t$ and $\Omega_t^\varepsilon$ do not move in time, it is clear that the weighted average $\mathcal{M}_\varepsilon$ commutes with the time derivative.
In the present case, however, we consider the moving sets $\Gamma_t$ and $\Omega_t^\varepsilon$ and use the material time derivatives along their velocities, so it is not trivial to compute the commutator of $\mathcal{M}_\varepsilon$ and the material derivatives, especially in the weak sense.
In Lemma \ref{L:Ave_WeMt}, we carry out this calculation by using properties of $\mathcal{M}_\varepsilon$ and the constant extension in the normal direction of $\Gamma_t$.

%%% Section 3 %%%
\section{Preliminaries} \label{S:Prelim}
In this section, we fix notations and give basic results on surfaces and thin domains.

Throughout this paper, we fix a coordinate system of $\mathbb{R}^n$.
For $i=1,\dots,n$, let $x_i$ be the $i$-th component of $x\in\mathbb{R}^n$ under the fixed coordinate system.
We write $c$ for a general positive constant independent of $t$ and $\varepsilon$.
For a Banach space $\mathcal{X}$, let $[\mathcal{X}]^\ast$ be its dual space and $\langle\cdot,\cdot\rangle_{\mathcal{X}}$ be the duality product between $[\mathcal{X}]^\ast$ and $\mathcal{X}$.

\subsection{Vectors and matrices} \label{SS:Pr_VecMa}
We write a vector $\mathbf{a}\in\mathbb{R}^n$ and a matrix $\mathbf{A}\in\mathbb{R}^{n\times n}$ as
\begin{align*}
  \mathbf{a} =
  \begin{pmatrix}
    a_1 \\ \vdots \\ a_n
  \end{pmatrix}
  = (a_1,\dots,a_n)^{\mathrm{T}}, \quad \mathbf{A} = (A_{ij})_{i,j} =
  \begin{pmatrix}
    A_{11} & \cdots & A_{1n} \\
    \vdots & \ddots & \vdots \\
    A_{n1} & \cdots & A_{nn}.
  \end{pmatrix}
\end{align*}
For $\mathbf{a},\mathbf{b}\in\mathbb{R}^n$, let $\mathbf{a}\cdot\mathbf{b}$, $|\mathbf{a}|$, and $\mathbf{a}\otimes\mathbf{b}=(a_ib_j)_{i,j}$ be the inner product, the norm, and the tensor product of $\mathbf{a}$ and $\mathbf{b}$, respectively.
Also, let $\mathbf{0}_n$ be the zero vector of $\mathbb{R}^n$.
Let $\mathbf{I}_n$ be the $n\times n$ the identity matrix and $\mathbf{A}^{\mathrm{T}}$ be the transpose of $\mathbf{A}\in\mathbb{R}^{n\times n}$.
We write
\begin{align*}
  \mathbf{A}:\mathbf{B} := \mathrm{tr}[\mathbf{A}^{\mathrm{T}}\mathbf{B}], \quad |\mathbf{A}| := \sqrt{\mathbf{A}:\mathbf{A}}, \quad \mathbf{A},\mathbf{B}\in\mathbb{R}^{n\times n}
\end{align*}
for the inner product and the Frobenius norm of matrices.
We will use the following basic inequalities for vectors (see e.g. \cite[Lemma 2.8]{Miu25pre_pLap} for the proof).

\begin{lemma} \label{L:pVec_Ineq}
  Let $p>2$.
  For all $\mathbf{a},\mathbf{b}\in\mathbb{R}^n$, we have
  \begin{gather}
    (|\mathbf{a}|^{p-2}\mathbf{a}-|\mathbf{b}|^{p-2}\mathbf{b})\cdot(\mathbf{a}-\mathbf{b}) \geq 0, \label{E:pVec_Coer} \\
    \bigl|\,|\mathbf{a}|^{p-2}\mathbf{a}-|\mathbf{b}|^{p-2}\mathbf{b}\,\bigr| \leq c(|\mathbf{a}|^{p-2}+|\mathbf{b}|^{p-2})|\mathbf{a}-\mathbf{b}|. \label{E:pVec_Lip}
  \end{gather}
\end{lemma}

\subsection{Moving surface} \label{SS:Pre_Sur}
For $t\in[0,T]$, let $\Gamma_t$ be a closed (i.e., compact and without boundary), connected, and oriented smooth hypersurface in $\mathbb{R}^n$.
We assume that $\Gamma_t$ is the boundary of a bounded domain $\Omega_t$ in $\mathbb{R}^n$.
Let $\bm{\nu}(\cdot,t)$ be the unit outward normal vector field of $\Gamma_t$ that points from $\Omega_t$ into $\mathbb{R}^n\setminus\Omega_t$.
Also, let $d(\cdot,t)$ be the signed distance function from $\Gamma_t$ increasing in the direction of $\bm{\nu}(\cdot,t)$.
We set
\begin{align*}
  S_T := \bigcup_{t\in(0,T)}\Gamma_t\times\{t\}, \quad \overline{S_T} := \bigcup_{t\in[0,T]}\Gamma_t\times\{t\}.
\end{align*}

\begin{assumption} \label{A:Flow_Sur}
  We make the following assumptions on the evolution of $\Gamma_t$.
  \begin{enumerate}
    \item There exists a smooth mapping (a flow map)
    \begin{align*}
      \Phi_{(\cdot)}^0\colon\Gamma_0\times[0,T]\to\mathbb{R}^n, \quad (Y,t) \mapsto \Phi_t^0(Y)
    \end{align*}
    such that $\Phi_0^0$ is the identity mapping on $\Gamma_0$ and $\Phi_t^0\colon\Gamma_0\to\Gamma_t$ is a diffeomorphism for all $t\in[0,T]$.
    We write $\Phi_{-t}^0$ for the inverse mapping of $\Phi_t^0$.
    \item There exist an open neighborhood $\mathcal{O}_T$ of $\overline{S_T}$ in $\mathbb{R}^n\times[0,T]$ and a continuous vector field $\tilde{\bm{\nu}}\colon\mathcal{O}_T\to\mathbb{R}^n$ such that $\tilde{\bm{\nu}}|_{\overline{S_T}}=\bm{\nu}$ on $\overline{S_T}$.
  \end{enumerate}
\end{assumption}

The condition (ii) states that the orientation of $\Gamma_t$ as the boundary of $\Omega_t$ is preserved in time.
By Assumption \ref{A:Flow_Sur}, the following results hold (see Section \ref{SS:PA_Sur} for the proofs).

\begin{lemma} \label{L:MS_ST}
  The sets $S_T$ and $\overline{S_T}$ are smooth hypersurfaces in $\mathbb{R}^{n+1}$.
  Moreover, $\overline{S_T}$ is compact in $\mathbb{R}^{n+1}$.
\end{lemma}

\begin{lemma} \label{L:MS_Reg}
  The mapping $\Phi_{-(\cdot)}^0$ and the vector field $\bm{\nu}$ are smooth on $\overline{S_T}$.
\end{lemma}

\begin{lemma} \label{L:MS_Tubu}
  There exists a constant $\delta>0$ such that the following statement holds: let
  \begin{align*}
    \overline{\mathcal{N}_t} := \{x\in\mathbb{R}^n \mid -\delta \leq d(x,t) \leq \delta\}, \quad t\in[0,T].
  \end{align*}
  Then, for each $t\in[0,T]$ and $x\in\overline{\mathcal{N}_t}$, there exists a unique $\pi(x,t)\in\Gamma_t$ such that
  \begin{align} \label{E:Fermi}
    x = \pi(x,t)+d(x,t)\bm{\nu}(\pi(x,t),t), \quad \nabla d(x,t) = \bm{\nu}(\pi(x,t),t).
  \end{align}
  Moreover, $\pi$ and $d$ are smooth on $\bigcup_{t\in[0,T]}\overline{\mathcal{N}_t}\times\{t\}$.
\end{lemma}

Let $\mathbf{P}:=\mathbf{I}_n-\bm{\nu}\otimes\bm{\nu}$ on $\overline{S_T}$.
The matrix $\mathbf{P}(\cdot,t)$ is the orthogonal projection onto the tangent plane of $\Gamma_t$.
Fix any $t\in[0,T]$.
For a function $\eta$ on $\Gamma_t$, we set
\begin{align*}
  \nabla_\Gamma\eta = (\underline{D}_1\eta,\dots,\underline{D}_n\eta)^{\mathrm{T}} := \mathbf{P}(\cdot,t)\nabla\tilde{\eta} \quad\text{on}\quad \Gamma_t,
\end{align*}
instead of writing $\nabla_{\Gamma_t}$ for simplicity.
Here, $\tilde{\eta}$ is an extension of $\eta$ to $\overline{\mathcal{N}_t}$ and the value of $\nabla_\Gamma\eta$ is independent of the choice of $\tilde{\eta}$.
We call $\nabla_\Gamma$ and $\underline{D}_i$ the tangential gradient and derivative on $\Gamma_t$, respectively.
Next, we define the tangential gradient matrix and the surface divergence of a vector field $\mathbf{v}$ on $\Gamma_t$ by
\begin{align*}
  \nabla_\Gamma\mathbf{v} := (\underline{D}_i\mathrm{v}_j)_{i,j}, \quad \mathrm{div}_\Gamma\mathbf{v} := \mathrm{tr}[\nabla_\Gamma\mathbf{v}] \quad\text{on}\quad \Gamma_t, \quad \mathbf{v} = (\mathrm{v}_1,\dots,\mathrm{v}_n)^{\mathrm{T}}.
\end{align*}
Here, $\mathbf{v}$ is not necessarily tangential and the indices of $\nabla_\Gamma\mathbf{v}$ are reversed in some literature.
We also define the Weingarten map $\mathbf{W}$ and the mean curvature $H$ of $\Gamma_t$ by
\begin{align*}
  \mathbf{W} := -\nabla_\Gamma\bm{\nu}, \quad H := -\mathrm{div}_\Gamma\bm{\nu} \quad\text{on}\quad \overline{S_T},
\end{align*}
which are smooth on $\overline{S_T}$ by Lemma \ref{L:MS_Reg}.

Let $\mathcal{H}^{n-1}$ be the Hausdorff measure of dimension $n-1$.
Fix any $t\in[0,T]$.
It is known (see e.g. \cite[Lemma 16.1]{GilTru01} and \cite[Theorem 2.10]{DziEll13_AN}) that the integration by parts formula
\begin{align} \label{E:IbP_Sur}
  \int_{\Gamma_t}\zeta\underline{D}_i\eta\,d\mathcal{H}^{n-1} = -\int_{\Gamma_t}\eta\{\underline{D}_i\zeta+\zeta [H\nu_i](\cdot,t)\}\,d\mathcal{H}^{n-1}, \quad i=1,\dots,n
\end{align}
holds for $\eta,\zeta\in C^1(\Gamma_t)$.
For $\eta\in L^1(\Gamma_t)$, we define $\underline{D}_i\eta\in L^1(\Gamma_t)$, if exists, by the function satisfying \eqref{E:IbP_Sur} for all $\zeta\in C^1(\Gamma_t)$.
We call $\underline{D}_i\eta$ the weak tangential derivative of $\eta$.
Also, for $q\in[1,\infty]$, we define the Sobolev space
\begin{align*}
  W^{1,q}(\Gamma_t) := \{\eta\in L^q(\Gamma_t) \mid \nabla_\Gamma\eta=(\underline{D}_1\eta,\dots,\underline{D}_n\eta)^{\mathrm{T}}\in L^q(\Gamma_t)^n\}
\end{align*}
and its norm as in the case of flat domains.
By Assumption \ref{A:Flow_Sur}, we have the equivalence of norms in time as follows (see Section \ref{SS:PA_Sur} for the proof).

\begin{lemma} \label{L:CoSur_Lq}
  Let $q\in[1,\infty]$ and $t\in[0,T]$.
  We have
  \begin{align} \label{E:CoSur_Lq}
    c^{-1}\|V\|_{L^q(\Gamma_0)} \leq \|v\|_{L^q(\Gamma_t)} \leq c\|V\|_{L^q(\Gamma_0)}
  \end{align}
  for all $v\in L^q(\Gamma_t)$ and $V:=v\circ\Phi_t^0$ on $\Gamma_0$.
  If in addition $v\in W^{1,q}(\Gamma_t)$, then
  \begin{align} \label{E:CoSur_W1q}
    c^{-1}\|\nabla_\Gamma V\|_{L^q(\Gamma_0)} \leq \|\nabla_\Gamma v\|_{L^q(\Gamma_t)} \leq c\|\nabla_\Gamma V\|_{L^q(\Gamma_0)}.
  \end{align}
  In particular, for the area $|\Gamma_t|$ of $\Gamma_t$, we see by \eqref{E:CoSur_Lq} with $q=1$ and $v\equiv1$ that
  \begin{align} \label{E:Sur_Area}
    c^{-1}|\Gamma_0| \leq |\Gamma_t| \leq c|\Gamma_0| \quad\text{for all}\quad t\in[0,T].
  \end{align}
\end{lemma}

Let $t\in[0,T]$.
For a function $\eta$ on $\Gamma_t$, we define $\bar{\eta}=\eta\circ\pi(\cdot,t)$, i.e.
\begin{align*}
  \bar{\eta}(x) := \eta\bigl(\pi(x,t)\bigr), \quad x\in\overline{\mathcal{N}_t}.
\end{align*}
This is the constant extension of $\eta$ in the normal direction of $\Gamma_t$.
Also, for a function $\eta$ on $\overline{S_T}$, we write $\bar{\eta}=\eta\circ\pi$ or, more precisely,
\begin{align*}
  \bar{\eta}(x,t) := \eta(\pi(x,t),t), \quad (x,t) \in \textstyle\bigcup_{s\in[0,T]}\overline{\mathcal{N}_s}\times\{s\}.
\end{align*}
In what follows, we frequently use these notations without mention.
For example, $\overline{\nabla_\Gamma\eta}$ is the constant extension of $\nabla_\Gamma\eta$.
Taking the constant $\delta$ in Lemma \ref{L:MS_Tubu} small if necessary, we have the following result (see \cite[Lemma 2.2]{Miu23} for the proof).

\begin{lemma} \label{L:CEGr_NB}
  There exists a constant $c>0$ such that
  \begin{align} \label{E:CEGr_NB}
    |\nabla\bar{\eta}(x)| \leq c\Bigl|\overline{\nabla_\Gamma\eta}(x)\Bigr|, \quad \Bigl|\nabla\bar{\eta}(x)-\overline{\nabla_\Gamma\eta}(x)\Bigr| \leq c|d(x,t)|\Bigl|\overline{\nabla_\Gamma\eta}(x)\Bigr|
  \end{align}
  for all $t\in[0,T]$, $x\in\overline{\mathcal{N}_t}$, and functions $\eta$ on $\Gamma_t$.
\end{lemma}

We define a vector field $\mathbf{v}_\Gamma\colon\overline{S_T}\to\mathbb{R}^n$ by
\begin{align} \label{E:Def_vSur}
  \mathbf{v}_\Gamma(y,t) := \partial_t\Phi_t^0\bigl(\Phi_{-t}^0(y)\bigr) = \frac{\partial}{\partial t}\Bigl(\Phi_t^0(Y)\Bigr)\Big|_{Y=\Phi_{-t}^0(y)}, \quad (y,t) \in \overline{S_T},
\end{align}
which is the total velocity of $\Gamma_t$ along the flow map $\Phi_t^0$.
By Assumption \ref{A:Flow_Sur} and Lemma \ref{L:MS_Reg}, we see that $\mathbf{v}_\Gamma$ is smooth on $\overline{S_T}$.
We also set
\begin{align*}
  V_\Gamma(y,t) := [\mathbf{v}_\Gamma\cdot\bm{\nu}](y,t), \quad \mathbf{v}_\Gamma^\tau(y,t) := [\mathbf{P}\mathbf{v}_\Gamma](y,t), \quad (y,t)\in\overline{S_T}
\end{align*}
and call $V_\Gamma$ the scalar outer normal velocity of $\Gamma_t$.
Note that the evolution of $\Gamma_t$ as a subset of $\mathbb{R}^n$ is determined solely by $V_\Gamma$, i.e., the flow map $\Phi_t^\nu$ defined by
\begin{align*}
  \Phi_0^\nu(Y) = Y, \quad \partial_t\Phi_t^\nu(Y) = [V_\Gamma\bm{\nu}](\Phi_t^\nu(Y),t), \quad (Y,t)\in\Gamma_0\times[0,T]
\end{align*}
gives the diffeomorphism $\Phi_t^\nu\colon\Gamma_0\to\Gamma_t$ for each $t\in[0,T]$.

For a function $\eta$ on $\overline{S_T}$, we define the material derivative by
\begin{align} \label{E:Def_MtSur}
  \partial_\Gamma^\bullet\eta(y,t) := \frac{\partial}{\partial t}\Bigl(\eta(\Phi_t^0(Y),t)\Bigr)\Big|_{Y=\Phi_{-t}^0(y)}, \quad (y,t)\in \overline{S_T}.
\end{align}
This is the time derivative along the velocity $\mathbf{v}_\Gamma$.
Indeed, the chain rule gives
\begin{align*}
  \partial_\Gamma^\bullet\eta(y,t) = \frac{\partial}{\partial t}\Bigl(\tilde{\eta}(\Phi_t^0(Y),t)\Bigr)\Big|_{Y=\Phi_{-t}^0(y)} = [\partial_t\tilde{\eta}+\mathbf{v}_\Gamma\cdot\nabla\tilde{\eta}](y,t),
\end{align*}
where $\tilde{\eta}(\cdot,t)$ is any extension of $\eta(\cdot,t)$ to $\overline{\mathcal{N}_t}$.
We can also define and get
\begin{align*}
  \partial_\Gamma^\circ\eta(y,t) := \frac{\partial}{\partial t}\Bigl(\eta(\Phi_t^\nu(Y),t)\Bigr)\Big|_{Y=\Phi_{-t}^\nu(y)} = [\partial_t\tilde{\eta}+(V_\Gamma\bm{\nu})\cdot\nabla\tilde{\eta}](y,t), \quad (y,t)\in \overline{S_T},
\end{align*}
where $\Phi_{-t}^\nu$ is the inverse mapping of $\Phi_t^\nu$.
We call $\partial_\Gamma^\circ\eta$ the normal time derivative of $\eta$.

\subsection{Moving thin domain} \label{SS:Pre_Thin}
Let $g_0$ and $g_1$ be smooth functions on $\overline{S_T}$ such that
\begin{align} \label{E:G_Bdd}
  c^{-1} \leq g := g_1-g_0 \leq c \quad\text{on}\quad \overline{S_T}
\end{align}
with some constant $c>0$.
We do not make any assumptions on the signs of $g_0$ and $g_1$.
For $t\in[0,T]$ and $\varepsilon>0$, let $\Omega_t^\varepsilon$ be the moving thin domain given by \eqref{E:Def_MTD}.
We write $\bm{\nu}^\varepsilon(\cdot,t)$ for the unit outward normal vector field of $\partial\Omega_t^\varepsilon$ and define the outer normal derivative $\partial_{\nu^\varepsilon}:=\bm{\nu}^\varepsilon(\cdot,t)\cdot\nabla$ on $\partial\Omega_t^\varepsilon$.
Also, we set
\begin{align*}
  Q_T^\varepsilon := \bigcup_{t\in(0,T)}\Omega_t^\varepsilon\times\{t\}, \quad \overline{Q_T^\varepsilon} := \bigcup_{t\in[0,T]}\overline{\Omega_t^\varepsilon}\times\{t\}, \quad \partial_\ell Q_T^\varepsilon := \bigcup_{t\in(0,T)}\partial\Omega_t^\varepsilon\times\{t\}.
\end{align*}
Since $g_0$ and $g_1$ are bounded on $\overline{S_T}$, we can take a constant $\varepsilon_0\in(0,1)$ such that $\varepsilon_0|g_i|\leq\delta$ for $i=0,1$, where $\delta$ is the constant given in Lemma \ref{L:MS_Tubu}.
Thus,
\begin{align*}
  \overline{\Omega_t^\varepsilon} \subset \overline{\mathcal{N}_t} \quad\text{for all}\quad \varepsilon\in(0,\varepsilon_0], \quad t\in[0,T].
\end{align*}
In what follows, we always assume $0<\varepsilon\leq\varepsilon_0$.
Then, we can use \eqref{E:Fermi} in $\overline{\Omega_t^\varepsilon}$ to get
\begin{align*}
  d(x,t) = r \in[\varepsilon g_0(y,t),\varepsilon g_1(y,t)] \quad\text{for}\quad x = y+r\bm{\nu}(y,t) \in \overline{\Omega_t^\varepsilon}, \quad (y,t) \in \overline{S_T}.
\end{align*}
Hence, $|d|\leq c\varepsilon$ in $\overline{Q_T^\varepsilon}$.
We frequently use this fact and $0<\varepsilon<1$ below.

Let us give basic facts on integrals over $\Omega_t^\varepsilon$.
We define
\begin{align} \label{E:Def_J}
  J(y,t,r) := \det[\mathbf{I}_n-r\mathbf{W}(y,t)], \quad (y,t) \in \overline{S_T}, \quad r\in[-\delta,\delta].
\end{align}
Since $J(y,t,0)=1$ and $\mathbf{W}$ is smooth on $\overline{S_T}$, we can write
\begin{align} \label{E:J_Poly}
  J(y,t,r) = \sum_{k=0}^nr^kJ_k(y,t), \quad J_0(y,t) \equiv 1, \quad (y,t)\in\overline{S_T}, \quad r\in[-\delta,\delta]
\end{align}
and $J_1,\dots,J_n$ are smooth on $\overline{S_T}$.
Thus, taking $\delta$ small if necessary, we may have
\begin{align} \label{E:J_Bdd}
  c^{-1} \leq J(y,t,r) \leq c \quad\text{for all}\quad (y,t) \in \overline{S_T}, \quad r\in[-\delta,\delta].
\end{align}
The function $J$ is the Jacobian in the change of variables formula
\begin{align} \label{E:CoV_MTD}
  \int_{\Omega_t^\varepsilon}\varphi(x)\,dx = \int_{\Gamma_t}\int_{\varepsilon g_0(y,t)}^{\varepsilon g_1(y,t)}\varphi^\sharp(y,r)J(y,t,r)\,dr\,d\mathcal{H}^{n-1}(y)
\end{align}
for a function $\varphi$ on $\Omega_t^\varepsilon$ (see e.g. \cite[Section 14.6]{GilTru01}), where we set
\begin{align} \label{E:Pull_MTD}
  \varphi^\sharp(y,r) := \varphi\bigl(y+r\bm{\nu}(y,t)\bigr), \quad y\in\Gamma_t, \, r\in[\varepsilon g_0(y,t),\varepsilon g_1(y,t)].
\end{align}
For $q\in[1,\infty)$, it follows from \eqref{E:J_Bdd} and \eqref{E:CoV_MTD} that
\begin{align} \label{E:Lq_MTD}
  c^{-1}\|\varphi\|_{L^q(\Omega_t^\varepsilon)}^q \leq \int_{\Gamma_t}\int_{\varepsilon g_0(y,t)}^{\varepsilon g_1(y,t)}|\varphi^\sharp(y,r)|^q\,dr\,d\mathcal{H}^{n-1}(y) \leq c\|\varphi\|_{L^q(\Omega_t^\varepsilon)}^q.
\end{align}
When $\eta$ is a function on $\Gamma_t$, we see by \eqref{E:CEGr_NB}, \eqref{E:G_Bdd}, and \eqref{E:Lq_MTD} that
\begin{align} \label{E:CE_Lq}
  \|\bar{\eta}\|_{L^q(\Omega_t^\varepsilon)} \leq c\varepsilon^{1/q}\|\eta\|_{L^q(\Gamma_t)}, \quad \|\nabla\bar{\eta}\|_{L^q(\Omega_t^\varepsilon)} \leq c\varepsilon^{1/q}\|\nabla_\Gamma\eta\|_{L^q(\Gamma_t)}.
\end{align}
Let $|\Omega_t^\varepsilon|$ be the volume of $\Omega_t^\varepsilon$.
Then, we have
\begin{align} \label{E:Vol_MTD}
  c^{-1}\varepsilon \leq |\Omega_t^\varepsilon| \leq c\varepsilon \quad\text{for all}\quad t\in[0,T]
\end{align}
by \eqref{E:CoV_MTD} with $\varphi\equiv1$, \eqref{E:Sur_Area}, \eqref{E:G_Bdd}, and \eqref{E:J_Bdd}.
Let us give basic inequalities on $\Omega_t^\varepsilon$.

\begin{lemma} \label{L:L2Lp_MTD}
  Let $p\in(2,\infty)$, $t\in[0,T]$, and $u\in L^p(\Omega_t^\varepsilon)$.
  Then,
  \begin{align} \label{E:L2Lp_MTD}
    \|u\|_{L^2(\Omega_t^\varepsilon)} \leq c\varepsilon^{1/2-1/p}\|u\|_{L^p(\Omega_t^\varepsilon)}.
  \end{align}
\end{lemma}

\begin{proof}
  Noting that $p>2$, we take $q>1$ such that $2/p+1/q=1$.
  Then,
  \begin{align*}
    \|u\|_{L^2(\Omega_t^\varepsilon)} \leq \|u\|_{L^p(\Omega_t^\varepsilon)}|\Omega_t^\varepsilon|^{1/2q} \leq c\varepsilon^{1/2q}\|u\|_{L^p(\Omega_t^\varepsilon)}
  \end{align*}
  by H\"{o}lder's inequality and \eqref{E:Vol_MTD}.
  Since $1/2q=1/2-1/p$, we get \eqref{E:L2Lp_MTD}.
\end{proof}

In what follows, we write $(u,1)_{L^2(\Omega_t^\varepsilon)}=\int_{\Omega_t^\varepsilon}u\,dx$ for $u\in L^1(\Omega_t^\varepsilon)$.
The next lemma shows that Poincar\'{e}'s inequality holds on $\Omega_t^\varepsilon$ uniformly in time and thickness.

\begin{lemma} \label{L:UP_MTD}
  Let $q\in[1,\infty)$.
  There exist constants $\varepsilon_1\in(0,\varepsilon_0]$ and $c>0$ such that
  \begin{align} \label{E:UP_MTD}
    \|u\|_{L^q(\Omega_t^\varepsilon)} \leq c\|\nabla u\|_{L^q(\Omega_t^\varepsilon)}
  \end{align}
  for all $\varepsilon\in(0,\varepsilon_1]$, $t\in[0,T]$, and $u\in W^{1,q}(\Omega_t^\varepsilon)$ satisfying $(u,1)_{L^2(\Omega_t^\varepsilon)}=0$.
\end{lemma}

The proof of Lemma \ref{L:UP_MTD} is given in Section \ref{SS:PA_UP}.
We assume $0<\varepsilon\leq\varepsilon_1$ below.

\begin{lemma} \label{L:Cor_UP}
  Let $q\in[1,\infty)$.
  Then, for all $t\in[0,T]$ and $u\in W^{1,q}(\Omega_t^\varepsilon)$,
  \begin{align} \label{E:Cor_UP}
    \|u\|_{L^q(\Omega_t^\varepsilon)} \leq c\Bigl(\|\nabla u\|_{L^q(\Omega_t^\varepsilon)}+\varepsilon^{-1+1/q}|(u,1)_{L^2(\Omega_t^\varepsilon)}|\Bigr).
  \end{align}
\end{lemma}

\begin{proof}
  Let $c_u^\varepsilon:=|\Omega_t^\varepsilon|^{-1}(u,1)_{L^2(\Omega_t^\varepsilon)}$ and $\varphi:=u-c_u^\varepsilon$ on $\Omega_t^\varepsilon$.
  Then,
  \begin{align*}
    \|u\|_{L^p(\Omega_t^\varepsilon)} &\leq \|\varphi\|_{L^q(\Omega_t^\varepsilon)}+\|c_u^\varepsilon\|_{L^q(\Omega_t^\varepsilon)} = \|\varphi\|_{L^q(\Omega_t^\varepsilon)}+|\Omega_t^\varepsilon|^{-1+1/q}|(u,1)_{L^2(\Omega_t^\varepsilon)}|.
  \end{align*}
  Moreover, since $\varphi\in W^{1,q}(\Omega_t^\varepsilon)$ and $(\varphi,1)_{L^2(\Omega_t^\varepsilon)}=0$, we can use \eqref{E:UP_MTD} to get
  \begin{align*}
    \|\varphi\|_{L^q(\Omega_t^\varepsilon)} \leq c\|\nabla\varphi\|_{L^q(\Omega_t^\varepsilon)} = c\|\nabla u\|_{L^q(\Omega_t^\varepsilon)}.
  \end{align*}
  By these inequalities and \eqref{E:Vol_MTD}, we obtain \eqref{E:Cor_UP}.
\end{proof}

\subsection{Flow map of the moving thin domain} \label{SS:Pre_FMTD}
Let us give a diffeomorphism from $\overline{\Omega_0^\varepsilon}$ to $\overline{\Omega_t^\varepsilon}$.
Recall that $\Phi_{\pm(\cdot)}^0$ are the mappings given in Assumption \ref{A:Flow_Sur}.

We follow the construction of \cite[Lemma 3.2]{Miu17}.
For $X\in\overline{\Omega_0^\varepsilon}$, let
\begin{align*}
  Y := \pi(X,0) \in \Gamma_0, \quad X_i := Y+\varepsilon g_i(Y,0)\bm{\nu}(Y,0) \in \partial\Omega_0^\varepsilon, \quad i = 0,1.
\end{align*}
Then, we see by \eqref{E:Fermi} that $X$ divides the line segment $X_0X_1$ internally, i.e.,
\begin{align} \label{E:X_divide}
  X = (1-\theta)X_0+\theta X_1 = Y+\varepsilon\{(1-\theta)g_0(Y,0)+\theta g_1(Y,0)\}\bm{\nu}(Y,0)
\end{align}
with some $\theta\in[0,1]$.
Based on this observation, for $t\in[0,T]$, we set
\begin{align*}
  y := \Phi_t^0(Y) \in \Gamma_t, \quad x_i := y+\varepsilon g_i(y,t)\bm{\nu}(y,t) \in \partial\Omega_t^\varepsilon, \quad i = 0,1
\end{align*}
and take the point $x\in\overline{\Omega_t^\varepsilon}$ dividing the line segment $x_0x_1$ in the same manner:
\begin{align*}
  x = (1-\theta)x_0+\theta x_1 =y+\varepsilon\{(1-\theta)g_0(y,t)+\theta g_1(y,t)\}\bm{\nu}(y,t).
\end{align*}
Thus, setting $\Phi_t^\varepsilon(X)=x$, we easily find that $\Phi_t^\varepsilon$ is a bijection from $\overline{\Omega_0^\varepsilon}$ onto $\overline{\Omega_t^\varepsilon}$.

To eliminate $\theta$, we see by \eqref{E:Fermi}, \eqref{E:X_divide}, and $|\bm{\nu}(Y,0)|^2=1$ that
\begin{align*}
  d(X,0) = (X-Y)\cdot\bm{\nu}(Y,0) = \varepsilon\{(1-\theta)g_0(Y,0)+\theta g_1(Y,0)\}.
\end{align*}
By this equality and $g_1-g_0=g$ on $\overline{S_T}$, we get
\begin{align*}
  \theta = \frac{d(X,0)-\varepsilon g_0(Y,0)}{\varepsilon g(Y,0)}.
\end{align*}
Thus, we can write $\Phi_t^\varepsilon(X)$ in terms of $X$ and $t$.
Similarly, we can also construct the inverse mapping $\Phi_{-t}^\varepsilon$ of $\Phi_t^\varepsilon$.
In summary, we get the following definitions and properties.

\begin{lemma} \label{L:Flow_MTD}
  For $(X,t) \in \overline{\Omega_0^\varepsilon}\times[0,T]$, we set $Y:=\pi(X,0)\in\Gamma_0$ and
  \begin{align} \label{E:Def_FloMTD}
    \begin{aligned}
      \Phi_t^\varepsilon(X) &:= \Phi_t^0(Y)+r^\varepsilon(X,t)\bm{\nu}(\Phi_t^0(Y),t), \\
      r^\varepsilon(X,t) &:= \frac{g(\Phi_t^0(Y),t)}{g(Y,0)}\{d(X,0)-\varepsilon g_0(Y,0)\}+\varepsilon g_0(\Phi_t^0(Y),t).
    \end{aligned}
  \end{align}
  Also, for $(x,t)\in\overline{Q_T^\varepsilon}$, let $z:=\pi(x,t)\in\Gamma_t$ and
  \begin{align*}
    \Phi_{-t}^\varepsilon(x) &:= \Phi_{-t}^0(z)+R^\varepsilon(x,t)\bm{\nu}(\Phi_{-t}^0(z),0), \\
    R^\varepsilon(x,t) &:= \frac{g(\Phi_{-t}^0(z),0)}{g(z,t)}\{d(x,t)-\varepsilon g_0(z,t)\}+\varepsilon g_0(\Phi_{-t}^0(z),0).
  \end{align*}
  Then, the mappings
  \begin{align*}
    \Phi_{(\cdot)}^\varepsilon\colon\overline{\Omega_0^\varepsilon}\times[0,T] \to \mathbb{R}^n, \quad \Phi_{-(\cdot)}^\varepsilon\colon\overline{Q_T^\varepsilon}\to\mathbb{R}^n
  \end{align*}
  are smooth and satisfy the following conditions:
  \begin{enumerate}
    \item $\Phi_0^\varepsilon$ is the identity mapping on $\overline{\Omega_0^\varepsilon}$,
    \item for each $t\in[0,T]$, the mappings
    \begin{align*}
      \Phi_t^\varepsilon\colon\overline{\Omega_0^\varepsilon}\to\overline{\Omega_t^\varepsilon}, \quad \Phi_t^\varepsilon|_{\partial\Omega_0^\varepsilon}\colon\partial\Omega_0^\varepsilon\to\partial\Omega_t^\varepsilon
    \end{align*}
    are differmorphisms with inverse mappings $\Phi_{-t}^\varepsilon$ and $\Phi_{-t}^\varepsilon|_{\partial\Omega_t^\varepsilon}$,
    \item for all $\alpha\in\mathbb{Z}_{\geq0}^n$ and $k\in\mathbb{Z}_{\geq0}$, we have
    \begin{align} \label{E:FTD_Est}
      \begin{aligned}
        \bigl|\partial_X^\alpha\partial_t^k\Phi_t^\varepsilon(X)\bigr| &\leq c_{\alpha,k}, \quad (X,t)\in\overline{\Omega_0^\varepsilon}\times[0,T], \\
        \bigl|\partial_x^\alpha\partial_t^k\Phi_{-t}^\varepsilon(x)\bigr| &\leq c_{\alpha,k}, \quad (x,t)\in\overline{Q_T^\varepsilon},
      \end{aligned}
    \end{align}
    where $c_{\alpha,k}>0$ is a constant depending on $\alpha$ and $k$ but independent of $\varepsilon$ and $t$.
  \end{enumerate}
\end{lemma}

\begin{proof}
  The statements follow from the constructions of $\Phi_{\pm(\cdot)}^\varepsilon$ and the fact that $\Phi_{\pm(\cdot)}^\varepsilon$ consist of functions which are smooth on the space-time compact sets
  \begin{align*}
    \Gamma_0\times[0,T], \quad \overline{S_T}, \quad \textstyle\bigcup_{t\in[0,T]}\overline{\mathcal{N}_t}\times\{t\}
  \end{align*}
  and independent of $\varepsilon$ (recall also that we assume $0<\varepsilon<1$).
\end{proof}

We define the velocity field $\mathbf{v}^\varepsilon\colon\overline{Q_T^\varepsilon}\to\mathbb{R}^n$ of $\overline{\Omega_t^\varepsilon}$ along the flow map $\Phi_t^\varepsilon$ by
\begin{align} \label{E:Def_vMTD}
  \mathbf{v}^\varepsilon(x,t) := \partial_t\Phi_t^\varepsilon\bigl(\Phi_{-t}^\varepsilon(x)\bigr) = \frac{\partial}{\partial t}\Bigl(\Phi_t^\varepsilon(X)\Bigr)\Big|_{X=\Phi_{-t}^\varepsilon(x)}, \quad (x,t) \in \overline{Q_T^\varepsilon}.
\end{align}
Note that $\mathbf{v}^\varepsilon(\cdot,t)$ is defined on the whole domain $\overline{\Omega_t^\varepsilon}$.
We also define
\begin{align*}
  V^\varepsilon(x,t) := [\mathbf{v}^\varepsilon\cdot\bm{\nu}^\varepsilon](x,t), \quad x\in\partial\Omega_t^\varepsilon, \quad t\in[0,T],
\end{align*}
and call it the scalar outer normal velocity of $\partial\Omega_t^\varepsilon$.

For a function $\varphi$ on $\overline{Q_T^\varepsilon}$, we define the material derivative along the velocity $\mathbf{v}^\varepsilon$ by
\begin{align} \label{E:Def_MtMTD}
  \partial_\varepsilon^\bullet\varphi(x,t) := \frac{\partial}{\partial t}\Bigl(\varphi(\Phi_t^\varepsilon(X),t)\Bigr)\Big|_{X=\Phi_{-t}^\varepsilon(x)}, \quad (x,t) \in \overline{Q_T^\varepsilon}.
\end{align}
By the chain rule, we have $\partial_\varepsilon^\bullet\varphi=\partial_t\varphi+\mathbf{v}^\varepsilon\cdot\nabla\varphi$ on $\overline{Q_T^\varepsilon}$.

The next results give relations between the velocity and the material derivative on $\overline{\Omega_t^\varepsilon}$ and the corresponding ones on $\Gamma_t$.
For the proofs, we refer to Section \ref{SS:PA_VMTD}.

\begin{lemma} \label{L:Vls_MTD}
  We define
  \begin{align*}
    \mathbf{G}_{\Gamma,g} := \nabla_\Gamma\mathbf{v}_\Gamma+\bm{\nu}\otimes\left[\frac{\partial_\Gamma^\bullet g}{g}\bm{\nu}+\partial_\Gamma^\bullet\bm{\nu}\right], \quad \sigma_{\Gamma,g} := \mathrm{div}_\Gamma\mathbf{v}_\Gamma+\frac{\partial_\Gamma^\bullet g}{g} \quad\text{on}\quad \overline{S_T}.
  \end{align*}
  Also, let $\nabla\mathbf{v}^\varepsilon=(\partial_i\mathrm{v}_j^\varepsilon)_{i,j}$ in $\overline{Q_T^\varepsilon}$ for $\mathbf{v}^\varepsilon=(\mathrm{v}_1^\varepsilon,\cdots,\mathrm{v}_n^\varepsilon)^{\mathrm{T}}$.
  Then,
  \begin{align} \label{E:Vls_MTD}
    |\mathbf{v}^\varepsilon-\bar{\mathbf{v}}_\Gamma| \leq c\varepsilon, \quad \Bigl|\nabla\mathbf{v}^\varepsilon-\overline{\mathbf{G}}_{\Gamma,g}\Bigr| \leq c\varepsilon, \quad |\mathrm{div}\,\mathbf{v}^\varepsilon-\bar{\sigma}_{\Gamma,g}| \leq c\varepsilon \quad\text{in}\quad \overline{Q_T^\varepsilon}.
  \end{align}
  In particular, since $\mathbf{v}_\Gamma$, $\bm{\nu}$, and $g$ are smooth on $\overline{S_T}$ and \eqref{E:G_Bdd} holds, we have
  \begin{align} \label{E:BdV_MTD}
    |\mathbf{v}^\varepsilon| \leq c, \quad |\nabla\mathbf{v}^\varepsilon| \leq c, \quad |\mathrm{div}\,\mathbf{v}^\varepsilon| \leq c \quad\text{in}\quad \overline{Q_T^\varepsilon}.
  \end{align}
\end{lemma}

\begin{lemma} \label{L:Dist_Mat}
  For all $(x,t)\in\overline{Q_T^\varepsilon}$, we have
  \begin{align} \label{E:Dist_Mat}
    \partial_\varepsilon^\bullet d(x,t) = \frac{\overline{\partial_\Gamma^\bullet g}(x,t)}{\bar{g}(x,t)}\{d(x,t)-\varepsilon\bar{g}_0(x,t)\}+\varepsilon\overline{\partial_\Gamma^\bullet g_0}(x,t).
  \end{align}
  Thus, by $g_0,g\in C^\infty(\overline{S_T})$, \eqref{E:G_Bdd}, and $|d|\leq c\varepsilon$ in $\overline{Q_T^\varepsilon}$, we get
  \begin{align} \label{E:DisMt_Bdd}
    |\partial_\varepsilon^\bullet d| \leq c\varepsilon \quad\text{in}\quad \overline{Q_T^\varepsilon}.
  \end{align}
  Also, let $\eta$ be a function on $\overline{S_T}$.
  Then,
  \begin{align} \label{E:CE_Mat}
    \partial_\varepsilon^\bullet\bar{\eta}(x,t) = \overline{\partial_\Gamma^\bullet\eta}(x,t) \quad\text{for all}\quad (x,t)\in\overline{Q_T^\varepsilon}.
  \end{align}
\end{lemma}

%%% Section 4 %%%
\section{Evolving Bochner spaces} \label{S:Boch}
Let us introduce function spaces on the moving sets $\Gamma_t$ and $\Omega_t^\varepsilon$.
We adopt the abstract framework of evolving Bochner spaces developed in \cite{AlCaDjEl23,AlElSt15_PM}.

In order to avoid repeating definitions and statements, we abuse the notations
\begin{align} \label{E:Abuse}
  \begin{aligned}
    &\Omega_t^0 := \Gamma_t, \quad Q_T^0 := S_T, \quad \mathbf{v}^0 := \mathbf{v}_\Gamma, \quad \partial_0^\bullet := \partial_\Gamma^\bullet \quad\text{when}\quad \varepsilon = 0, \\
    &(\nabla_\varepsilon,\mathrm{div}_\varepsilon,dx_\varepsilon) :=
    \begin{cases}
      (\nabla,\mathrm{div},dx) &\text{if}\quad \varepsilon > 0, \\
      (\nabla_\Gamma,\mathrm{div}_\Gamma,d\mathcal{H}^{n-1}) &\text{if}\quad \varepsilon = 0
    \end{cases}
  \end{aligned}
\end{align}
throughout this section.
Also, in what follows, we sometimes suppress the time variable of functions and write
\begin{align*}
  (u_1,u_2)_{L^2(\Omega_t^\varepsilon)} = \int_{\Omega_t^\varepsilon}u_1(x)u_2(x)\,dx_\varepsilon, \quad u_1,u_2\colon\Omega_t^\varepsilon\to\mathbb{R}
\end{align*}
whenever the right-hand integral makes sense.

Let $\varepsilon\geq0$, and let $\Phi_{\pm(\cdot)}^\varepsilon$ be the mappings given in Assumption \ref{A:Flow_Sur} (when $\varepsilon=0$) and Lemma \ref{L:Flow_MTD} (when $\varepsilon>0$).
For functions $U$ on $\Omega_0^\varepsilon$ and $u$ on $\Omega_t^\varepsilon$, we define
\begin{align*}
  [\phi_t^\varepsilon U](x) := U\bigl(\Phi_{-t}^\varepsilon(x)\bigr), \quad x\in\Omega_t^\varepsilon, \qquad [\phi_{-t}^\varepsilon u](X) := u\bigl(\Phi_t^\varepsilon(X)\bigr), \quad X\in\Omega_0^\varepsilon.
\end{align*}
Clearly, $\phi_{-t}^\varepsilon$ is the inverse of $\phi_t^\varepsilon$.
Let $\mathcal{X}=L^r,W^{1,r}$ with $r\in[1,\infty]$.
Then,
\begin{align*}
  \phi_t^\varepsilon\colon\mathcal{X}(\Omega_0^\varepsilon)\to\mathcal{X}(\Omega_t^\varepsilon), \quad \phi_{-t}^\varepsilon\colon\mathcal{X}(\Omega_t^\varepsilon)\to\mathcal{X}(\Omega_0^\varepsilon)
\end{align*}
are linear and bounded by Lemma \ref{L:CoSur_Lq} (when $\varepsilon=0$) and Lemma \ref{L:Flow_MTD} (when $\varepsilon>0$).
Also, we can see that the pair $(\mathcal{X}(\Omega_t^\varepsilon),\phi_t^\varepsilon)_{t\in[0,T]}$ satisfies \cite[Assumption 2.1]{AlCaDjEl23}.
We write
\begin{align*}
  \mathcal{X}_T(Q_T^\varepsilon)  :=\bigcup_{t\in[0,T]}\mathcal{X}(\Omega_t^\varepsilon)\times\{t\}, \quad u\colon[0,T]\to\mathcal{X}_T(Q_T^\varepsilon), \quad u(t) = (\hat{u}(t),t)
\end{align*}
and identify $u(t)$ with $\hat{u}(t)$.
For $q\in[1,\infty]$, we define
\begin{align*}
  L_{\mathcal{X}}^q(Q_T^\varepsilon) &:= \{u\colon[0,T]\to\mathcal{X}_T(Q_T^\varepsilon) \mid \phi_{-(\cdot)}^\varepsilon{u}(\cdot)\in L^q(0,T;\mathcal{X}(\Omega_0^\varepsilon))\}.
\end{align*}
The space $L_{\mathcal{X}}^q(Q_T^\varepsilon)$ is a Banach space equipped with norm
\begin{align} \label{E:Def_LBq}
  \|u\|_{L_{\mathcal{X}}^q(Q_T^\varepsilon)} :=
  \begin{cases}
    \left(\int_0^T\|u(t)\|_{\mathcal{X}(\Omega_t^\varepsilon)}^q\,dt\right)^{1/q} &\text{if}\quad q\neq\infty, \\
    \esup_{t\in[0,T]}\|u(t)\|_{\mathcal{X}(\Omega_t^\varepsilon)} &\text{if}\quad q=\infty.
  \end{cases}
\end{align}
Moreover, there exists a constant $c>0$ independent of $\varepsilon$ such that
\begin{align} \label{E:LBq_Equi}
  c^{-1}\|u\|_{L_{\mathcal{X}}^q(Q_T^\varepsilon)} \leq \|\phi_{-(\cdot)}^\varepsilon u(\cdot)\|_{L^q(0,T;\mathcal{X}(\Omega_0^\varepsilon))} \leq c\|u\|_{L_{\mathcal{X}}^q(Q_T^\varepsilon)}
\end{align}
by Lemma \ref{L:CoSur_Lq} (when $\varepsilon=0$) and Lemma \ref{L:Flow_MTD} (when $\varepsilon>0$).
We also define
\begin{align*}
  C_{\mathcal{X}}(Q_T^\varepsilon) &:= \{u\colon[0,T]\to\mathcal{X}_T(Q_T^\varepsilon) \mid \phi_{-(\cdot)}^\varepsilon u(\cdot)\in C([0,T];\mathcal{X}(\Omega_0^\varepsilon))\}, \\
  \mathcal{D}_{\mathcal{X}}(Q_T^\varepsilon) &:= \{u\colon[0,T]\to\mathcal{X}_T(Q_T^\varepsilon) \mid \phi_{-(\cdot)}^\varepsilon u(\cdot)\in\mathcal{D}(0,T;\mathcal{X}(\Omega_0^\varepsilon))\},
\end{align*}
with $\mathcal{D}=C_c^\infty$.
When $q\neq\infty$, we can see that $\mathcal{D}_{\mathcal{X}}(Q_T^\varepsilon)$ is dense in $L_{\mathcal{X}}^q(Q_T^\varepsilon)$ by using \eqref{E:LBq_Equi} and applying cut-off and mollification in time to $\phi_{-(\cdot)}^\varepsilon u(\cdot)$.
Also, the next result is useful in actual calculations.

\begin{lemma} \label{L:EvBo_Den}
  Let $q,r\in[1,\infty)$.
  Also, let
  \begin{align} \label{E:Def_S0T}
    C_{0,T}^\infty(\overline{Q_T^\varepsilon}) := \{u\in C^\infty(\overline{Q_T^\varepsilon}) \mid \text{$u(t)=0$ in $\overline{\Omega_t^\varepsilon}$ near $t=0,T$}\}.
  \end{align}
  For each $u\in\mathcal{D}_{W^{1,r}}(Q_T^\varepsilon)$, there exists a sequence $\{u_k\}_{k=1}^\infty$ in $C_{0,T}^\infty(\overline{Q_T^\varepsilon})$ such that
  \begin{align} \label{E:EvBo_Den}
    \lim_{k\to\infty}\|u-u_k\|_{L_{W^{1,r}}^q(Q_T^\varepsilon)} = 0, \quad \lim_{k\to\infty}\|\partial_\varepsilon^\bullet u-\partial_\varepsilon^\bullet u_k\|_{L_{W^{1,r}}^q(Q_T^\varepsilon)} = 0.
  \end{align}
  Moreover, $C_{0,T}^\infty(\overline{Q_T^\varepsilon})$ is dense in $L_{W^{1,r}}^q(Q_T^\varepsilon)$.
\end{lemma}

\begin{proof}
  Let $u\in\mathcal{D}_{W^{1,r}}(Q_T^\varepsilon)$ and $U:=\phi_{-(\cdot)}^\varepsilon u(\cdot)\in\mathcal{D}(0,T;W^{1,r}(\Omega_0^\varepsilon))$.
  Since $r\neq\infty$ and $\partial\Omega_0^\varepsilon$ (or $\Gamma_0$ when $\varepsilon=0)$ is smooth, $C^\infty(\overline{\Omega_0^\varepsilon})$ is dense in $W^{1,r}(\Omega_0^\varepsilon)$ (see e.g. \cite{AdaFou03}).
  Also, $U(t)=0$ near $t=0,T$.
  Thus, as in \cite[Lemma 2.2]{Mas84}, we can take functions of the form
  \begin{align*}
    U_k(X,t) = \sum_{\ell=1}^{L_k}\lambda_{k,\ell}(t)\psi_{k,\ell}(X), \quad (X,t) \in \overline{\Omega_0^\varepsilon}\times[0,T], \quad k\in\mathbb{N}
  \end{align*}
  with $L_k\in\mathbb{N}$, $\lambda_{k,\ell}\in C_c^\infty(0,T)$, and $\psi_{k,\ell}\in C^\infty(\overline{\Omega_0^\varepsilon})$ such that
  \begin{align} \label{Pf_EBD:Uk}
    \lim_{k\to\infty}\|U-U_k\|_{L^q(0,T;W^{1,r}(\Omega_0^\varepsilon))} = 0, \quad \lim_{k\to\infty}\|\partial_tU-\partial_tU_k\|_{L^q(0,T;W^{1,r}(\Omega_0^\varepsilon))} = 0.
  \end{align}
  Recall that $q\neq\infty$.
  For each $k\in\mathbb{N}$, let $u_k:=\phi_{(\cdot)}^\varepsilon U_k(\cdot)$, i.e.,
  \begin{align*}
    u_k(x,t) = U_k(\Phi_{-t}^\varepsilon(x),t) = \sum_{\ell=1}^{L_k}\lambda_{k,\ell}(t)\psi_{k,\ell}\bigl(\Phi_{-t}^\varepsilon(x)\bigr), \quad (x,t)\in\overline{Q_T^\varepsilon}.
  \end{align*}
  Then, $u_k\in C_{0,T}^\infty(\overline{Q_T^\varepsilon})$ by $\lambda_{k,\ell}\in C_c^\infty(0,T)$ and the smoothness of $\psi_{k,\ell}$ and $\Phi_{-(\cdot)}^\varepsilon$ (see Lemma \ref{L:MS_Reg} for $\varepsilon=0$ and Lemma \ref{L:Flow_MTD} for $\varepsilon>0$).
  Moreover, by \eqref{E:Def_MtSur} and \eqref{E:Def_MtMTD},
  \begin{align*}
    \partial_tU(X,t) = \partial_\varepsilon^\bullet u(\Phi_t^\varepsilon(X),t) = [\phi_{-t}^\varepsilon\partial_\varepsilon^\bullet u(t)](X), \quad (X,t)\in\overline{\Omega_0^\varepsilon}\times[0,T]
  \end{align*}
  and the same relation holds for $U_k$ and $u_k$.
  Hence, we have \eqref{E:EvBo_Den} by \eqref{E:LBq_Equi} and \eqref{Pf_EBD:Uk}.
  Also, the density of $C_{0,T}^\infty(\overline{Q_T^\varepsilon})$ in $L_{W^{1,r}}^q(Q_T^\varepsilon)$ follows from that of $\mathcal{D}_{W^{1,r}}(Q_T^\varepsilon)$ and \eqref{E:EvBo_Den}.
\end{proof}

Let $[\phi_t^\varepsilon]^\ast\colon[\mathcal{X}(\Omega_t^\varepsilon)]^\ast \to [\mathcal{X}(\Omega_0^\varepsilon)]^\ast$ be the dual operator of $\phi_t^\varepsilon$.
We write
\begin{align*}
  [\mathcal{X}]_T^\ast(Q_T^\varepsilon):=\bigcup_{t\in[0,T]}[\mathcal{X}(\Omega_t^\varepsilon)]^\ast\times\{t\}, \quad f\colon[0,T]\to[\mathcal{X}]_T^\ast(Q_T^\varepsilon), \quad f(t) = (\hat{f}(t),t)
\end{align*}
and again identify $f(t)$ with $\hat{f}(t)$.
For $q\in[1,\infty]$, we define
\begin{align*}
  L_{[\mathcal{X}]^\ast}^q(Q_T^\varepsilon) := \{f\colon[0,T]\to[\mathcal{X}]_T^\ast(Q_T^\varepsilon) \mid [\phi_{(\cdot)}^\varepsilon]^\ast f(\cdot)\in L^q(0,T;[\mathcal{X}(\Omega_0^\varepsilon)]^\ast)\}
\end{align*}
and the norm $\|f\|_{L_{[\mathcal{X}]^\ast}^q(Q_T^\varepsilon)}$ as in \eqref{E:Def_LBq}.
If $q\neq\infty$ and $\mathcal{X}(\Omega_t^\varepsilon)$ is reflexive, then
\begin{align*}
  [L_{\mathcal{X}}^q(Q_T^\varepsilon)]^\ast=L_{[\mathcal{X}]^\ast}^{q'}(Q_T^\varepsilon), \quad \langle f,u\rangle_{L_{\mathcal{X}}^q(Q_T^\varepsilon)} = \int_0^T\langle f(t),u(t)\rangle_{\mathcal{X}(\Omega_t^\varepsilon)}\,dt
\end{align*}
with $1/q+1/q'=1$.
Thus, $L_{\mathcal{X}}^q(Q_T^\varepsilon)$ is reflexive if $q\neq1,\infty$.

When $u\in C^1(\overline{Q_T^\varepsilon})$, we have (recall that we use the notations \eqref{E:Abuse})
\begin{align*}
  \frac{d}{dt}\int_{\Omega_t^\varepsilon}u\,dx_\varepsilon = \int_{\Omega_t^\varepsilon}(\partial_\varepsilon^\bullet u+u\,\mathrm{div}_\varepsilon\mathbf{v}^\varepsilon)\,dx_\varepsilon, \quad t\in(0,T)
\end{align*}
by the Reynolds transport theorem (or the Leibniz formula, see e.g. \cite{DziEll07,DziEll13_AN} for $\varepsilon=0$ and \cite{Gur81} for $\varepsilon>0$).
Based on this formula, we define a weak material derivative as follows.

\begin{definition} \label{D:We_MatDe}
  Let $p\in(2,\infty)$.
  We say that $u\in L_{W^{1,p}}^1(Q_T^\varepsilon)$ has the weak material derivative if there exists a functional $z\in L_{[W^{1,p}]^\ast}^1(Q_T^\varepsilon)$ such that
  \begin{align} \label{E:Def_WeMat}
     \int_0^T\langle z,\psi\rangle_{W^{1,p}(\Omega_t^\varepsilon)}\,dt = -\int_0^T(u,\partial_\varepsilon^\bullet\psi+\psi\,\mathrm{div}_\varepsilon\mathbf{v}^\varepsilon)_{L^2(\Omega_t^\varepsilon)}\,dt
  \end{align}
  for all $\psi\in\mathcal{D}_{W^{1,p}}(Q_T^\varepsilon)$.
  In this case, we write $z=\partial_\varepsilon^\bullet u$.
\end{definition}

\begin{remark} \label{R:We_MatDe}
  By Lemma \ref{L:EvBo_Den}, we may replace $\mathcal{D}_{W^{1,p}}(Q_T^\varepsilon)$ by $C_{0,T}^\infty(\overline{Q_T^\varepsilon})$ in Definition \ref{D:We_MatDe}.
\end{remark}

\begin{definition} \label{D:TD_SoBo}
  Let $p\in(2,\infty)$ and $p'\in(1,2)$ satisfy $1/p+1/p'=1$.
  We define
  \begin{align*}
    \mathbb{W}^{p,p'}(Q_T^\varepsilon) &:= \{u\in L_{W^{1,p}}^p(Q_T^\varepsilon) \mid \partial_\varepsilon^\bullet u\in L_{[W^{1,p}]^\ast}^{p'}(Q_T^\varepsilon)\}, \\
    \|u\|_{\mathbb{W}^{p,p'}(Q_T^\varepsilon)} &:= \|u\|_{L_{W^{1,p}}^p(Q_T^\varepsilon)}+\|\partial_\varepsilon^\bullet u\|_{L_{[W^{1,p}]^\ast}^{p'}(Q_T^\varepsilon)}.
  \end{align*}
\end{definition}

The space $\mathbb{W}^{p,p'}(Q_T^\varepsilon)$ is a Banach space.
In \cite[Definition 3.17]{AlCaDjEl23}, it is written as
\begin{align*}
  \mathbb{W}^{p,p'}(Q_T^\varepsilon) = \mathbb{W}^{p,p'}(W^{1,p}(\Omega_{(\cdot)}^\varepsilon),[W^{1,p}(\Omega_{(\cdot)}^\varepsilon)]^\ast).
\end{align*}
As in \cite[Propisition 6.5]{AlCaDjEl23}, we can observe that the spaces
\begin{align*}
  \mathbb{W}^{p,p'}(Q_T^\varepsilon) \quad\text{and}\quad \{U\in L^p(0,T;W^{1,p}(\Omega_0^\varepsilon)) \mid \partial_tU\in L^{p'}(0,T;[W^{1,p}(\Omega_0^\varepsilon)]^\ast)\}
\end{align*}
are equivalent in the sense of \cite[Definition 3.19]{AlCaDjEl23} by Assumption \ref{A:Flow_Sur} and Lemmas \ref{L:MS_Reg} and \ref{L:CoSur_Lq} (when $\varepsilon=0$) and by Lemma \ref{L:Flow_MTD} (when $\varepsilon>0$).
Moreover, the Gelfand triple structure
\begin{align*}
  W^{1,p}(\Omega_t^\varepsilon) \hookrightarrow L^2(\Omega_t^\varepsilon) \hookrightarrow [W^{1,p}(\Omega_t^\varepsilon)]^\ast, \quad t\in[0,T]
\end{align*}
holds and the first embedding is compact by $p>2$.
Thus, we can use the abstract results \cite[Theorems 4.5 and 5.2]{AlCaDjEl23} to get the following transport and Aubin--Lions lemmas.

\begin{lemma} \label{L:Trans}
  The embedding $\mathbb{W}^{p,p'}(Q_T^\varepsilon)\hookrightarrow C_{L^2}(Q_T^\varepsilon)$ is continuous, i.e.,
  \begin{align*}
    \|u\|_{C_{L^2}(Q_T^\varepsilon)} := \sup_{t\in[0,T]}\|u(t)\|_{L^2(\Omega_t^\varepsilon)} \leq c_\varepsilon\|u\|_{\mathbb{W}^{p,p'}(Q_T^\varepsilon)}, \quad u\in\mathbb{W}^{p,p'}(Q_T^\varepsilon)
  \end{align*}
  with some constant $c_\varepsilon>0$ depending on $\varepsilon$.
  Moreover,
  \begin{align} \label{E:Trans}
    \frac{d}{dt}(u_1,u_2)_{L^2(\Omega_t^\varepsilon)} = \langle\partial_\varepsilon^\bullet u_1,u_2\rangle_{W^{1,p}(\Omega_t^\varepsilon)}+\langle\partial_\varepsilon^\bullet u_2,u_1\rangle_{W^{1,p}(\Omega_t^\varepsilon)}+(u_1,u_2\,\mathrm{div}_\varepsilon\mathbf{v}^\varepsilon\bigr)_{L^2(\Omega_t^\varepsilon)}
  \end{align}
  for all $u_1,u_2\in\mathbb{W}^{p,p'}(Q_T^\varepsilon)$ and almost all $t\in(0,T)$.
\end{lemma}

\begin{lemma} \label{L:AuLi}
  The embedding $\mathbb{W}^{p,p'}(Q_T^\varepsilon)\hookrightarrow L_{L^2}^2(Q_T^\varepsilon)$ is compact.
\end{lemma}

We will also use the following results.

\begin{lemma} \label{L:Wpp_Mult}
  Let $\chi\in C^1(\overline{Q_T^\varepsilon})$ and $u\in\mathbb{W}^{p,p'}(Q_T^\varepsilon)$.
  Then, $\chi u\in\mathbb{W}^{p,p'}(Q_T^\varepsilon)$ and
  \begin{align} \label{E:WeMa_Mult}
    \langle\partial_\varepsilon^\bullet(\chi u),\psi\rangle_{W^{1,p}(\Omega_t^\varepsilon)} = \langle\partial_\varepsilon^\bullet u,\chi\psi\rangle_{W^{1,p}(\Omega_t^\varepsilon)}+(u,\psi\partial_\varepsilon^\bullet\chi)_{L^2(\Omega_t^\varepsilon)}
  \end{align}
  for all $\psi\in L_{W^{1,p}}^p(Q_T^\varepsilon)$ and almost all $t\in(0,T)$.
  Moreover, for a.a. $t\in(0,T)$,
  \begin{align} \label{E:Tr_Mult}
    \frac{d}{dt}(u,\chi u\bigr)_{L^2(\Omega_t^\varepsilon)} = 2\langle\partial_\varepsilon^\bullet u,\chi u\rangle_{W^{1,p}(\Omega_t^\varepsilon)}+(u,u\partial_\varepsilon^\bullet\chi\bigr)_{L^2(\Omega_t^\varepsilon)}+(u,\chi u\,\mathrm{div}_\varepsilon\mathbf{v}^\varepsilon)_{L^2(\Omega_t^\varepsilon)}.
  \end{align}
\end{lemma}

\begin{proof}
  We refer to \cite[Lemma 2.6]{Miu25pre_pLap} for the proof.
\end{proof}

\begin{lemma} \label{L:Wpp_Comp}
  Let $\Lambda\in C^2(\mathbb{R})$ such that $\Lambda''$ is bounded on $\mathbb{R}$.
  Then,
  \begin{align*}
    \Lambda'(u) = \Lambda'\circ u \in L_{W^{1,p}}^p(Q_T^\varepsilon), \quad \Lambda(u) = \Lambda\circ u \in C_{L^1}(Q_T^\varepsilon)
  \end{align*}
  for all $u\in \mathbb{W}^{p,p'}(Q_T^\varepsilon)$.
  Moreover, for almost all $t\in(0,T)$, we have
  \begin{align} \label{E:Wpp_Comp}
    \frac{d}{dt}\int_{\Omega_t^\varepsilon}[\Lambda(u)](t)\,dx_\varepsilon = \langle\partial_\varepsilon^\bullet u(t),[\Lambda'(u)](t)\rangle_{W^{1,p}(\Omega_t^\varepsilon)}+\int_{\Omega_t^\varepsilon}[\Lambda(u)\,\mathrm{div}_\varepsilon\mathbf{v}^\varepsilon](t)\,dx_\varepsilon.
  \end{align}
\end{lemma}

\begin{proof}
  We refer to \cite[Lemma 2.7]{Miu25pre_pLap} for the proof.
\end{proof}

%%% Section 5 %%%
\section{Weak solution to the thin-domain problem} \label{S:WSMTD}
This section gives the definition and some estimates of a weak solution to \eqref{E:pLap_MTD}.
In what follows, we fix $p\in(2,\infty)$, and let $p'\in(1,2)$ satisfy $1/p+1/p'=1$.
Note that
\begin{align*}
  \bigl\|\,|\mathbf{a}|^{p-2}\mathbf{a}\,\|_{L^{p'}(\Omega_t^\varepsilon)} = \bigl\|\,|\mathbf{a}|^{p-1}\,\|_{L^{p'}(\Omega_t^\varepsilon)} = \|\mathbf{a}\|_{L^p(\Omega_t^\varepsilon)}^{p-1}, \quad \mathbf{a} \in [L^p(\Omega_t^\varepsilon)]^n
\end{align*}
by $p'=p/(p-1)$, and the same formula holds for integrals over $\Gamma_t$, $Q_T^\varepsilon$, and $S_T$.
We use this formula frequently below without mention.

Suppose that $u^\varepsilon$ and $f^\varepsilon$ are smooth and satisfy \eqref{E:pLap_MTD}.
Using
\begin{align*}
  \partial_tu^\varepsilon = \partial_\varepsilon^\bullet u^\varepsilon-\mathbf{v}^\varepsilon\cdot\nabla u^\varepsilon = \partial_\varepsilon^\bullet u^\varepsilon-\mathrm{div}(u^\varepsilon\mathbf{v}^\varepsilon)+u^\varepsilon\,\mathrm{div}\,\mathbf{v}^\varepsilon \quad\text{in}\quad Q_T^\varepsilon
\end{align*}
and $\partial_{\nu^\varepsilon}=\bm{\nu}^\varepsilon\cdot\nabla$ and $V^\varepsilon=\mathbf{v}^\varepsilon\cdot\bm{\nu}^\varepsilon$ on $\partial_\ell Q_T^\varepsilon$, we rewrite \eqref{E:pLap_MTD} as
\begin{align*}
  \left\{
  \begin{aligned}
    &\partial_\varepsilon^\bullet u^\varepsilon-\mathrm{div}(|\nabla u^\varepsilon|^{p-2}\nabla u^\varepsilon+u^\varepsilon\mathbf{v}^\varepsilon)+u^\varepsilon\,\mathrm{div}\,\mathbf{v}^\varepsilon = f^\varepsilon \quad \text{in} \quad Q_T^\varepsilon, \\
    &(|\nabla u^\varepsilon|^{p-2}\nabla u^\varepsilon+u^\varepsilon\mathbf{v}^\varepsilon)\cdot\bm{\nu}^\varepsilon = 0 \quad \text{on} \quad \partial_\ell Q_T^\varepsilon.
  \end{aligned}
  \right.
\end{align*}
We multiply the equation by a test function $\psi$, integrate over $Q_T^\varepsilon$, and carry out integration by parts and use the boundary condition to the divergence term.
Then, we get
\begin{multline*}
  \int_0^T\left(\int_{\Omega_t^\varepsilon}\{(\partial_\varepsilon^\bullet u^\varepsilon)\psi+(|\nabla u^\varepsilon|^{p-2}\nabla u^\varepsilon+u^\varepsilon\mathbf{v}^\varepsilon)\cdot\nabla\psi+u^\varepsilon\psi\,\mathrm{div}\,\mathbf{v}^\varepsilon\}\,dx\right)\,dt \\
  = \int_0^T\left(\int_{\Omega_t^\varepsilon}f^\varepsilon\psi\,dx\right)\,dt.
\end{multline*}
Based on this formula, we define a weak solution to \eqref{E:pLap_MTD} as follows.

\begin{definition} \label{D:WS_MTD}
  For given $u_0^\varepsilon\in L^2(\Omega_0^\varepsilon)$ and $f^\varepsilon\in L_{[W^{1,p}]^\ast}^{p'}(Q_T^\varepsilon)$, we say that $u^\varepsilon$ is a weak solution to \eqref{E:pLap_MTD} if $u^\varepsilon\in\mathbb{W}^{p,p'}(Q_T^\varepsilon)$ and it satisfies
  \begin{align} \label{E:WeFo_MTD}
    \begin{aligned}
      &\int_0^T\langle\partial_\varepsilon^\bullet u^\varepsilon,\psi\rangle_{W^{1,p}(\Omega_t^\varepsilon)}\,dt+\int_0^T(|\nabla u^\varepsilon|^{p-2}\nabla u^\varepsilon,\nabla\psi)_{L^2(\Omega_t^\varepsilon)}\,dt \\
      &\qquad +\int_0^T(u^\varepsilon,\mathbf{v}^\varepsilon\cdot\nabla\psi+\psi\,\mathrm{div}\,\mathbf{v}^\varepsilon)_{L^2(\Omega_t^\varepsilon)}\,dt = \int_0^T\langle f^\varepsilon,\psi\rangle_{W^{1,p}(\Omega_t^\varepsilon)}\,dt
    \end{aligned}
  \end{align}
  for all $\psi\in L_{W^{1,p}}^p(Q_T^\varepsilon)$ and the initial condition $u^\varepsilon(0)=u_0^\varepsilon$ in $L^2(\Omega_0^\varepsilon)$.
\end{definition}

Note that $L_{W^{1,p}}^p(Q_T^\varepsilon)\subset L_{H^1}^2(Q_T^\varepsilon)$ by $p>2$ and $T<\infty$.
Moreover,
\begin{align*}
  \left|\int_0^T(|\nabla u^\varepsilon|^{p-2}\nabla u^\varepsilon,\nabla\psi)_{L^2(\Omega_t^\varepsilon)}\,dt\right| \leq \|\nabla u^\varepsilon\|_{L_{L^p}^p(Q_T^\varepsilon)}^{p-1}\|\nabla\psi\|_{L_{L^p}^p(Q_T^\varepsilon)}
\end{align*}
by H\"{o}lder's inequality.
Thus, the weak form \eqref{E:WeFo_MTD} makes sense.
Also, the initial condition can be considered in $L^2(\Omega_0^\varepsilon)$ by Lemma \ref{L:Trans}.

The existence and uniqueness of a weak solution was shown in \cite{Miu25pre_pLap} for a general moving domain.
We do not repeat the proof and just give the statement below.

\begin{proposition} \label{P:ExUn_MTD}
  For all $u_0^\varepsilon\in L^2(\Omega_0^\varepsilon)$ and $f^\varepsilon\in L_{[W^{1,p}]^\ast}^{p'}(Q_T^\varepsilon)$, there exists a unique weak solution $u^\varepsilon\in\mathbb{W}^{p,p'}(Q_T^\varepsilon)$ to \eqref{E:pLap_MTD}.
\end{proposition}

Next, we estimate $u^\varepsilon$ explicitly in terms of $\varepsilon$.
We write $c$ and $c_T$ for general positive constants independent of $\varepsilon$ and $t$, but $c_T$ depends on $T$ in general.

\begin{proposition} \label{P:Mass_MTD}
  Suppose that $u_0^\varepsilon\in L^2(\Omega_0^\varepsilon)$ and $f^\varepsilon\in L_{[W^{1,p}]^\ast}^{p'}(Q_T^\varepsilon)$ satisfy \eqref{E:Data_MTD}, and let $u^\varepsilon$ be the unique weak solution to \eqref{E:pLap_MTD} given by Proposition \ref{P:ExUn_MTD}.
  Then,
  \begin{align}
    |(u^\varepsilon(t),1)_{L^2(\Omega_t^\varepsilon)}| &\leq c_T\varepsilon, \label{E:Mass_MTD} \\
    \|u^\varepsilon(t)\|_{W^{1,p}(\Omega_t^\varepsilon)} &\leq c_T\Bigl(\|\nabla u^\varepsilon(t)\|_{L^p(\Omega_t^\varepsilon)}+\varepsilon^{1/p}\Bigr) \label{E:W1p_MTD}
  \end{align}
  for all $t\in[0,T]$.
\end{proposition}

\begin{proof}
  Let $\psi\equiv1$ in \eqref{E:WeFo_MTD} with $T$ replaced by $t\in[0,T]$.
  Then, we have
  \begin{align*}
    (u^\varepsilon(t),1)_{L^2(\Omega_t^\varepsilon)} = (u_0^\varepsilon,1)_{L^2(\Omega_0^\varepsilon)}+\int_0^t\langle f^\varepsilon,1\rangle_{W^{1,p}(\Omega_s^\varepsilon)}\,ds
  \end{align*}
  by \eqref{E:Trans} and $\partial_\varepsilon^\bullet\psi=0$.
  Moreover, by H\"{o}lder's inequality, \eqref{E:Vol_MTD}, and \eqref{E:Data_MTD},
  \begin{align*}
    |(u_0^\varepsilon,1)_{L^2(\Omega_0^\varepsilon)}| \leq \|u_0^\varepsilon\|_{L^2(\Omega_0^\varepsilon)}|\Omega_0^\varepsilon|^{1/2} \leq c\varepsilon.
  \end{align*}
  Since $\|1\|_{W^{1,p}(\Omega_s^\varepsilon)}=|\Omega_s^\varepsilon|^{1/p} \leq c\varepsilon^{1/p}$ by \eqref{E:Vol_MTD}, we also have
  \begin{align*}
    \left|\int_0^t\langle f^\varepsilon,1\rangle_{W^{1,p}(\Omega_s^\varepsilon)}\,ds\right| &\leq \int_0^t\|f^\varepsilon\|_{[W^{1,p}(\Omega_s^\varepsilon)]^\ast}\|1\|_{W^{1,p}(\Omega_s^\varepsilon)}\,ds \\
    &\leq c\varepsilon^{1/p}T^{1/p}\|f^\varepsilon\|_{L_{[W^{1,p}]^\ast}^{p'}(Q_T^\varepsilon)} \leq c\varepsilon T^{1/p}
  \end{align*}
  by H\"{o}lder's inequality, \eqref{E:Data_MTD}, and $1/p+1/p'=1$.
  Thus, \eqref{E:Mass_MTD} follows.
  Also, we have \eqref{E:W1p_MTD} by \eqref{E:Cor_UP} with $q=p$ and \eqref{E:Mass_MTD}.
\end{proof}

\begin{proposition} \label{P:Ener_MTD}
  Under the assumptions of Proposition \ref{P:Mass_MTD}, we have
  \begin{align} \label{E:Ener_MTD}
    \|u^\varepsilon(t)\|_{L^2(\Omega_t^\varepsilon)}^2+\int_0^t\|u^\varepsilon\|_{W^{1,p}(\Omega_s^\varepsilon)}^p\,ds \leq c_T\varepsilon \quad\text{for all}\quad t\in[0,T].
  \end{align}
\end{proposition}

\begin{proof}
  Let $\psi=u^\varepsilon$ in \eqref{E:WeFo_MTD} with $T$ replaced by $t\in[0,T]$.
  Then, by \eqref{E:Trans},
  \begin{align*}
    &\frac{1}{2}\|u^\varepsilon(t)\|_{L^2(\Omega_t^\varepsilon)}^2+\int_0^t\|\nabla u^\varepsilon\|_{L^p(\Omega_s^\varepsilon)}^p\,ds+\int_0^t(u^\varepsilon,\mathbf{v}^\varepsilon\cdot\nabla u^\varepsilon)_{L^2(\Omega_s^\varepsilon)}\,ds \\
    &\qquad +\frac{1}{2}\int_0^t(u^\varepsilon,u^\varepsilon\,\mathrm{div}\,\mathbf{v}^\varepsilon)_{L^2(\Omega_s^\varepsilon)}\,ds = \frac{1}{2}\|u_0^\varepsilon\|_{L^2(\Omega_0^\varepsilon)}^2+\int_0^t\langle f^\varepsilon,u^\varepsilon\rangle_{W^{1,p}(\Omega_s^\varepsilon)}\,ds.
  \end{align*}
  We further apply \eqref{E:BdV_MTD} and H\"{o}lder's and Young's inequalities to get
  \begin{align*}
    &\frac{1}{2}\|u^\varepsilon(t)\|_{L^2(\Omega_t^\varepsilon)}^2+\int_0^t\|\nabla u^\varepsilon\|_{L^p(\Omega_s^\varepsilon)}^p\,ds \\
    &\qquad \leq \frac{1}{2}\|u_0^\varepsilon\|_{L^2(\Omega_0^\varepsilon)}^2+\gamma\int_0^t\Bigl(\|u^\varepsilon\|_{W^{1,p}(\Omega_s^\varepsilon)}^p+\|\nabla u^\varepsilon\|_{L^2(\Omega_s^\varepsilon)}^2\Bigr)ds \\
    &\qquad\qquad +c_\gamma\int_0^t\Bigl(\|u^\varepsilon\|_{L^2(\Omega_s^\varepsilon)}^2+\|f^\varepsilon\|_{[W^{1,p}(\Omega_s^\varepsilon)]^\ast}^{p'}\Bigr)\,ds
  \end{align*}
  with any $\gamma>0$, where $c_\gamma>0$ is a constant depending only on $\gamma$.
  Moreover, since $p>2$, it follows from \eqref{E:L2Lp_MTD}, \eqref{E:W1p_MTD}, and Young's inequality that
  \begin{align*}
    \|u^\varepsilon\|_{W^{1,p}(\Omega_s^\varepsilon)}^p+\|\nabla u^\varepsilon\|_{L^2(\Omega_s^\varepsilon)}^2 &\leq c_T\Bigl(\|\nabla u^\varepsilon\|_{L^p(\Omega_s^\varepsilon)}^p+\varepsilon\Bigr)+c\varepsilon^{1-2/p}\|\nabla u^\varepsilon\|_{L^p(\Omega_t^\varepsilon)}^2 \\
    &\leq c_T\Bigl(\|\nabla u^\varepsilon\|_{L^p(\Omega_s^\varepsilon)}^p+\varepsilon\Bigr).
  \end{align*}
  Thus, taking $\gamma:=1/2c_T$, we deduce from the above inequalities that
  \begin{align*}
    &\frac{1}{2}\|u^\varepsilon(t)\|_{L^2(\Omega_t^\varepsilon)}^2+\frac{1}{2}\int_0^t\|\nabla u^\varepsilon\|_{L^p(\Omega_s^\varepsilon)}^p\,ds \\
    &\qquad \leq \frac{1}{2}\|u_0^\varepsilon\|_{L^2(\Omega_0^\varepsilon)}^2+c_T\int_0^t\Bigl(\|u^\varepsilon\|_{L^2(\Omega_s^\varepsilon)}^2+\|f^\varepsilon\|_{[W^{1,p}(\Omega_s^\varepsilon)]^\ast}^{p'}+\varepsilon\Bigr)\,ds,
  \end{align*}
  and we multiply both sides by two and apply Gronwall's inequality to find that
  \begin{align*}
    \|u^\varepsilon(t)\|_{L^2(\Omega_t^\varepsilon)}^2+\int_0^t\|\nabla u^\varepsilon\|_{L^p(\Omega_s^\varepsilon)}^p\,ds &\leq c_T\left\{\|u_0^\varepsilon\|_{L^2(\Omega_0^\varepsilon)}^2+\int_0^t\Bigl(\|f^\varepsilon\|_{[W^{1,p}(\Omega_s^\varepsilon)]^\ast}^{p'}+\varepsilon\Bigr)\,ds\right\}.
  \end{align*}
  By this inequality, \eqref{E:Data_MTD}, and \eqref{E:W1p_MTD}, we obtain \eqref{E:Ener_MTD}.
\end{proof}

\begin{proposition} \label{P:L2H1_MTD}
  Under the assumptions of Proposition \ref{P:Mass_MTD}, we have
  \begin{align} \label{E:L2H1_MTD}
    \int_0^T\|u^\varepsilon\|_{H^1(\Omega_t^\varepsilon)}^2\,dt \leq c_T\varepsilon.
  \end{align}
\end{proposition}

\begin{proof}
  We use \eqref{E:L2Lp_MTD} (to $u^\varepsilon$ and $\nabla u^\varepsilon$) and H\"{o}lder's inequality to get
  \begin{align*}
    \int_0^T\|u^\varepsilon\|_{H^1(\Omega_t^\varepsilon)}^2\,dt &\leq c\varepsilon^{1-2/p}\int_0^T\|u^\varepsilon\|_{W^{1,p}(\Omega_t^\varepsilon)}^2\,dt \\
    &\leq c\varepsilon^{1-2/p}T^{1-2/p}\left(\int_0^T\|u^\varepsilon\|_{W^{1,p}(\Omega_t^\varepsilon)}^p\,dt\right)^{2/p}.
  \end{align*}
  Applying \eqref{E:Ener_MTD} with $t=T$ to the second line, we obtain \eqref{E:L2H1_MTD}.
\end{proof}

\begin{proposition} \label{P:dudt_MTD}
  Under the assumptions of Proposition \ref{P:Mass_MTD}, we have
  \begin{align} \label{E:dudt_MTD}
    \|\partial_\varepsilon^\bullet u^\varepsilon\|_{L_{[W^{1,p}]^\ast}^{p'}(Q_T^\varepsilon)} \leq c_T\varepsilon^{1/p'}.
  \end{align}
\end{proposition}

\begin{proof}
  Let $\psi\in L_{W^{1,p}}^p(Q_T^\varepsilon)$.
  We apply \eqref{E:BdV_MTD} and H\"{o}lder's inequality to \eqref{E:WeFo_MTD} to get
  \begin{align*}
    \left|\int_0^T\langle\partial_\varepsilon^\bullet u^\varepsilon,\psi\rangle_{W^{1,p}(\Omega_t^\varepsilon)}\,dt\right| &\leq \|\nabla u^\varepsilon\|_{L_{L^p}^p(Q_T^\varepsilon)}^{p-1}\|\nabla\psi\|_{L_{L^p}^p(Q_T^\varepsilon)} \\
    &\qquad +c\|u^\varepsilon\|_{L_{L^2}^2(Q_T^\varepsilon)}\|\psi\|_{L_{H^1}^2(Q_T^\varepsilon)}\\
    &\qquad +\|f^\varepsilon\|_{L_{[W^{1,p}]^\ast}^{p'}(Q_T^\varepsilon)}\|\psi\|_{L_{W^{1,p}}^p(Q_T^\varepsilon)}.
  \end{align*}
  To the right-hand side, we further use \eqref{E:Data_MTD}, \eqref{E:Ener_MTD}, and
  \begin{align*}
    \|\psi\|_{L_{H^1}^2(Q_T^\varepsilon)} \leq c\varepsilon^{1/2-1/p}\|\psi\|_{L_{W^{1,p}}^2(Q_T^\varepsilon)} \leq c\varepsilon^{1/2-1/p}T^{1/2-1/p}\|\psi\|_{L_{W^{1,p}}^p(Q_T^\varepsilon)}
  \end{align*}
  by \eqref{E:L2Lp_MTD} and H\"{o}lder's inequality in time.
  Then, since $1-1/p=1/p'$, we get
  \begin{align*}
    \left|\int_0^T\langle\partial_\varepsilon^\bullet u^\varepsilon,\psi\rangle_{W^{1,p}(\Omega_t^\varepsilon)}\,dt\right| \leq c_T\varepsilon^{1/p'}\|\psi\|_{L_{W^{1,p}}^p(Q_T^\varepsilon)}.
  \end{align*}
  Hence, the inequality \eqref{E:dudt_MTD} follows.
\end{proof}

%%% Section 6 %%%
\section{Weighted average in the thin direction} \label{S:WeAve}
In this section, we study the weighted average of a function on $\Omega_t^\varepsilon$.
Also, we define the constant extension of a functional on $\Gamma_t$ by using the weighted average.

\subsection{Definition and basic properties} \label{SS:WA_Def}
Let $t\in[0,T]$.
For a function $\varphi$ on $\Omega_t^\varepsilon$, we define the weighted average of $\varphi$ in the thin direction (or the normal direction of $\Gamma_t$) by
\begin{align*}
  \mathcal{M}_\varepsilon\varphi(y) := \frac{1}{\varepsilon g(y,t)}\int_{\varepsilon g_0(y,t)}^{\varepsilon g_1(y,t)}\varphi\bigl(y+r\bm{\nu}(y,t)\bigr)J(y,t,r)\,dr, \quad y\in\Gamma_t,
\end{align*}
where $J$ is the Jacobian given by \eqref{E:Def_J}.
This definition and \eqref{E:CoV_MTD} imply
\begin{align} \label{E:Ave_Pair}
  \int_{\Omega_t^\varepsilon}\varphi(x)\bar{\eta}(x)\,dx = \varepsilon\int_{\Gamma_t}g(y,t)\mathcal{M}_\varepsilon\varphi(y)\eta(y)\,d\mathcal{H}^{n-1}(y)
\end{align}
for a function $\eta$ on $\Gamma_t$ and its constant extension $\bar{\eta}$ in the normal direction of $\Gamma_t$.

Let us give properties of the weighted average.
In what follows, we use the notation \eqref{E:Pull_MTD} and suppress the variables $y$ and $t$.
For example, we write
\begin{align*}
  \mathcal{M}_\varepsilon\varphi = \frac{1}{\varepsilon g}\int_{\varepsilon g_0}^{\varepsilon g_1}\varphi^\sharp(r)J(r)\,dr.
\end{align*}
Using \eqref{E:G_Bdd}, \eqref{E:J_Poly}, and \eqref{E:J_Bdd}, we can get the following two results as in the case \cite{Miu25_GL,Miu25pre_CH} of fixed surfaces and thin domains.
We omit details here.

\begin{lemma} \label{L:Ave_Lq}
  Let $q\in[1,\infty)$ and $\varphi\in L^q(\Omega_t^\varepsilon)$.
  Then,
  \begin{align} \label{E:Ave_Lq}
    \|\mathcal{M}_\varepsilon\varphi\|_{L^q(\Gamma_t)} \leq c\varepsilon^{-1/q}\|\varphi\|_{L^q(\Omega_t^\varepsilon)}.
  \end{align}
\end{lemma}

\begin{proof}
  We refer to \cite[Lemma 5.1]{Miu25_GL} and \cite[Lemma 4.1]{Miu25pre_CH} for the proof.
\end{proof}

\begin{lemma} \label{L:AvDf_Lq}
  Let $q\in[1,\infty)$ and $\varphi\in W^{1,q}(\Omega_t^\varepsilon)$.
  Then,
  \begin{align} \label{E:AvDf_Lq}
    \Bigl\|\varphi-\overline{\mathcal{M}_\varepsilon\varphi}\Bigr\|_{L^q(\Omega_t^\varepsilon)} \leq c\varepsilon\|\varphi\|_{W^{1,q}(\Omega_t^\varepsilon)}.
  \end{align}
\end{lemma}

\begin{proof}
  We refer to \cite[Lemma 5.2]{Miu25_GL} and \cite[Lemma 4.2]{Miu25pre_CH} for the proof.
\end{proof}

Let $\eta$ be a function on $\Gamma_t$ and $\bar{\eta}$ be its constant extension.
We write
\begin{align*}
  d^0\bar{\eta} := \bar{\eta}, \quad d^1\bar{\eta} = d\bar{\eta} := d(\cdot,t)\bar{\eta}, \quad d^\ell\bar{\eta} := \{d(\cdot,t)\}^\ell\bar{\eta} \quad\text{on}\quad \overline{\mathcal{N}_t}, \quad \ell\in\mathbb{Z}_{\geq0},
\end{align*}
where $\overline{\mathcal{N}_t}$ is the tubular neighborhood of $\Gamma_t$ (see Lemma \ref{L:MS_Tubu}).
Let us compute the weighted average of $d^\ell\bar{\eta}$.
We use the functions $J_0\equiv1$ and $J_1,\dots,J_n$ on $\overline{S_T}$ given in \eqref{E:J_Poly}.

\begin{lemma} \label{L:Ave_CE}
  Let $\eta$ be a function on $\Gamma_t$.
  Then,
  \begin{align} \label{E:Ave_CE}
    \mathcal{M}_\varepsilon(d^\ell\bar{\eta}) = \eta\sum_{k=0}^n\frac{\varepsilon^{k+\ell}J_k}{k+\ell+1}\sum_{m=0}^{k+\ell}g_1^{k+\ell-m}g_0^m \quad\text{on}\quad \Gamma_t, \quad \ell\in\mathbb{Z}_{\geq0}.
  \end{align}
  In particular, if $\eta\in W^{1,q}(\Gamma_t)$ with $q\in[1,\infty)$, then
  \begin{align} \label{E:AvCE_W1q}
    \|\mathcal{M}_\varepsilon\bar{\eta}-\eta\|_{W^{1,q}(\Gamma_t)} \leq c\varepsilon\|\eta\|_{W^{1,q}(\Gamma_t)}, \quad \|\mathcal{M}_\varepsilon(d\bar{\eta})\|_{W^{1,q}(\Gamma_t)} \leq c\varepsilon\|\eta\|_{W^{1,q}(\Gamma_t)}.
  \end{align}
\end{lemma}

\begin{proof}
  Since $\bar{\eta}^\sharp(r)=\eta$ is independent of $r$ and $d^\sharp(r)=r$, we see by \eqref{E:J_Poly} that
  \begin{align*}
    \mathcal{M}_\varepsilon(d^\ell\bar{\eta}) &= \frac{\eta}{\varepsilon g}\int_{\varepsilon g_0}^{\varepsilon g_1}r^\ell J(r)\,dr = \frac{\eta}{\varepsilon g}\sum_{k=0}^n\frac{J_k}{k+\ell+1}\{(\varepsilon g_1)^{k+\ell+1}-(\varepsilon g_0)^{k+\ell+1}\},
  \end{align*}
  and the right-hand side is equal to that of \eqref{E:Ave_CE}.
  Also, since $J_0\equiv1$, and since $g_0$, $g_1$, and $J_k$ are smooth on $\overline{S_T}$, we have \eqref{E:AvCE_W1q} by \eqref{E:Ave_CE}.
\end{proof}

Next, we give some results on the tangential gradient of the weighted average.

\begin{lemma} \label{L:Ave_TDr}
  Let $\varphi\in C(\overline{\Omega_t^\varepsilon})\cap C^1(\Omega_t^\varepsilon)$.
  Then,
  \begin{align} \label{E:Ave_TDr}
    \nabla_\Gamma\mathcal{M}_\varepsilon\varphi = \mathcal{M}_\varepsilon(\mathbf{B}\nabla\varphi)+\mathcal{M}_\varepsilon[(\partial_\nu\varphi+\varphi f_J)\mathbf{b}_\varepsilon]+\mathcal{M}_\varepsilon(\varphi\mathbf{b}_J) \quad\text{on}\quad \Gamma_t.
  \end{align}
  Here, $\partial_\nu\varphi:=\bar{\bm{\nu}}\cdot\nabla\varphi$ is the derivative of $\varphi$ in the direction of $\bar{\bm{\nu}}$.
  Also,
  \begin{align*}
    \mathbf{B}(x,t) \in \mathbb{R}^{n\times n}, \quad \mathbf{b}_\varepsilon(x,t)\in\mathbb{R}^n, \quad \mathbf{b}_J(x,t) \in \mathbb{R}^n, \quad f_J(x,t) \in \mathbb{R}
  \end{align*}
  for $(x,t)\in\bigcup_{s\in[0,T]}\overline{\mathcal{N}_s}\times\{s\}$ are given by
  \begin{align*}
    \mathbf{B}^\sharp(y,t,r) &:= \mathbf{P}(y,t)-r\mathbf{W}(y,t), \\
    \mathbf{b}_\varepsilon^\sharp(y,t,r) &:= \frac{1}{g(y,t)}\Bigl\{\bigl(r-\varepsilon g_0(y,t)\bigr)\nabla_\Gamma g_y(y,t)+\bigl(\varepsilon g_1(y,t)-r\bigr)\nabla_\Gamma g_0(y,t)\Bigr\}, \\
    \mathbf{b}_J^\sharp(y,t,r) &:= \frac{\nabla_\Gamma J(y,t,r)}{J(y,t,r)}, \quad f_J^\sharp(y,t,r) := \frac{\partial_rJ(y,t,r)}{J(y,t,r)}
  \end{align*}
  under the notation \eqref{E:Pull_MTD} with $(y,t)\in\overline{S_T}$ and $r\in[-\delta,\delta]$.
\end{lemma}

\begin{proof}
  We refer to \cite[Lemma 5.6]{Miu25_GL} for the proof.
\end{proof}

\begin{lemma} \label{L:ATDr_Diff}
  Let $\varphi\in C(\overline{\Omega_t^\varepsilon})\cap C^1(\Omega_t^\varepsilon)$.
  Then,
  \begin{align} \label{E:ATDr_Diff}
    |\nabla_\Gamma\mathcal{M}_\varepsilon\varphi-\mathbf{P}\mathcal{M}_\varepsilon(\nabla\varphi)| \leq c\varepsilon\mathcal{M}_\varepsilon(|\varphi|+|\nabla\varphi|) \quad\text{on}\quad \Gamma_t.
  \end{align}
\end{lemma}

\begin{proof}
  When $(y,t)\in\overline{S_T}$ and $r\in[\varepsilon g_0(y,t),\varepsilon g_1(y,t)]$, we have
  \begin{align*}
    &|\mathbf{B}^\sharp(y,t,r)-\mathbf{P}(y,t)| = |r\mathbf{W}(y,t)| \leq c\varepsilon, \\
    &|\mathbf{b}_\varepsilon^\sharp(y,t,r)| \leq c\varepsilon, \quad |\mathbf{b}_J^\sharp(y,t,r)| \leq c\varepsilon, \quad |f_J^\sharp(y,t,r)| \leq c
  \end{align*}
  by \eqref{E:G_Bdd}, \eqref{E:J_Poly}, \eqref{E:J_Bdd}, and the smoothness of $\mathbf{W}$, $g_0$, and $g_1$ on $\overline{S_T}$.
  Moreover,
  \begin{align*}
    \mathbf{P}\mathcal{M}_\varepsilon(\nabla\varphi) = \mathcal{M}_\varepsilon\Bigl(\overline{\mathbf{P}}\nabla\varphi\Bigr) \quad\text{on}\quad \Gamma_t, \quad |\partial_\nu\varphi| = |\bar{\bm{\nu}}\cdot\nabla\varphi| \leq |\nabla\varphi| \quad\text{in}\quad \Omega_t^\varepsilon.
  \end{align*}
  Applying these relations to \eqref{E:Ave_TDr}, we find that \eqref{E:ATDr_Diff} is valid.
\end{proof}

\begin{lemma} \label{L:Ave_W1q}
  Let $q\in[1,\infty)$ and $\varphi\in W^{1,q}(\Omega_t^\varepsilon)$.
  Then,
  \begin{align}
    \|\nabla_\Gamma\mathcal{M}_\varepsilon\varphi\|_{L^q(\Gamma_t)} & \leq c\Bigl(\varepsilon^{1-1/q}\|\varphi\|_{L^q(\Omega_t^\varepsilon)}+\varepsilon^{-1/q}\|\nabla\varphi\|_{L^q(\Omega_t^\varepsilon)}\Bigr), \label{E:Ave_GrLq} \\
    \|\mathcal{M}_\varepsilon\varphi\|_{W^{1,q}(\Gamma_t)} &\leq c\varepsilon^{-1/q}\|\varphi\|_{W^{1,q}(\Omega_t^\varepsilon)}. \label{E:Ave_W1q}
  \end{align}
\end{lemma}

\begin{proof}
  We see by \eqref{E:ATDr_Diff} and $|\mathbf{P}\mathbf{a}|\leq|\mathbf{a}|$ on $\overline{S_T}$ for $\mathbf{a}\in\mathbb{R}^n$ that
  \begin{align*}
    |\nabla_\Gamma\mathcal{M}_\varepsilon\varphi| \leq c\{\varepsilon\mathcal{M}_\varepsilon(|\varphi|)+\mathcal{M}_\varepsilon(|\nabla\varphi|)\} \quad\text{on}\quad \Gamma_t.
  \end{align*}
  By this inequality and \eqref{E:Ave_Lq}, we get \eqref{E:Ave_GrLq} and \eqref{E:Ave_W1q}.
\end{proof}

\begin{lemma} \label{L:Adist_W1q}
  Let $q\in[1,\infty)$ and $\varphi\in W^{1,q}(\Omega_t^\varepsilon)$.
  Then,
  \begin{align} \label{E:Adist_W1q}
    \|\mathcal{M}_\varepsilon(d\varphi)\|_{W^{1,q}(\Gamma_t)} \leq c\varepsilon^{1-1/q}\|\varphi\|_{W^{1,q}(\Omega_t^\varepsilon)}.
  \end{align}
  Here, we use the notation $[d\varphi](x)=d(x,t)\varphi(x)$ for $x\in\Omega_t^\varepsilon$.
\end{lemma}

\begin{proof}
  It follows from \eqref{E:Ave_Lq} and $|d|\leq c\varepsilon$ in $\overline{Q_T^\varepsilon}$ that
  \begin{align} \label{Pf_Adi:Lq}
    \|\mathcal{M}_\varepsilon(d\varphi)\|_{L^q(\Gamma_t)} \leq c\varepsilon^{-1/q}\|d\varphi\|_{L^q(\Omega_t^\varepsilon)} \leq c\varepsilon^{1-1/q}\|\varphi\|_{L^q(\Omega_t^\varepsilon)}.
  \end{align}
  Next, we see by \eqref{E:ATDr_Diff} that
  \begin{align} \label{Pf_Adi:TDr}
    |\nabla_\Gamma\mathcal{M}_\varepsilon(d\varphi)| \leq |\mathbf{P}\mathcal{M}_\varepsilon[\nabla(d\varphi)]|+c\varepsilon\mathcal{M}_\varepsilon(|d\varphi|+|\nabla(d\varphi)|) \quad\text{on}\quad \Gamma_t.
  \end{align}
  Since $\nabla(d\varphi)=\varphi\nabla d+d\nabla\varphi=\varphi\bar{\bm{\nu}}+d\nabla\varphi$ in $\Omega_t^\varepsilon$ by \eqref{E:Fermi}, we have
  \begin{align*}
    \mathbf{P}\mathcal{M}_\varepsilon[\nabla(d\varphi)] = (\mathcal{M}_\varepsilon\varphi)\mathbf{P}\bm{\nu}+\mathbf{P}\mathcal{M}_\varepsilon(d\nabla\varphi) = \mathbf{P}\mathcal{M}_\varepsilon(d\nabla\varphi) \quad\text{on}\quad \Gamma_t
  \end{align*}
  by $\mathbf{P}\bm{\nu}=\mathbf{0}_n$ on $\overline{S_T}$.
  Hence, by $|\mathbf{P}\mathbf{a}|\leq|\mathbf{a}|$ on $\overline{S_T}$ for $\mathbf{a}\in\mathbb{R}^n$ and $|d|\leq c\varepsilon$ in $\overline{Q_T^\varepsilon}$,
  \begin{align*}
    |\mathbf{P}\mathcal{M}_\varepsilon[\nabla(d\varphi)]| = |\mathbf{P}\mathcal{M}_\varepsilon(d\nabla\varphi)| \leq c\varepsilon\mathcal{M}_\varepsilon(|\nabla\varphi|) \quad\text{on}\quad \Gamma_t.
  \end{align*}
  Also, since $|\bar{\bm{\nu}}|=1$ and $|d|\leq c\varepsilon\leq c$ in $\overline{Q_T^\varepsilon}$, it follows that
  \begin{align*}
    |d\varphi|+|\nabla(d\varphi)| =|d\varphi|+|\varphi\bar{\bm{\nu}}+d\nabla\varphi| \leq c(|\varphi|+|\nabla\varphi|) \quad\text{in}\quad \Omega_t^\varepsilon.
  \end{align*}
  Applying the above two inequalities to \eqref{Pf_Adi:TDr}, we find that
  \begin{align*}
    |\nabla_\Gamma\mathcal{M}_\varepsilon(d\varphi)| \leq c\varepsilon\mathcal{M}_\varepsilon(|\varphi|+|\nabla\varphi|) \quad\text{on}\quad \Gamma_t.
  \end{align*}
  Thus, we use this inequality and \eqref{E:Ave_Lq} to deduce that
  \begin{align*}
    \|\nabla_\Gamma\mathcal{M}_\varepsilon(d\varphi)\|_{L^q(\Gamma_t^\varepsilon)} \leq c\varepsilon^{1-1/q}\|\varphi\|_{W^{1,q}(\Omega_t^\varepsilon)},
  \end{align*}
  and we obtain \eqref{E:Adist_W1q} by this inequality and \eqref{Pf_Adi:Lq}.
\end{proof}

Let $q\in(1,\infty)$ and $f\in[W^{1,q}(\Gamma_t)]^\ast$.
We define the ``constant extension'' $\bar{f}$ by
\begin{align} \label{E:Def_CEFu}
  \langle\bar{f},\varphi\rangle_{W^{1,q}(\Omega_t^\varepsilon)} := \varepsilon\langle f,g\mathcal{M}_\varepsilon\varphi\rangle_{W^{1,q}(\Gamma_t)}, \quad \varphi\in W^{1,q}(\Omega_t^\varepsilon)
\end{align}
based on \eqref{E:Ave_Pair}.
Note that $\bar{f}$ actually depends on $\varepsilon$ through \eqref{E:Def_CEFu}.

\begin{lemma} \label{L:CE_Func}
  For $f\in[W^{1,q}(\Gamma_t)]^\ast$, we have $\bar{f}\in[W^{1,q}(\Omega_t^\varepsilon)]^\ast$ and
  \begin{align} \label{E:CE_Func}
    \|\bar{f}\|_{[W^{1,q}(\Omega_t^\varepsilon)]^\ast} \leq c\varepsilon^{1/q'}\|f\|_{[W^{1,q}(\Gamma_t)]^\ast}, \quad \frac{1}{q}+\frac{1}{q'} = 1.
  \end{align}
\end{lemma}

\begin{proof}
  By $g\in C^\infty(\overline{S_T})$ and \eqref{E:Ave_W1q}, we have
  \begin{align*}
    |\langle\bar{f},\varphi\rangle_{W^{1,q}(\Omega_t^\varepsilon)}| = \varepsilon|\langle f,g\mathcal{M}_\varepsilon\varphi\rangle_{W^{1,q}(\Gamma_t)}| &\leq \varepsilon\|f\|_{[W^{1,q}(\Gamma_t)]^\ast}\|g\mathcal{M}_\varepsilon\varphi\|_{W^{1,q}(\Gamma_t)} \\
    &\leq c\varepsilon^{1-1/q}\|f\|_{[W^{1,q}(\Gamma_t)]^\ast}\|\varphi\|_{W^{1,q}(\Omega_t^\varepsilon)}
  \end{align*}
  for all $\varphi\in W^{1,q}(\Omega_t^\varepsilon)$.
  Since $1-1/q=1/q'$, we get \eqref{E:CE_Func} by this inequality.
\end{proof}

\subsection{Weak material derivative of the weighted average} \label{SS:WA_WeMa}
Let $C_{0,T}^\infty(\overline{Q_T^\varepsilon})$ and $\mathbb{W}^{p,p'}(Q_T^\varepsilon)$ be the function spaces given in \eqref{E:Def_S0T} and Definition \ref{D:TD_SoBo}, respectively, with $p\in(2,\infty)$ and $1/p+1/p'=1$.
Recall that we abuse the notations \eqref{E:Abuse} when $\varepsilon=0$.
We write $c_T$ for a general positive constant depending on $T$ but independent of $t$ and $\varepsilon$.

\begin{lemma} \label{L:Ave_WeMt}
  If $u\in\mathbb{W}^{p,p'}(Q_T^\varepsilon)$, then $\mathcal{M}_\varepsilon u\in\mathbb{W}^{p,p'}(S_T)$ and
  \begin{align} \label{E:Ave_WeMt}
    \|\partial_\Gamma^\bullet\mathcal{M}_\varepsilon u\|_{L_{[W^{1,p}]^\ast}^{p'}(S_T)} \leq c_T\Bigl(\varepsilon^{-1/p'}\|\partial_\varepsilon^\bullet u\|_{L_{[W^{1,p}]^\ast}^{p'}(Q_T^\varepsilon)}+\varepsilon^{1-1/p}\|u\|_{L_{L^p}^p(Q_T^\varepsilon)}\Bigr).
  \end{align}
  Moreover, $g\mathcal{M}_\varepsilon u\in\mathbb{W}^{p,p'}(S_T)$, and for all $\eta\in L_{W^{1,p}}^p(S_T)$, we have
  \begin{align} \label{E:AWM_Eq}
    \begin{aligned}
      &\frac{1}{\varepsilon}\int_0^T\Bigl\{\langle\partial_\varepsilon^\bullet u,\bar{\eta}\rangle_{W^{1,p}(\Omega_t^\varepsilon)}+(u,\bar{\eta}\,\mathrm{div}\,\mathbf{v}^\varepsilon)_{L^2(\Omega_\varepsilon)}\Bigr\}\,dt \\
      &\qquad = \int_0^T\Bigl\{\langle\partial_\Gamma^\bullet(g\mathcal{M}_\varepsilon u),\eta\rangle_{W^{1,p}(\Gamma_t)}+(g\mathcal{M}_\varepsilon u,\eta\,\mathrm{div}_\Gamma\mathbf{v}_\Gamma)_{L^2(\Gamma_t)}\Bigr\}\,dt.
    \end{aligned}
  \end{align}
\end{lemma}

\begin{proof}
  We have $\mathcal{M}_\varepsilon u\in L_{W^{1,p}}^p(S_T)$ by $u\in L_{W^{1,p}}^p(Q_T^\varepsilon)$ and \eqref{E:Ave_W1q}.

  Let $\eta\in C_{0,T}^\infty(\overline{S_T})$.
  Then, $\bar{\eta}=\eta\circ\pi\in C_{0,T}^\infty(\overline{Q_T^\varepsilon})$ by the smoothness of $\pi$ and thus
  \begin{align*}
    \frac{1}{\varepsilon}\int_0^T\Bigl\{\langle\partial_\varepsilon^\bullet u,\bar{\eta}\rangle_{W^{1,p}(\Omega_t^\varepsilon)}+(u,\bar{\eta}\,\mathrm{div}\,\mathbf{v}^\varepsilon)_{L^2(\Omega_\varepsilon)}\Bigr\}\,dt = -\frac{1}{\varepsilon}\int_0^T(u,\partial_\varepsilon^\bullet\bar{\eta})_{L^2(\Omega_t^\varepsilon)}\,dt
  \end{align*}
  by \eqref{E:Def_WeMat}.
  We further use \eqref{E:CE_Mat} and \eqref{E:Ave_Pair} to the right-hand side.
  Then,
  \begin{align} \label{Pf_AWM:duav}
    \frac{1}{\varepsilon}\int_0^T\Bigl\{\langle\partial_\varepsilon^\bullet u,\bar{\eta}\rangle_{W^{1,p}(\Omega_t^\varepsilon)}+(u,\bar{\eta}\,\mathrm{div}\,\mathbf{v}^\varepsilon)_{L^2(\Omega_\varepsilon)}\Bigr\}\,dt = -\int_0^T(g\mathcal{M}_\varepsilon u,\partial_\Gamma^\bullet\eta)_{L^2(\Gamma_t)}\,dt.
  \end{align}
  Let $\sigma_{\Gamma,g}=\mathrm{div}_\Gamma\mathbf{v}_\Gamma+[\partial_\Gamma^\bullet g/g]$ on $\overline{S_T}$.
  We see by \eqref{E:Ave_Pair} that
  \begin{align*}
    \frac{1}{\varepsilon}\int_0^T\bigl(u,\overline{\eta\sigma_{\Gamma,g}}\bigr)_{L^2(\Omega_t^\varepsilon)}\,dt = \int_0^T(\mathcal{M}_\varepsilon u,(\partial_\Gamma^\bullet g)\eta+g\eta\,\mathrm{div}_\Gamma\mathbf{v}_\Gamma)_{L^2(\Gamma_t)}\,dt.
  \end{align*}
  By this equality, \eqref{Pf_AWM:duav}, and $\partial_\Gamma^\bullet(g\eta)=(\partial_\Gamma^\bullet g)\eta+g(\partial_\Gamma^\bullet\eta)$ on $\overline{S_T}$, we obtain
  \begin{align*}
    &\frac{1}{\varepsilon}\int_0^T\Bigl\{\langle\partial_\varepsilon^\bullet u,\bar{\eta}\rangle_{W^{1,p}(\Omega_t^\varepsilon)}+\bigl(u,\bar{\eta}(\mathrm{div}\,\mathbf{v}^\varepsilon-\bar{\sigma}_{\Gamma,g})\bigr)_{L^2(\Omega_\varepsilon)}\Bigr\}\,dt \\
    &\qquad = -\int_0^T(\mathcal{M}_\varepsilon u,\partial_\Gamma^\bullet(g\eta)+g\eta\,\mathrm{div}_\Gamma\mathbf{v}_\Gamma)_{L^2(\Gamma_t)}\,dt.
  \end{align*}
  In this equality, we replace $\eta$ by $\eta/g$.
  Moreover, we apply \eqref{E:Vls_MTD} and H\"{o}lder's inequality to the left-hand side.
  Then, we find that
  \begin{align*}
    &\left|\int_0^T(\mathcal{M}_\varepsilon u,\partial_\Gamma^\bullet\eta+\eta\,\mathrm{div}_\Gamma\mathbf{v}_\Gamma)_{L^2(\Gamma_t)}\,dt\right| \\
    &\qquad \leq \frac{1}{\varepsilon}\|\partial_\varepsilon^\bullet u\|_{L_{[W^{1,p}]^\ast}^{p'}(Q_T^\varepsilon)}\|\bar{\eta}/\bar{g}\|_{L_{W^{1,p}}^p(Q_T^\varepsilon)}+c\|u\|_{L_{L^2}^2(Q_T^\varepsilon)}\|\bar{\eta}/\bar{g}\|_{L_{L^2}^2(Q_T^\varepsilon)}.
  \end{align*}
  Moreover, we see by \eqref{E:L2Lp_MTD} and H\"{o}lder's inequality in time that
  \begin{align*}
    \|u\|_{L_{L^2}^2(Q_T^\varepsilon)}\|\bar{\eta}/\bar{g}\|_{L_{L^2}^2(Q_T^\varepsilon)} &\leq c_T\varepsilon^{1-2/p}\|u\|_{L_{L^p}^p(Q_T^\varepsilon)}\|\bar{\eta}/\bar{g}\|_{L_{L^p}^p(Q_T^\varepsilon)}.
  \end{align*}
  Also, it follows from \eqref{E:CE_Lq}, $g\in C^\infty(\overline{S_T})$, and \eqref{E:G_Bdd} that
  \begin{align*}
    \|\bar{\eta}/\bar{g}\|_{L_{L^p}^p(Q_T^\varepsilon)} \leq \|\bar{\eta}/\bar{g}\|_{L_{W^{1,p}}^p(Q_T^\varepsilon)} \leq c\varepsilon^{1/p}\|\eta/g\|_{L_{W^{1,p}}^p(S_T)} \leq c\varepsilon^{1/p}\|\eta\|_{L_{W^{1,p}}^p(S_T)}.
  \end{align*}
  Combining the above inequalities and using $-1+1/p=-1/p'$, we have
  \begin{align*}
    &\left|\int_0^T(\mathcal{M}_\varepsilon u,\partial_\Gamma^\bullet\eta+\eta\,\mathrm{div}_\Gamma\mathbf{v}_\Gamma)_{L^2(\Gamma_t)}\,dt\right| \\
    &\qquad \leq c_T\Bigl(\varepsilon^{-1/p'}\|\partial_\varepsilon^\bullet u\|_{L_{[W^{1,p}]^\ast}^{p'}(Q_T^\varepsilon)}+\varepsilon^{1-1/p}\|u\|_{L_{L^p}^p(Q_T^\varepsilon)}\Bigr)\|\eta\|_{L_{W^{1,p}}^p(S_T)}
  \end{align*}
  for all $\eta\in C_{0,T}^\infty(\overline{S_T})$.
  Since $C_{0,T}^\infty(\overline{S_T})$ is dense in $L_{W^{1,p}}^p(S_T)$ by Lemma \ref{L:EvBo_Den}, we see by the above inequality, Definition \ref{D:We_MatDe}, and Remark \ref{R:We_MatDe} that
  \begin{align*}
    \partial_\Gamma^\bullet\mathcal{M}_\varepsilon u \in [L_{W^{1,p}}^p(S_T)]^\ast = L_{[W^{1,p}]^\ast}^{p'}(S_T), \quad \mathcal{M}_\varepsilon u \in \mathbb{W}^{p,p'}(S_T)
  \end{align*}
  and the inequality \eqref{E:Ave_WeMt} holds.
  Moreover, $g\mathcal{M}_\varepsilon u\in\mathbb{W}^{p,p'}(S_T)$ by Lemma \ref{L:Wpp_Mult}, and we have \eqref{E:AWM_Eq} for $\eta\in C_{0,T}^\infty(\overline{S_T})$ by using \eqref{E:Def_WeMat} (with $\varepsilon=0$) to the right-hand side of \eqref{Pf_AWM:duav}.
  Thus, by a density argument, we find that \eqref{E:AWM_Eq} also holds for $\eta\in L_{W^{1,p}}^p(S_T)$.
\end{proof}

%%% Section 7 %%%
\section{Thin-film limit problem} \label{S:TFLPr}
The goal of this section is to prove Theorems \ref{T:TFL} and \ref{T:Lim_UnEx}.
We use the notations given in the previous sections, and write $c$ and $c_T$ for general positive constants independent of $t$ and $\varepsilon$, but $c_T$ depends on $T$ in general.
Also, we suppress the time variable of functions in the time integrals.

Let $p\in(2,\infty)$ and $p'\in(1,2)$ satisfy $1/p+1/p'=1$.
Throughout this section, we impose the assumptions of Theorem \ref{T:TFL} and denote by $u^\varepsilon$ the unique weak solution to \eqref{E:pLap_MTD} given in Proposition \ref{P:ExUn_MTD}.

\subsection{Convergence of the averaged solution} \label{SS:TF_CoAv}
First, we see that the weighted average of $u^\varepsilon$ converges in an appropriate sense.

\begin{proposition} \label{P:uAv_Bdd}
  We define
  \begin{align} \label{E:Def_veps}
    v^\varepsilon := \mathcal{M}_\varepsilon u^\varepsilon, \quad \mathbf{w}^\varepsilon := \mathcal{M}_\varepsilon(|\nabla u^\varepsilon|^{p-2}\nabla u^\varepsilon), \quad \zeta^\varepsilon := \mathcal{M}_\varepsilon(\partial_\nu u^\varepsilon) \quad\text{on}\quad S_T,
  \end{align}
  where $\partial_\nu u^\varepsilon=\bar{\bm{\nu}}\cdot\nabla u^\varepsilon$ is the derivative of $u^\varepsilon$ in the direction of $\bar{\bm{\nu}}$.
  Then,
  \begin{align*}
    v^\varepsilon\in\mathbb{W}^{p,p'}(S_T) \subset C_{L^2}(S_T), \quad \mathbf{w}^\varepsilon\in [L_{L^{p'}}^{p'}(S_T)]^n, \quad \zeta^\varepsilon\in L_{L^p}^p(S_T).
  \end{align*}
  Moreover, they are bounded uniformly in $\varepsilon$.
  More precisely,
  \begin{align} \label{E:uAv_Bdd}
    \|v^\varepsilon\|_{\mathbb{W}^{p,p'}(S_T)} \leq c_T, \quad \|\mathbf{w}^\varepsilon\|_{L_{L^{p'}}^{p'}(S_T)} \leq c_T, \quad \|\zeta^\varepsilon\|_{L_{L^p}^p(S_T)} \leq c_T.
  \end{align}
\end{proposition}

\begin{proof}
  The regularity result follows from Definition \ref{D:WS_MTD} and Lemmas \ref{L:Trans}, \ref{L:Ave_Lq}, \ref{L:Ave_W1q}, and \ref{L:Ave_WeMt}.
  Next, we see by \eqref{E:Ave_Lq}, \eqref{E:Ave_W1q}, and $|\partial_\nu u^\varepsilon|=|\bar{\bm{\nu}}\cdot\nabla u^\varepsilon|\leq|\nabla u^\varepsilon|$ in $Q_T^\varepsilon$ that
  \begin{align*}
    \|v^\varepsilon\|_{L_{W^{1,p}}^p(S_T)} &\leq c\varepsilon^{-1/p}\|u^\varepsilon\|_{L_{W^{1,p}}^p(Q_T^\varepsilon)}, \\
    \|\mathbf{w}^\varepsilon\|_{L_{L^{p'}}^{p'}(S_T)} &\leq c\varepsilon^{-1/p'}\Bigl\|\,|\nabla u^\varepsilon|^{p-2}\nabla u^\varepsilon\,\Bigr\|_{L_{L^{p'}}^{p'}(Q_T^\varepsilon)} = c\varepsilon^{-1+1/p}\|\nabla u^\varepsilon\|_{L_{L^p}^p(Q_T^\varepsilon)}^{p-1}, \\
    \|\zeta^\varepsilon\|_{L_{L^p}^p(S_T)} &\leq c\varepsilon^{-1/p}\|\partial_\nu u^\varepsilon\|_{L_{L^p}^p(Q_T^\varepsilon)} \leq c\varepsilon^{-1/p}\|\nabla u^\varepsilon\|_{L_{L^p}^p(Q_T^\varepsilon)}.
  \end{align*}
  Note that we used $1/p'=1-1/p$ in the second line.
  Also, by \eqref{E:Ave_WeMt},
  \begin{align*}
    \|\partial_\Gamma^\bullet v^\varepsilon\|_{L_{[W^{1,p}]^\ast}^{p'}(S_T)} \leq c_T\Bigl(\varepsilon^{-1/p'}\|\partial_\varepsilon^\bullet u^\varepsilon\|_{L_{[W^{1,p}]^\ast}^{p'}(Q_T^\varepsilon)}+\varepsilon^{1-1/p}\|u^\varepsilon\|_{L_{L^p}^p(Q_T^\varepsilon)}\Bigr).
  \end{align*}
  Since $u_0^\varepsilon$ and $f^\varepsilon$ satisfies \eqref{E:Data_MTD}, we can apply \eqref{E:Ener_MTD} and \eqref{E:dudt_MTD} to the above inequalities.
  By this fact and $0<\varepsilon<1$, we obtain \eqref{E:uAv_Bdd}.
\end{proof}

\begin{proposition} \label{P:veps_Conv}
  There exist functions
  \begin{align*}
    v\in\mathbb{W}^{p,p'}(S_T)\subset C_{L^2}(S_T), \quad \mathbf{w}\in [L_{L^{p'}}^{p'}(S_T)]^n, \quad \zeta\in L_{L^p}^p(S_T)
  \end{align*}
  such that, up to a subsequence,
  \begin{align} \label{E:veps_WeCo}
    \begin{alignedat}{2}
      \lim_{\varepsilon\to0}v^\varepsilon &= v &\quad &\text{weakly in $\mathbb{W}^{p,p'}(S_T)$}, \\
      \lim_{\varepsilon\to0}\mathbf{w}^\varepsilon &= \mathbf{w} &\quad &\text{weakly in $[L_{L^{p'}}^{p'}(S_T)]^n$}, \\
      \lim_{\varepsilon\to0}\zeta^\varepsilon &= \zeta &\quad &\text{weakly in $L_{L^p}^p(S_T)$}.
    \end{alignedat}
  \end{align}
  Moreover, up to a subsequence,
  \begin{align} \label{E:veps_StCo}
    \lim_{\varepsilon\to0}v^\varepsilon = v \quad\text{strongly in $L_{L^2}^2(S_T)$ and thus in $L_{p'}^{p'}(S_T)$}.
  \end{align}
\end{proposition}

\begin{proof}
  The statement follows from \eqref{E:uAv_Bdd}, Lemmas \ref{L:Trans} and \ref{L:AuLi}, and the fact that $L_{L^2}^2(S_T)$ is continuously embedded into $L_{L^{p'}}^{p'}(S_T)$ by $p'<2$.
\end{proof}

It turns out later that the pair $(v,\zeta)$ is a unique weak solution to the limit problem \eqref{E:pLap_Lim} and $\mathbf{w}$ is uniquely determined by $v$ and $\zeta$.
Thus, the convergence \eqref{E:veps_WeCo} up to a subsequence can be replaced by that of the whole sequence $\varepsilon\to0$.
Based on this fact, we simply write the limit $\varepsilon\to0$ in the following discussions.

\begin{proposition} \label{P:gvep_WeCp}
  We have $gv^\varepsilon,gv\in\mathbb{W}^{p,p'}(S_T)$ and
  \begin{align} \label{E:gvep_WeCo}
    \lim_{\varepsilon\to0}\partial_\Gamma^\bullet(gv^\varepsilon) = \partial_\Gamma^\bullet(gv) \quad\text{weakly in $L_{[W^{1,p}]^\ast}^{p'}(S_T)$}.
  \end{align}
\end{proposition}

\begin{proof}
  By $g\in C^\infty(\overline{S_T})$ and Lemma \ref{L:Wpp_Mult}, we get $gv^\varepsilon,gv\in\mathbb{W}^{p,p'}(S_T)$.
  Also,
  \begin{alignat*}{2}
    \lim_{\varepsilon\to0}v^\varepsilon &= v &\quad &\text{weakly in $L_{W^{1,p}}^p(S_T)$}, \\
    \lim_{\varepsilon\to0}\partial_\Gamma^\bullet v^\varepsilon &= \partial_\Gamma^\bullet v &\quad &\text{weakly in $L_{[W^{1,p}]^\ast}^{p'}(S_T)$}
  \end{alignat*}
  by \eqref{E:veps_WeCo}, and thus we see by \eqref{E:WeMa_Mult} (with $\varepsilon=0$) that \eqref{E:gvep_WeCo} follows.
\end{proof}

\subsection{Limit weak form} \label{SS:TF_LiWF}
Next, we derive a weak form satisfied by $v$ and $\mathbf{w}$.

\begin{proposition} \label{P:LiWF_uw}
  For all $\eta\in L_{W^{1,p}}^p(S_T)$, we have
  \begin{align} \label{E:LiWF_uw}
    \begin{aligned}
      &\int_0^T\langle\partial_\Gamma^\bullet(gv),\eta\rangle_{W^{1,p}(\Gamma_t)}\,dt+\int_0^T(g\mathbf{w},\nabla_\Gamma\eta)_{L^2(\Gamma_t)}\,dt \\
      &\qquad +\int_0^T(gv,\mathbf{v}_\Gamma\cdot\nabla_\Gamma\eta+\eta\,\mathrm{div}_\Gamma\mathbf{v}_\Gamma)_{L^2(\Gamma_t)}\,dt = \int_0^T\langle f,g\eta\rangle_{W^{1,p}(\Gamma_t)}\,dt.
    \end{aligned}
  \end{align}
\end{proposition}

\begin{proof}
  Let $\eta\in L_{W^{1,p}}^p(S_T)$.
  Then, $\bar{\eta}\in L_{W^{1,p}}^p(Q_T^\varepsilon)$ by \eqref{E:CE_Lq}.
  We set $\psi=\bar{\eta}$ in \eqref{E:WeFo_MTD} and divide both sides by $\varepsilon$ to get $\sum_{k=1}^3I_k^\varepsilon=I_4^\varepsilon$, where
  \begin{align} \label{Pf_LWF:Ieps}
    \begin{aligned}
      I_1^\varepsilon &:= \frac{1}{\varepsilon}\int_0^T\langle\partial_\varepsilon^\bullet u^\varepsilon,\bar{\eta}\rangle_{W^{1,p}(\Omega_t^\varepsilon)}\,dt+\frac{1}{\varepsilon}\int_0^T(u^\varepsilon,\bar{\eta}\,\mathrm{div}\,\mathbf{v}^\varepsilon)_{L^2(\Omega_t^\varepsilon)}\,dt, \\
      I_2^\varepsilon &:= \frac{1}{\varepsilon}\int_0^T(|\nabla u^\varepsilon|^{p-2}\nabla u^\varepsilon,\nabla\bar{\eta})_{L^2(\Omega_t^\varepsilon)}\,dt, \\
      I_3^\varepsilon &:= \frac{1}{\varepsilon}\int_0^T(u^\varepsilon,\mathbf{v}^\varepsilon\cdot\nabla\bar{\eta})_{L^2(\Omega_t^\varepsilon)}\,dt, \\
      I_4^\varepsilon &:= \frac{1}{\varepsilon}\int_0^T\langle f^\varepsilon,\bar{\eta}\rangle_{W^{1,p}(\Omega_t^\varepsilon)}\,dt.
    \end{aligned}
  \end{align}
  Let us consider the limit of each $I_k^\varepsilon$ as $\varepsilon\to0$.
  First, we have
  \begin{align} \label{Pf_LWF:I1}
    \begin{aligned}
      \lim_{\varepsilon\to0}I_1^\varepsilon &= \lim_{\varepsilon\to0}\left(\int_0^T\langle\partial_\Gamma^\bullet(gv^\varepsilon),\eta\rangle_{W^{1,p}(\Gamma_t)}\,dt+\int_0^T(gv^\varepsilon,\eta\,\mathrm{div}_\Gamma\mathbf{v}_\Gamma)_{L^2(\Gamma_t)}\,dt\right) \\
      &= \int_0^T\langle\partial_\Gamma^\bullet(gv),\eta\rangle_{W^{1,p}(\Gamma_t)}\,dt+\int_0^T(gv,\eta\,\mathrm{div}_\Gamma\mathbf{v}_\Gamma)_{L^2(\Gamma_t)}\,dt
    \end{aligned}
  \end{align}
  by \eqref{E:AWM_Eq}, \eqref{E:Def_veps}, \eqref{E:veps_WeCo}, and \eqref{E:gvep_WeCo}.
  Next, we split $I_2^\varepsilon=I_{2,1}^\varepsilon+I_{2,2}^\varepsilon$ into
  \begin{align*}
    I_{2,1}^\varepsilon &:= \frac{1}{\varepsilon}\int_0^T(|\nabla u^\varepsilon|^{p-2}\nabla u^\varepsilon,\nabla\bar{\eta})_{L^2(\Omega_t^\varepsilon)}\,dt-\int_0^T(g\mathbf{w}^\varepsilon,\nabla_\Gamma\eta)_{L^2(\Gamma_t)}\,dt, \\
    I_{2,2}^\varepsilon &:= \int_0^T(g\mathbf{w}^\varepsilon,\nabla_\Gamma\eta)_{L^2(\Gamma_t)}\,dt.
  \end{align*}
  Then, we see by \eqref{E:Ave_Pair} and \eqref{E:Def_veps} that
  \begin{align*}
    I_{2,1}^\varepsilon = \frac{1}{\varepsilon}\int_0^T\Bigl(|\nabla u^\varepsilon|^{p-2}\nabla u^\varepsilon,\nabla\bar{\eta}-\overline{\nabla_\Gamma\eta}\Bigr)_{L^2(\Gamma_t)}\,dt.
  \end{align*}
  Thus, by \eqref{E:CEGr_NB}, $|d|\leq c\varepsilon$ in $Q_T^\varepsilon$, H\"{o}lder's inequality, \eqref{E:CE_Lq}, and \eqref{E:Ener_MTD},
  \begin{align*}
    |I_{2,1}^\varepsilon| \leq c\|\nabla u^\varepsilon\|_{L_{L^p}^p(Q_T^\varepsilon)}^{p-1}\Bigl\|\overline{\nabla_\Gamma\eta}\Bigr\|_{L_{L^p}^p(Q_T^\varepsilon)} \leq c_T\varepsilon\|\nabla_\Gamma\eta\|_{L_{L^p}^p(S_T)} \to 0
  \end{align*}
  as $\varepsilon\to0$.
  From this result and \eqref{E:veps_WeCo}, we deduce that
  \begin{align} \label{Pf_LWF:I2}
    \lim_{\varepsilon\to0}I_2^\varepsilon = \lim_{\varepsilon\to0}I_{2,2}^\varepsilon = \int_0^T(g\mathbf{w},\nabla_\Gamma\eta)_{L^2(\Gamma_t)}\,dt.
  \end{align}
  Let us consider $I_3^\varepsilon$.
  We write $I_3^\varepsilon=I_{3,1}^\varepsilon+I_{3,2}^\varepsilon$, where
  \begin{align*}
    I_{3,1}^\varepsilon &:= \frac{1}{\varepsilon}\int_0^T(u^\varepsilon,\mathbf{v}^\varepsilon\cdot\nabla\bar{\eta})_{L^2(\Omega_t^\varepsilon)}-\frac{1}{\varepsilon}\int_0^T\Bigl(u^\varepsilon,\overline{\mathbf{v}_\Gamma\cdot\nabla_\Gamma\eta}\Bigr)_{L^2(\Omega_t^\varepsilon)}\,dt, \\
    I_{3,2}^\varepsilon &:= \frac{1}{\varepsilon}\int_0^T\Bigl(u^\varepsilon,\overline{\mathbf{v}_\Gamma\cdot\nabla_\Gamma\eta}\Bigr)_{L^2(\Omega_t^\varepsilon)}\,dt.
  \end{align*}
  It follows from \eqref{E:CEGr_NB}, $|d|\leq c\varepsilon$ in $\overline{Q_T^\varepsilon}$, \eqref{E:Vls_MTD}, \eqref{E:BdV_MTD}, and \eqref{E:CE_Lq} that
  \begin{align*}
    \Bigl\|\mathbf{v}^\varepsilon\cdot\nabla\bar{\eta}-\overline{\mathbf{v}_\Gamma\cdot\nabla_\Gamma\eta}\Bigr\|_{L^2(\Omega_t^\varepsilon)} \leq c\varepsilon\Bigl\|\overline{\nabla_\Gamma\eta}\Bigr\|_{L^2(\Omega_t^\varepsilon)} \leq c\varepsilon^{3/2}\|\nabla_\Gamma\eta\|_{L^2(\Gamma_t)}.
  \end{align*}
  Thus, by H\"{o}lder's inequality and \eqref{E:Ener_MTD}, we find that
  \begin{align*}
    |I_{3,1}^\varepsilon| \leq c\varepsilon^{1/2}\int_0^T\|u^\varepsilon\|_{L^2(\Omega_t^\varepsilon)}\|\nabla_\Gamma\eta\|_{L^2(\Gamma_t)}\,dt \leq c_T\varepsilon\int_0^T\|\nabla_\Gamma\eta\|_{L^2(\Gamma_t)}\,dt \to 0
  \end{align*}
  as $\varepsilon\to0$ (note that $\nabla_\Gamma\eta\in L_{L^p}^p(S_T)\subset L_{L^2}^1(S_T)$ by $p>2$).
  Moreover, since
  \begin{align*}
    I_{3,2}^\varepsilon = \int_0^T(gv^\varepsilon,\mathbf{v}_\Gamma\cdot\nabla_\Gamma\eta)_{L^2(\Gamma_t)}\,dt
  \end{align*}
  by \eqref{E:Ave_Pair} and \eqref{E:Def_veps}, it follows from \eqref{E:veps_WeCo} that
  \begin{align} \label{Pf_LWF:I3}
    \lim_{\varepsilon\to0}I_3^\varepsilon = \lim_{\varepsilon\to0}I_{3,2}^\varepsilon = \int_0^T(gv,\mathbf{v}_\Gamma\cdot\nabla_\Gamma\eta)_{L^2(\Gamma_t)}\,dt.
  \end{align}
  Lastly, we consider $I_4^\varepsilon$.
  We set
  \begin{align*}
    I_{4,1}^\varepsilon &:= \frac{1}{\varepsilon}\int_0^T\langle f^\varepsilon,\bar{\eta}\rangle_{W^{1,p}(\Omega_t^\varepsilon)}\,dt-\frac{1}{\varepsilon}\int_0^T\langle\bar{f},\bar{\eta}\rangle_{W^{1,p}(\Omega_t^\varepsilon)}\,dt, \\
    I_{4,2}^\varepsilon &:= \frac{1}{\varepsilon}\int_0^T\langle\bar{f},\bar{\eta}\rangle_{W^{1,p}(\Omega_t^\varepsilon)}\,dt-\int_0^T\langle f,g\eta\rangle_{W^{1,p}(\Gamma_t)}\,dt
  \end{align*}
  so that $I_4^\varepsilon=I_{4,1}^\varepsilon+I_{4,2}^\varepsilon+\int_0^T\langle f,g\eta\rangle_{W^{1,p}(\Gamma_t)}\,dt$.
  Then,
  \begin{align*}
    |I_{4,1}^\varepsilon| &\leq \frac{1}{\varepsilon}\|f^\varepsilon-\bar{f}\|_{L_{[W^{1,p}]^\ast}^{p'}(Q_T^\varepsilon)}\|\bar{\eta}\|_{L_{W^{1,p}}^p(Q_T^\varepsilon)} \\
    &\leq c\varepsilon^{-1+1/p}\|f^\varepsilon-\bar{f}\|_{L_{[W^{1,p}]^\ast}^{p'}(Q_T^\varepsilon)}\|\eta\|_{L_{W^{1,p}}^p(S_T)} \to 0
  \end{align*}
  as $\varepsilon\to0$ by H\"{o}lder's inequality, \eqref{E:CE_Lq}, $-1+1/p=-1/p'$, and \eqref{E:ExF_StCo}.
  Also,
  \begin{align*}
    |I_{4,2}^\varepsilon| &= \left|\int_0^T\langle f,g\mathcal{M}_\varepsilon\bar{\eta}\rangle_{W^{1,p}(\Gamma_t)}\,dt-\int_0^T\langle f,g\eta\rangle_{W^{1,p}(\Gamma_t)}\,dt\right| \\
    &\leq \|f\|_{L_{[W^{1,p}]^\ast}^{p'}(S_T)}\|g(\mathcal{M}_\varepsilon\bar{\eta}-\eta)\|_{L_{W^{1,p}}^p(S_T)} \\
    &\leq c\varepsilon\|f\|_{L_{[W^{1,p}]^\ast}^{p'}(S_T)}\|\eta\|_{L_{W^{1,p}}^p(S_T)} \to 0
  \end{align*}
  as $\varepsilon\to0$ by \eqref{E:Def_CEFu}, $g\in C^\infty(\overline{S_T})$, and \eqref{E:AvCE_W1q}.
  Therefore,
  \begin{align} \label{Pf_LWF:I4}
    \lim_{\varepsilon\to0}I_4^\varepsilon = \int_0^T\langle f,g\eta\rangle_{W^{1,p}(\Gamma_t)}\,dt,
  \end{align}
  and we get \eqref{E:LiWF_uw} by letting $\varepsilon\to0$ in $\sum_{k=1}^3I_k^\varepsilon=I_4^\varepsilon$ and using \eqref{Pf_LWF:I1}--\eqref{Pf_LWF:I4}.
\end{proof}

We also verify the initial condition for $v$.

\begin{proposition} \label{P:LiWF_Ini}
  We have $v(0)=v_0$ in $L^2(\Gamma_0)$.
\end{proposition}

\begin{proof}
  Let $\Phi_{\pm(\cdot)}^0$ be the mappings given in Assumption \ref{A:Flow_Sur}.
  Also, let $\theta=\theta(t)$ be a smooth function on $[0,T]$ satisfying $\theta(0)=1$ and $\theta(T)=0$.
  We set
  \begin{align*}
    \eta(y,t) := \theta(t)\eta_0\bigl(\Phi_{-t}^0(y)\bigr), \quad (y,t) \in \overline{S_T}
  \end{align*}
  for any $\eta_0\in C^\infty(\Gamma_0)$.
  Then,
  \begin{align} \label{Pf_LIn:Test}
    \eta \in C^\infty(\overline{S_T}), \quad \eta(0) = \eta_0 \quad\text{on}\quad \Gamma_0, \quad \eta(T) = 0 \quad\text{on}\quad \Gamma_T.
  \end{align}
  We set $\psi=\bar{\eta}$ in \eqref{E:WeFo_MTD}, divide both sides by $\varepsilon$, and use \eqref{E:Trans}.
  Then, we have
  \begin{align*}
    -\frac{1}{\varepsilon}(u_0^\varepsilon,\bar{\eta}_0)_{L^2(\Omega_0^\varepsilon)}-\frac{1}{\varepsilon}\int_0^T\Bigl(u^\varepsilon,\overline{\partial_\Gamma^\bullet\eta}\Bigr)_{L^2(\Omega_t^\varepsilon)}\,dt+I_2^\varepsilon+I_3^\varepsilon = I_4^\varepsilon
  \end{align*}
  by \eqref{E:CE_Mat} and \eqref{Pf_LIn:Test}, where $I_2^\varepsilon$, $I_3^\varepsilon$, and $I_4^\varepsilon$ are given by \eqref{Pf_LWF:Ieps}.
  Let $\varepsilon\to0$.
  Then,
  \begin{align*}
    \lim_{\varepsilon\to0}\frac{1}{\varepsilon}(u_0^\varepsilon,\bar{\eta}_0)_{L^2(\Omega_0^\varepsilon)} = \lim_{\varepsilon\to0}(g\mathcal{M}_\varepsilon u_0^\varepsilon,\eta_0)_{L^2(\Gamma_0)} = (gv_0,\eta_0)_{L^2(\Gamma_0)}
  \end{align*}
  by \eqref{E:Ave_Pair} and \eqref{E:u0Av_WeCo}.
  Also, we see by \eqref{E:Ave_Pair}, \eqref{E:Def_veps}, and \eqref{E:veps_WeCo} that
  \begin{align*}
    \lim_{\varepsilon\to0}\frac{1}{\varepsilon}\int_0^T\Bigl(u^\varepsilon,\overline{\partial_\Gamma^\bullet\eta}\Bigr)_{L^2(\Omega_t^\varepsilon)}\,dt = \lim_{\varepsilon\to0}\int_0^T(gv^\varepsilon,\partial_\Gamma^\bullet\eta)_{L^2(\Gamma_t)}\,dt = \int_0^T(gv,\partial_\Gamma^\bullet\eta)_{L^2(\Gamma_t)}\,dt.
  \end{align*}
  By these results and \eqref{Pf_LWF:I2}--\eqref{Pf_LWF:I4}, we obtain
  \begin{align*}
    &(g(0)v_0,\eta_0)_{L^2(\Gamma_0)}-\int_0^T(gv,\partial_\Gamma^\bullet\eta)_{L^2(\Gamma_t)}\,dt+\int_0^T(g\mathbf{w},\nabla_\Gamma\eta)_{L^2(\Gamma_t)}\,dt \\
    &\qquad +\int_0^T(gv,\mathbf{v}_\Gamma\cdot\nabla\eta)_{L^2(\Gamma_t)}\,dt = \int_0^T\langle f,g\eta\rangle_{W^{1,p}(\Gamma_t)}\,dt.
  \end{align*}
  Also, we take the above $\eta$ in \eqref{E:LiWF_uw} and use \eqref{E:Trans} (with $\varepsilon=0$) and \eqref{Pf_LIn:Test} to get
  \begin{align*}
    &(g(0)v(0),\eta_0)_{L^2(\Gamma_0)}-\int_0^T(gv,\partial_\Gamma^\bullet\eta)_{L^2(\Gamma_t)}\,dt+\int_0^T(g\mathbf{w},\nabla_\Gamma\eta)_{L^2(\Gamma_t)}\,dt \\
    &\qquad +\int_0^T(gv,\mathbf{v}_\Gamma\cdot\nabla\eta)_{L^2(\Gamma_t)}\,dt = \int_0^T\langle f,g\eta\rangle_{W^{1,p}(\Gamma_t)}\,dt.
  \end{align*}
  Comparing these relations, we find that
  \begin{align*}
    (g(0)v(0),\eta_0)_{L^2(\Gamma_0)} = (g(0)v_0,\eta_0)_{L^2(\Gamma_0)} \quad\text{for all}\quad \eta_0 \in C^\infty(\Gamma_0).
  \end{align*}
  Thus, $v(0)=v_0$ in $L^2(\Gamma_0)$ by \eqref{E:G_Bdd} and the density of $C^\infty(\Gamma_0)$ in $L^2(\Gamma_0)$.
\end{proof}

\subsection{Characterization of the gradient term} \label{SS:TF_Char}
Let us characterize $\mathbf{w}$ and $\zeta$ in terms of $v$.
First, we determine the normal component of $\mathbf{w}$.
Recall that $C_{0,T}^\infty(\overline{S_T})$ is the function space given in \eqref{E:Def_S0T} with $\varepsilon=0$ (see also \eqref{E:Abuse} for the notations).

\begin{proposition} \label{P:Chw_Nor}
  We have $\mathbf{w}\cdot\bm{\nu}+V_\Gamma v=0$ a.e. on $S_T$, where $V_\Gamma=\mathbf{v}_\Gamma\cdot\bm{\nu}$ is the scalar outer normal velocity of $\Gamma_t$.
\end{proposition}

\begin{proof}
  Let $\eta\in C_{0,T}^\infty(\overline{S_T})$.
  We divide \eqref{E:WeFo_MTD} by $\varepsilon$ and substitute
  \begin{align*}
    \psi(x,t) = [d\bar{\eta}](x,t) = d(x,t)\bar{\eta}(x,t), \quad (x,t)\in Q_T^\varepsilon
  \end{align*}
  to get $\sum_{j=1}^3K_j^\varepsilon=K_4^\varepsilon$, where
  \begin{align*}
    K_1^\varepsilon &:= \frac{1}{\varepsilon}\int_0^T\langle\partial_\varepsilon^\bullet u^\varepsilon,d\bar{\eta}\rangle_{W^{1,p}(\Omega_\varepsilon)}\,dt+\frac{1}{\varepsilon}\int_0^T(u^\varepsilon,d\bar{\eta}\,\mathrm{div}\,\mathbf{v}^\varepsilon)_{L^2(\Omega_t^\varepsilon)}\,dt, \\
    K_2^\varepsilon &:= \frac{1}{\varepsilon}\int_0^T\bigl(|\nabla u^\varepsilon|^{p-2}\nabla u^\varepsilon,\nabla(d\bar{\eta})\bigr)_{L^2(\Omega_t^\varepsilon)}\,dt, \\
    K_3^\varepsilon &:= \frac{1}{\varepsilon}\int_0^T\bigl(u^\varepsilon,\mathbf{v}^\varepsilon\cdot\nabla(d\bar{\eta})\bigr)_{L^2(\Omega_t^\varepsilon)}\,dt, \\
    K_4^\varepsilon &:= \frac{1}{\varepsilon}\int_0^T\langle f^\varepsilon,d\bar{\eta}\rangle_{W^{1,p}(\Omega_t^\varepsilon)}\,dt.
  \end{align*}
  Let us first show $K_1^\varepsilon,K_4^\varepsilon\to0$ as $\varepsilon\to0$.
  We have
  \begin{align*}
    K_1^\varepsilon = -\frac{1}{\varepsilon}\int_0^T\bigl(u^\varepsilon,\partial_\varepsilon^\bullet(d\bar{\eta})\bigr)_{L^2(\Omega_t^\varepsilon)}\,dt = \frac{1}{\varepsilon}\int_0^T\Bigl(u^\varepsilon,\bar{\eta}\,\partial_\varepsilon^\bullet d+d\,\overline{\partial_\Gamma^\bullet\eta}\Bigr)_{L^2(\Omega_t^\varepsilon)}\,dt
  \end{align*}
  by \eqref{E:Trans}, $\bar{\eta}(0)=\bar{\eta}(T)=0$, and \eqref{E:CE_Mat}.
  Thus,
  \begin{align*}
    |K_1^\varepsilon| &\leq c\int_0^T\|u^\varepsilon\|_{L^2(\Omega_t^\varepsilon)}\left(\|\bar{\eta}\|_{L^2(\Omega_t^\varepsilon)}+\Bigl\|\overline{\partial_\Gamma^\bullet\eta}\Bigr\|_{L^2(\Omega_t^\varepsilon)}\right)\,dt \\
    &\leq c_T\varepsilon\int_0^T\Bigl(\|\eta\|_{L^2(\Gamma_t)}+\|\partial_\Gamma^\bullet\eta\|_{L^2(\Gamma_t)}\Bigr)\,dt \to 0
  \end{align*}
  as $\varepsilon\to0$ by $|d|\leq c\varepsilon$ in $\overline{Q_T^\varepsilon}$, \eqref{E:DisMt_Bdd}, H\"{o}lder's inequality, \eqref{E:CE_Lq}, and \eqref{E:Ener_MTD}.
  Next, let
  \begin{align*}
    K_{4,1}^\varepsilon := \frac{1}{\varepsilon}\int_0^T\langle f^\varepsilon-\bar{f},d\bar{\eta}\rangle_{W^{1,p}(\Omega_t^\varepsilon)}\,dt, \quad K_{4,2}^\varepsilon := \frac{1}{\varepsilon}\int_0^T\langle\bar{f},d\bar{\eta}\rangle_{W^{1,p}(\Omega_t^\varepsilon)}\,dt
  \end{align*}
  so that $K_4^\varepsilon=K_{4,1}^\varepsilon+K_{4,2}^\varepsilon$.
  Noting that $\nabla d=\overline{\bm{\nu}}$ in $\overline{Q_T^\varepsilon}$ by \eqref{E:Fermi}, we have
  \begin{align*}
    \nabla(d\bar{\eta}) = \overline{\eta\bm{\nu}}+d\nabla\bar{\eta}, \quad |\bar{\bm{\nu}}| = 1, \quad |d| \leq c\varepsilon \leq c \quad\text{in}\quad \overline{Q_T^\varepsilon}.
  \end{align*}
  By these relations and \eqref{E:CE_Lq}, we see that
  \begin{align*}
    \|d\bar{\eta}\|_{L_{W^{1,p}}^p(Q_T^\varepsilon)} \leq c\|\bar{\eta}\|_{L_{W^{1,p}}^p(Q_T^\varepsilon)} \leq c\varepsilon^{1/p}\|\eta\|_{L_{W^{1,p}}^p(S_T)}.
  \end{align*}
  Thus, it follows from H\"{o}lder's inequality, $-1+1/p=-1/p'$, and \eqref{E:ExF_StCo} that
  \begin{align*}
    |K_{4,1}^\varepsilon| &\leq \frac{1}{\varepsilon}\|f^\varepsilon-\bar{f}\|_{L_{[W^{1,p}]^\ast}^{p'}(Q_T^\varepsilon)}\|d\bar{\eta}\|_{L_{W^{1,p}}^p(Q_T^\varepsilon)} \\
    &\leq c\varepsilon^{-1+1/p}\|f^\varepsilon-\bar{f}\|_{L_{[W^{1,p}]^\ast}^{p'}(Q_T^\varepsilon)}\|\eta\|_{L_{W^{1,p}}^p(S_T)} \to 0
  \end{align*}
  as $\varepsilon\to0$.
  Also, by \eqref{E:Def_CEFu}, H\"{o}lder's inequality, $g\in C^\infty(\overline{S_T})$, and \eqref{E:AvCE_W1q}, we have
  \begin{align*}
    |K_{4,2}^\varepsilon| = \left|\int_0^T\langle f,g\mathcal{M}_\varepsilon(d\bar{\eta})\rangle_{W^{1,p}(\Gamma_t)}\,dt\right| &\leq \|f\|_{L_{[W^{1,p}]^\ast}^{p'}(S_T)}\|g\mathcal{M}_\varepsilon(d\bar{\eta})\|_{L_{W^{1,p}}^p(S_T)} \\
    &\leq c\varepsilon\|f\|_{L_{[W^{1,p}]^\ast}^{p'}(S_T)}\|\eta\|_{L_{W^{1,p}}^p(S_T)} \to 0
  \end{align*}
  as $\varepsilon\to0$.
  Hence, $K_4^\varepsilon=K_{4,1}^\varepsilon+K_{4,2}^\varepsilon\to0$ as $\varepsilon\to0$.

  Let us consider $K_2^\varepsilon$ and $K_3^\varepsilon$.
  We define
  \begin{align*}
    K_{2,1}^\varepsilon &:= \frac{1}{\varepsilon}\int_0^T(|\nabla u^\varepsilon|^{p-2}\nabla u^\varepsilon,d\nabla\bar{\eta})_{L^2(\Omega_t^\varepsilon)}\,dt, \\
    K_{2,2}^\varepsilon &:= \frac{1}{\varepsilon}\int_0^T\Bigl(|\nabla u^\varepsilon|^{p-2}\nabla u^\varepsilon,\overline{\eta\bm{\nu}}\Bigr)_{L^2(\Omega_t^\varepsilon)}\,dt.
  \end{align*}
  Then, $K_2^\varepsilon=K_{2,1}^\varepsilon+K_{2,2}^\varepsilon$ since $\nabla d=\bar{\bm{\nu}}$ in $\overline{Q_T^\varepsilon}$.
  We see that
  \begin{align*}
    |K_{2,1}^\varepsilon| \leq c\|\nabla u^\varepsilon\|_{L_{L^p}^p(Q_T^\varepsilon)}^{p-1}\Bigl\|\overline{\nabla_\Gamma\eta}\Bigr\|_{L_{L^p}^p(Q_T^\varepsilon)} \leq c_T\varepsilon\|\nabla_\Gamma\eta\|_{L_{L^p}^p(S_T)} \to 0
  \end{align*}
  as $\varepsilon\to0$ by $|d|\leq c\varepsilon$ in $\overline{Q_T^\varepsilon}$, H\"{o}lder's inequality, \eqref{E:CE_Lq}, and \eqref{E:Ener_MTD}.
  Also,
  \begin{align*}
    K_{2,2}^\varepsilon = \int_0^T(g\mathbf{w}^\varepsilon,\eta\bm{\nu})_{L^2(\Gamma_t)}\,dt = \int_0^T(g\mathbf{w}^\varepsilon\cdot\bm{\nu},\eta)_{L^2(\Gamma_t)}\,dt
  \end{align*}
  by \eqref{E:Ave_Pair} and \eqref{E:Def_veps}.
  Thus, we use \eqref{E:veps_WeCo} to get
  \begin{align} \label{Pf_CwN:K2}
    \lim_{\varepsilon\to0}K_2^\varepsilon = \lim_{\varepsilon\to0}K_{2,2}^\varepsilon = \int_0^T(g\mathbf{w}\cdot\bm{\nu},\eta)_{L^2(\Gamma_t)}\,dt.
  \end{align}
  For $K_3^\varepsilon$, we again use $\nabla d=\bar{\bm{\nu}}$ in $\overline{Q_T^\varepsilon}$ and write $K_3^\varepsilon=K_{3,1}^\varepsilon+K_{3,2}^\varepsilon$ with
  \begin{align*}
    K_{3,1}^\varepsilon &:= \frac{1}{\varepsilon}\int_0^T\bigl(u^\varepsilon,\mathbf{v}^\varepsilon\cdot[d\nabla\bar{\eta}]\bigr)_{L^2(\Omega_t^\varepsilon)}\,dt+\frac{1}{\varepsilon}\int_0^T\Bigl(u^\varepsilon,[\mathbf{v}^\varepsilon-\bar{\mathbf{v}}_\Gamma]\cdot\bigl[\overline{\eta\bm{\nu}}\bigr]\Bigr)_{L^2(\Omega_t^\varepsilon)}\,dt, \\
    K_{3,2}^\varepsilon &:= \frac{1}{\varepsilon}\int_0^T\Bigl(u^\varepsilon,\bar{\mathbf{v}}_\Gamma\cdot\bigl[\overline{\eta\bm{\nu}}\bigr]\Bigr)_{L^2(\Omega_t^\varepsilon)}\,dt.
  \end{align*}
  We apply $|\bar{\bm{\nu}}|=1$ and $|d|\leq c\varepsilon$ in $\overline{Q_T^\varepsilon}$, \eqref{E:CEGr_NB}, \eqref{E:Vls_MTD}, \eqref{E:BdV_MTD} to $K_{3,1}^\varepsilon$ and then use H\"{o}lder's inequality, \eqref{E:CE_Lq}, and \eqref{E:Ener_MTD} to find that
  \begin{align*}
    |K_{3,1}^\varepsilon| &\leq c\int_0^T\|u^\varepsilon\|_{L^2(\Omega_t^\varepsilon)}\left(\|\bar{\eta}\|_{L^2(\Omega_t^\varepsilon)}+\Bigl\|\overline{\nabla_\Gamma\eta}\Bigr\|_{L^2(\Omega_t^\varepsilon)}\right)\,dt \\
    &\leq c_T\varepsilon\int_0^T\|\eta\|_{H^1(\Gamma_t)}\,dt \to 0
  \end{align*}
  as $\varepsilon\to0$.
  Also, by \eqref{E:Ave_Pair}, \eqref{E:Def_veps}, and $V_\Gamma=\mathbf{v}_\Gamma\cdot\bm{\nu}$, we have
  \begin{align*}
    K_{3,2}^\varepsilon = \int_0^T\bigl(gv^\varepsilon,\mathbf{v}_\Gamma\cdot[\eta\bm{\nu}]\bigr)_{L^2(\Gamma_t)}\,dt = \int_0^T(gV_\Gamma v^\varepsilon,\eta)_{L^2(\Gamma_t)}\,dt.
  \end{align*}
  Hence, it follows from \eqref{E:veps_WeCo} that
  \begin{align} \label{Pf_CwN:K3}
    \lim_{\varepsilon\to0}K_3^\varepsilon = \lim_{\varepsilon\to0}K_{3,2}^\varepsilon = \int_0^T(gV_\Gamma v,\eta)_{L^2(\Gamma_t)}\,dt.
  \end{align}
  Now, we send $\varepsilon\to0$ in $\sum_{j=1}^3K_j^\varepsilon=K_4^\varepsilon$ and use $K_1^\varepsilon,K_4^\varepsilon\to0$, \eqref{Pf_CwN:K2}, and \eqref{Pf_CwN:K3} to get
  \begin{align*}
    \int_0^T\bigl(g[\mathbf{w}\cdot\bm{\nu}+V_\Gamma v],\eta\bigr)_{L^2(\Gamma_t)}\,dt = 0 \quad\text{for all}\quad \eta\in C_{0,T}^\infty(\overline{S_T}).
  \end{align*}
  By this result and \eqref{E:G_Bdd}, we conclude that $\mathbf{w}\cdot\bm{\nu}+V_\Gamma v=0$ a.e. on $S_T$.
\end{proof}

Next, we determine the whole part of $\mathbf{w}$ by a monotonicity argument.
In this step, the strong convergence \eqref{E:veps_StCo} of $v^\varepsilon$ plays a crucial role.

\begin{proposition} \label{P:Chw_All}
  We have
  \begin{align} \label{E:Chw_All}
    \mathbf{w} = (|\nabla_\Gamma v|^2+\zeta^2)^{(p-2)/2}(\nabla_\Gamma v+\zeta\bm{\nu}) \quad\text{a.e. on}\quad S_T.
  \end{align}
\end{proposition}

\begin{proof}
  Let $\mathbf{b}_{v,\zeta}^\varepsilon:=\nabla_\Gamma v^\varepsilon+\zeta^\varepsilon\bm{\nu}$ and $\mathbf{b}_{v,\zeta}:=\nabla_\Gamma v+\zeta\bm{\nu}$ on $S_T$.
  Then, by \eqref{E:veps_WeCo},
  \begin{align} \label{Pf_CwA:b_WC}
    \lim_{\varepsilon\to0}\mathbf{b}_{v,\zeta}^\varepsilon = \mathbf{b}_{v,\zeta} \quad\text{weakly in $[L_{L^p}^p(S_T)]^n$}.
  \end{align}
  Also, since $\mathbf{a}=\mathbf{P}\mathbf{a}+(\mathbf{a}\cdot\bm{\nu})\bm{\nu}$ on $S_T$ for $\mathbf{a}\in\mathbb{R}^n$ and
  \begin{align*}
    \mathcal{M}_\varepsilon(\nabla u^\varepsilon)\cdot\bm{\nu} = \mathcal{M}_\varepsilon(\nabla u^\varepsilon\cdot\bar{\bm{\nu}}) = \mathcal{M}_\varepsilon(\partial_\nu u^\varepsilon) = \zeta^\varepsilon \quad\text{on}\quad S_T,
  \end{align*}
  we have $\mathcal{M}_\varepsilon(\nabla u^\varepsilon)=\mathbf{P}\mathcal{M}_\varepsilon(\nabla u^\varepsilon)+\zeta^\varepsilon\bm{\nu}$ on $S_T$.
  Thus,
  \begin{align*}
    |\mathcal{M}_\varepsilon(\nabla u^\varepsilon)-\mathbf{b}_{v,\zeta}^\varepsilon| = |\mathbf{P}\mathcal{M}_\varepsilon(\nabla u^\varepsilon)-\nabla_\Gamma v^\varepsilon| \leq c\varepsilon\mathcal{M}_\varepsilon(|u^\varepsilon|+|\nabla u^\varepsilon|) \quad\text{on}\quad S_T
  \end{align*}
  by \eqref{E:Def_veps} and \eqref{E:ATDr_Diff}, and it follows from \eqref{E:Ave_Lq} and then \eqref{E:Ener_MTD} that
  \begin{align*}
    \|\mathcal{M}_\varepsilon(\nabla u^\varepsilon)-\mathbf{b}_{v,\zeta}^\varepsilon\|_{L_{L^p}^p(S_T)} &\leq c\varepsilon\|\mathcal{M}_\varepsilon(|u^\varepsilon|+|\nabla u^\varepsilon|)\|_{L_{L^p}^p(S_T)} \\
    &\leq c\varepsilon^{1-1/p}\|u^\varepsilon\|_{L_{W^{1,p}}^p(Q_T^\varepsilon)} \\
    &\leq c_T\varepsilon \to 0
  \end{align*}
  as $\varepsilon\to0$.
  By this result and \eqref{Pf_CwA:b_WC}, we find that
  \begin{align} \label{Pf_CwA:AvGr}
    \lim_{\varepsilon\to0}\mathcal{M}_\varepsilon(\nabla u^\varepsilon) = \mathbf{b}_{v,\zeta} \quad\text{weakly in $[L_{L^p}^p(S_T)]^n$}.
  \end{align}
  Now, let $\mathbf{z}\in[L_{L^p}^p(S_T)]^n$.
  For each $t\in(0,T)$, we see by \eqref{E:pVec_Coer} that
  \begin{align*}
    (|\nabla u^\varepsilon|^{p-2}\nabla u^\varepsilon-|\bar{\mathbf{z}}|^{p-2}\bar{\mathbf{z}},\nabla u^\varepsilon-\bar{\mathbf{z}})_{L^2(\Omega_t^\varepsilon)} \geq 0.
  \end{align*}
  Let $\theta=\theta(t)$ be any nonnegative function in $C_c^\infty(0,T)$.
  We multiply the above inequality by $\theta(t)/\varepsilon\geq0$ and integrate it over $(0,T)$.
  Then, we have $\sum_{k=1}^4L_k^\varepsilon\geq0$, where
  \begin{align*}
    L_1^\varepsilon &:= \frac{1}{\varepsilon}\int_0^T\theta(|\nabla u^\varepsilon|^{p-2}\nabla u^\varepsilon,\nabla u^\varepsilon)_{L^2(\Omega_t^\varepsilon)}\,dt, \\
    L_2^\varepsilon &:= -\frac{1}{\varepsilon}\int_0^T\theta(|\nabla u^\varepsilon|^{p-2}\nabla u^\varepsilon,\bar{\mathbf{z}})_{L^2(\Omega_t^\varepsilon)}\,dt, \\
    L_3^\varepsilon &:= -\frac{1}{\varepsilon}\int_0^T\theta(|\bar{\mathbf{z}}|^{p-2}\bar{\mathbf{z}},\nabla u^\varepsilon)_{L^2(\Omega_t^\varepsilon)}\,dt, \\
    L_4^\varepsilon &:= \frac{1}{\varepsilon}\int_0^T\theta(|\bar{\mathbf{z}}|^{p-2}\bar{\mathbf{z}},\bar{\mathbf{z}})_{L^2(\Omega_t^\varepsilon)}\,dt.
  \end{align*}
  Let us derive the limit of each $L_k^\varepsilon$ as $\varepsilon\to0$.
  We first consider $L_2^\varepsilon$ and get
  \begin{align} \label{Pf_CwA:L2}
    \lim_{\varepsilon\to0}L_2^\varepsilon = \lim_{\varepsilon\to0}\left(-\int_0^T\theta(g\mathbf{w}^\varepsilon,\mathbf{z})_{L^2(\Gamma_t)}\,dt\right) = -\int_0^T\theta(g\mathbf{w},\mathbf{z})_{L^2(\Gamma_t)}\,dt
  \end{align}
  by \eqref{E:Ave_Pair}, \eqref{E:Def_veps}, and \eqref{E:veps_WeCo}.
  Next, we have
  \begin{align} \label{Pf_CwA:L3}
    \begin{aligned}
      \lim_{\varepsilon\to0}L_3^\varepsilon &= \lim_{\varepsilon\to0}\left(-\int_0^T\theta\bigl(g|\mathbf{z}|^{p-2}\mathbf{z},\mathcal{M}_\varepsilon(\nabla u^\varepsilon)\bigr)_{L^2(\Gamma_t)}\,dt\right) \\
      &= -\int_0^T\theta(g|\mathbf{z}|^{p-2}\mathbf{z},\mathbf{b}_{v,\zeta})_{L^2(\Gamma_t)}\,dt
    \end{aligned}
  \end{align}
  by \eqref{E:Ave_Pair}, \eqref{Pf_CwA:AvGr}, and $\theta g|\mathbf{z}|^{p-2}\mathbf{z}\in[L_{L^{p'}}^{p'}(S_T)]^n$.
  To compute $L_4^\varepsilon$, we see that
  \begin{align*}
    \|\mathcal{M}_\varepsilon\bar{\mathbf{z}}-\mathbf{z}\|_{L_{L^p}^p(S_T)} \leq c\varepsilon\|\mathbf{z}\|_{L_{L^p}^p(S_T)} \to 0 \quad\text{as}\quad \varepsilon\to0
  \end{align*}
  by \eqref{E:Ave_CE} and the boundedness of $g_0$, $g_1$, and $J_k$ on $\overline{S_T}$.
  Thus, by \eqref{E:Ave_Pair},
  \begin{align} \label{Pf_CwA:L4}
    \lim_{\varepsilon\to0}L_4^\varepsilon = \lim_{\varepsilon\to0}\int_0^T\theta(g|\mathbf{z}|^{p-2}\mathbf{z},\mathcal{M}_\varepsilon\bar{\mathbf{z}})_{L^2(\Gamma_t)}\,dt = \int_0^T\theta(g|\mathbf{z}|^{p-2}\mathbf{z},\mathbf{z})_{L^2(\Gamma_t)}\,dt.
  \end{align}
  The most difficult term is $L_1^\varepsilon$.
  We set $\psi=\theta u^\varepsilon$ in \eqref{E:WeFo_MTD} to get $L_1^\varepsilon=\sum_{k=1}^3L_{1,k}^\varepsilon$, where
  \begin{align*}
    L_{1,1}^\varepsilon &:= -\frac{1}{\varepsilon}\int_0^T\langle\partial_\varepsilon^\bullet u^\varepsilon,\theta u^\varepsilon\rangle_{W^{1,p}(\Omega_t^\varepsilon)}\,dt, \\
    L_{1,2}^\varepsilon &:= -\frac{1}{\varepsilon}\int_0^T\theta(u^\varepsilon,\mathbf{v}^\varepsilon\cdot\nabla u^\varepsilon+u^\varepsilon\,\mathrm{div}\,\mathbf{v}^\varepsilon)_{L^2(\Omega_t^\varepsilon)}\,dt, \\
    L_{1,3}^\varepsilon &:= \frac{1}{\varepsilon}\int_0^T\theta\langle f^\varepsilon,u^\varepsilon\rangle_{W^{1,p}(\Omega_t^\varepsilon)}\,dt.
  \end{align*}
  First, we consider $L_{1,2}^\varepsilon$.
  Let $\sigma_{\Gamma,g}=\mathrm{div}_\Gamma\mathbf{v}_\Gamma+[\partial_\Gamma^\bullet g/g]$ on $\overline{S_T}$ and
  \begin{align*}
    \xi_1^\varepsilon := \mathbf{v}^\varepsilon\cdot\nabla u^\varepsilon+u^\varepsilon\,\mathrm{div}\,\mathbf{v}^\varepsilon, \quad \xi_2^\varepsilon := \bar{\mathbf{v}}_\Gamma\cdot\nabla u^\varepsilon+u^\varepsilon\bar{\sigma}_{\Gamma,g} \quad\text{on}\quad Q_T^\varepsilon.
  \end{align*}
  We see by \eqref{E:Vls_MTD}, the smoothness of $g$ and $\mathbf{v}_\Gamma$ on $\overline{S_T}$, and \eqref{E:G_Bdd} that
  \begin{align} \label{Pf_CwA:xi_Bd}
    |\xi_1^\varepsilon-\xi_2^\varepsilon| \leq c\varepsilon(|u^\varepsilon|+|\nabla u^\varepsilon|), \quad |\xi_2^\varepsilon| \leq c(|u^\varepsilon|+|\nabla u^\varepsilon|) \quad\text{on}\quad Q_T^\varepsilon.
  \end{align}
  Also, since $\mathcal{M}_\varepsilon\xi_2^\varepsilon=\mathbf{v}_\Gamma\cdot\mathcal{M}_\varepsilon(\nabla u^\varepsilon)+v^\varepsilon\sigma_{\Gamma,g}$ on $S_T$, we have
  \begin{align} \label{Pf_CwA:Avxi2}
    \lim_{\varepsilon\to0}\mathcal{M}_\varepsilon\xi_2^\varepsilon = \mathbf{v}_\Gamma\cdot\mathbf{b}_{v,\zeta}+v\sigma_{\Gamma,g} \quad\text{weakly in $L_{L^p}^p(S_T)$}
  \end{align}
  by \eqref{E:veps_WeCo} and \eqref{Pf_CwA:AvGr}.
  We split $L_{1,2}^\varepsilon=\sum_{k=1}^3L_{1,2,k}^\varepsilon$ into
  \begin{align*}
    L_{1,2,1}^\varepsilon &:= -\frac{1}{\varepsilon}\int_0^T\theta(u^\varepsilon,\xi_1^\varepsilon)_{L^2(\Omega_t^\varepsilon)}+\frac{1}{\varepsilon}\int_0^T\theta(u^\varepsilon,\xi_2^\varepsilon)_{L^2(\Omega_t^\varepsilon)}\,dt, \\
    L_{1,2,2}^\varepsilon &:= -\frac{1}{\varepsilon}\int_0^T\theta(u^\varepsilon,\xi_2^\varepsilon)_{L^2(\Omega_t^\varepsilon)}+\frac{1}{\varepsilon}\int_0^T\theta(\bar{v}^\varepsilon,\xi_2^\varepsilon)_{L^2(\Omega_t^\varepsilon)}\,dt, \\
    L_{1,2,3}^\varepsilon &:= -\frac{1}{\varepsilon}\int_0^T\theta(\bar{v}^\varepsilon,\xi_2^\varepsilon)_{L^2(\Omega_t^\varepsilon)}\,dt.
  \end{align*}
  To $L_{1,2,1}^\varepsilon$, we apply \eqref{Pf_CwA:xi_Bd}, $\theta\in C_c^\infty(0,T)$, H\"{o}lder's inequality, and \eqref{E:L2H1_MTD}.
  Then,
  \begin{align*}
    |L_{1,2,1}^\varepsilon| \leq c\int_0^T\|u^\varepsilon\|_{L^2(\Omega_t^\varepsilon)}\|u^\varepsilon\|_{H^1(\Omega_t^\varepsilon)}\,dt \leq c\int_0^T\|u^\varepsilon\|_{H^1(\Omega_t^\varepsilon)}^2\,dt \leq c_T\varepsilon.
  \end{align*}
  Also, since $v^\varepsilon=\mathcal{M}_\varepsilon u^\varepsilon$, we see by H\"{o}lder's inequality, \eqref{E:AvDf_Lq}, \eqref{Pf_CwA:xi_Bd}, and \eqref{E:L2H1_MTD} that
  \begin{align*}
    |L_{1,2,2}^\varepsilon| \leq \frac{c}{\varepsilon}\int_0^T\|u^\varepsilon-\bar{v}^\varepsilon\|_{L^2(\Omega_t^\varepsilon)}\|\xi_2^\varepsilon\|_{L^2(\Omega_t^\varepsilon)}\,dt \leq c\int_0^T\|u^\varepsilon\|_{H^1(\Omega_t^\varepsilon)}^2\,dt \leq c_T\varepsilon.
  \end{align*}
  Thus, $L_{1,2,k}^\varepsilon\to0$ as $\varepsilon\to0$ for $k=1,2$.
  Moreover, since
  \begin{align*}
    L_{1,2,3}^\varepsilon = -\int_0^T\theta(gv^\varepsilon,\mathcal{M}_\varepsilon\xi_2^\varepsilon)_{L^2(\Gamma_t)}\,dt
  \end{align*}
  by \eqref{E:Ave_Pair}, we apply the strong convergence \eqref{E:veps_StCo} of $v^\varepsilon$ and the weak convergence \eqref{Pf_CwA:Avxi2} of $\mathcal{M}_\varepsilon\xi_2^\varepsilon$ to the right-hand side to find that
  \begin{align} \label{Pf_CwA:L1_02}
    \lim_{\varepsilon\to0}L_{1,2}^\varepsilon = \lim_{\varepsilon\to0}L_{1,2,3}^\varepsilon = -\int_0^T\theta(gv,\mathbf{v}_\Gamma\cdot\mathbf{b}_{v,\zeta}+v\sigma_{\Gamma,g})_{L^2(\Gamma_t)}\,dt.
  \end{align}
  Next, we apply \eqref{E:Tr_Mult} to $L_{1,1}^\varepsilon$.
  Then, since $\theta=\theta(t)$ and $\theta(0)=\theta(T)=0$, we have
  \begin{align*}
    L_{1,1}^\varepsilon = \frac{1}{2\varepsilon}\int_0^T\theta'(u^\varepsilon,u^\varepsilon)_{L^2(\Omega_t^\varepsilon)}\,dt+\frac{1}{2\varepsilon}\int_0^T\theta(u^\varepsilon,u^\varepsilon\,\mathrm{div}\,\mathbf{v}^\varepsilon)_{L^2(\Omega_t^\varepsilon)}\,dt
  \end{align*}
  with $\theta'=d\theta/dt$.
  Thus, as in the case of $L_{1,2}^\varepsilon$, we can get
  \begin{align} \label{Pf_CwA:L1_01}
    \lim_{\varepsilon\to0}L_{1,1}^\varepsilon = \frac{1}{2}\int_0^T\theta'(gv,v)_{L^2(\Gamma_t)}\,dt+\frac{1}{2}\int_0^T\theta(gv,v\sigma_{\Gamma,g})_{L^2(\Gamma_t)}\,dt
  \end{align}
  by using \eqref{E:Vls_MTD}, \eqref{E:L2H1_MTD}, \eqref{E:Ave_Pair}, \eqref{E:AvDf_Lq}, and the weak and strong convergence \eqref{E:veps_WeCo} and \eqref{E:veps_StCo} of $v^\varepsilon$.
  For $L_{1,3}^\varepsilon$, we split it into the sum of
  \begin{align*}
    L_{1,3,1}^\varepsilon := \frac{1}{\varepsilon}\int_0^T\theta\langle f^\varepsilon-\bar{f},u^\varepsilon\rangle_{W^{1,p}(\Omega_t^\varepsilon)}\,dt, \quad L_{1,3,2}^\varepsilon := \frac{1}{\varepsilon}\int_0^T\theta\langle\bar{f},u^\varepsilon\rangle_{W^{1,p}(\Omega_t^\varepsilon)}\,dt.
  \end{align*}
  By H\"{o}lder's inequality, \eqref{E:Ener_MTD}, $-1+1/p=-1/p'$, and \eqref{E:ExF_StCo}, we see that
  \begin{align*}
    |L_{1,3,1}^\varepsilon| &\leq \frac{c}{\varepsilon}\|f^\varepsilon-\bar{f}\|_{L_{[W^{1,p}]^\ast}^{p'}(Q_T^\varepsilon)}\|u^\varepsilon\|_{L_{W^{1,p}}^p(Q_T^\varepsilon)} \\
    &\leq c\varepsilon^{-1+1/p}\|f^\varepsilon-\bar{f}\|_{L_{[W^{1,p}]^\ast}^{p'}(Q_T^\varepsilon)} \to 0
  \end{align*}
  as $\varepsilon\to0$.
  Also, since $\bar{f}$ is given by \eqref{E:Def_CEFu} and $v^\varepsilon=\mathcal{M}_\varepsilon u^\varepsilon$, we have
  \begin{align} \label{Pf_CwA:L1_03}
    \lim_{\varepsilon\to0}L_{1,3}^\varepsilon = \lim_{\varepsilon\to0}L_{1,3,2}^\varepsilon = \lim_{\varepsilon\to0}\int_0^T\theta\langle f,gv^\varepsilon\rangle_{W^{1,p}(\Gamma_t)}\,dt = \int_0^T\theta\langle f,gv\rangle_{W^{1,p}(\Gamma_t)}\,dt
  \end{align}
  by \eqref{E:veps_WeCo}.
  Thus, by $L_1^\varepsilon=\sum_{k=1}^3L_{1,k}^\varepsilon$ and \eqref{Pf_CwA:L1_02}--\eqref{Pf_CwA:L1_03}, we find that
  \begin{align} \label{Pf_CwA:L1_all}
    \begin{aligned}
      \lim_{\varepsilon\to0}L_1^\varepsilon &= \frac{1}{2}\int_0^T\theta'(gv,v)_{L^2(\Gamma_t)}\,dt-\int_0^T\theta(gv,\mathbf{v}_\Gamma\cdot\mathbf{b}_{v,\zeta})_{L^2(\Gamma_t)}\,dt \\
      &\qquad -\frac{1}{2}\int_0^T\theta(gv,v\sigma_{\Gamma,g})_{L^2(\Gamma_t)}\,dt+\int_0^T\theta\langle f,gv\rangle_{W^{1,p}(\Gamma_t)}\,dt.
    \end{aligned}
  \end{align}
  On the other hand, we apply \eqref{E:WeMa_Mult} (with $\varepsilon=0$) to \eqref{E:LiWF_uw} to get
  \begin{align*}
    &\int_0^T\langle\partial_\Gamma^\bullet v,g\eta\rangle_{W^{1,p}(\Gamma_t)}\,dt+\int_0^T(v,\eta\partial_\Gamma^\bullet g)_{L^2(\Gamma_t)}\,dt+\int_0^T(g\mathbf{w},\nabla_\Gamma\eta)_{L^2(\Gamma_t)}\,dt \\
    &\qquad +\int_0^T(gv,\mathbf{v}_\Gamma\cdot\nabla_\Gamma\eta+\eta\,\mathrm{div}_\Gamma\mathbf{v}_\Gamma)_{L^2(\Gamma_t)}\,dt = \int_0^T\langle f,g\eta\rangle_{W^{1,p}(\Gamma_t)}\,dt
  \end{align*}
  for $\eta\in L_{W^{1,p}}^p(S_T)$.
  In this equality, we set $\eta=\theta v$, use
  \begin{align*}
    \int_0^T\langle\partial_\Gamma^\bullet v,g\theta v\rangle_{W^{1,p}(\Gamma_t)}\,dt = -\frac{1}{2}\int_0^T\Bigl\{\bigl(v,v\partial_\Gamma^\bullet(g\theta)\bigr)_{L^2(\Gamma_t)}+(v,g\theta v\,\mathrm{div}_\Gamma\mathbf{v}_\Gamma)_{L^2(\Gamma_t)}\Bigr\}\,dt
  \end{align*}
  by \eqref{E:Tr_Mult} (with $\varepsilon=0$ and $\chi=g\theta$) and $\theta(0)=\theta(T)=0$, and apply
  \begin{align*}
    \partial_\Gamma^\bullet(g\theta) = \theta\partial_\Gamma^\bullet g+g\partial_\Gamma^\bullet\theta = \theta\partial_\Gamma^\bullet g+g\theta' \quad\text{on}\quad \overline{S_T}
  \end{align*}
  by $\theta=\theta(t)$.
  Then, we find that
  \begin{align*}
    &-\frac{1}{2}\int_0^T\theta'(gv,v)_{L^2(\Gamma_t)}\,dt+\int_0^T\theta(g\mathbf{w},\nabla_\Gamma v)_{L^2(\Gamma_t)}\,dt+\int_0^T\theta(gv,\mathbf{v}_\Gamma\cdot\nabla_\Gamma v)_{L^2(\Gamma_t)}\,dt \\
    &\qquad\qquad +\frac{1}{2}\int_0^T\theta\Bigl\{(v,v\partial_\Gamma^\bullet g)_{L^2(\Gamma_t)}+(gv,v\,\mathrm{div}_\Gamma\mathbf{v}_\Gamma)_{L^2(\Gamma_t)}\Bigr\}\,dt = \int_0^T\theta\langle f,gv\rangle_{W^{1,p}(\Gamma_t)}\,dt.
  \end{align*}
  We substitute this expression for the last term of \eqref{Pf_CwA:L1_all} and use
  \begin{align*}
    (gv,v\sigma_{\Gamma,g})_{L^2(\Gamma_t)} = (v,v\partial_\Gamma^\bullet g)_{L^2(\Gamma_t)}+(gv,v\,\mathrm{div}_\Gamma\mathbf{v}_\Gamma)_{L^2(\Gamma_t)}
  \end{align*}
  by $\sigma_{\Gamma,g}=\mathrm{div}_\Gamma\mathbf{v}_\Gamma+[\partial_\Gamma^\bullet g/g]$ on $\overline{S_T}$.
  Then, we obtain
  \begin{align} \label{Pf_CwA:L1_ano}
    \lim_{\varepsilon\to0}L_1^\varepsilon = \int_0^T\theta(g\mathbf{w},\nabla_\Gamma v)_{L^2(\Gamma_t)}\,dt+\int_0^T\theta\bigl(gv,\mathbf{v}_\Gamma\cdot[\nabla_\Gamma v-\mathbf{b}_{v,\zeta}]\bigr)_{L^2(\Gamma_t)}\,dt.
  \end{align}
  Now, we send $\varepsilon\to0$ in $\sum_{k=1}^4L_k^\varepsilon\geq0$ and use \eqref{Pf_CwA:L2}--\eqref{Pf_CwA:L4} and \eqref{Pf_CwA:L1_ano} to get
  \begin{align*}
   &\int_0^T\theta(g\mathbf{w},\nabla_\Gamma v-\mathbf{z})_{L^2(\Gamma_t)}\,dt+\int_0^T\theta(g|\mathbf{z}|^{p-2}\mathbf{z},\mathbf{z}-\mathbf{b}_{v,\zeta})_{L^2(\Gamma_t)}\,dt \\
   &\qquad +\int_0^T\theta\bigl(gv,\mathbf{v}_\Gamma\cdot[\nabla_\Gamma v-\mathbf{b}_{v,\zeta}]\bigr)_{L^2(\Gamma_t)}\,dt \geq 0.
  \end{align*}
  Moreover, since $\nabla_\Gamma v=\mathbf{b}_{v,\zeta}-\zeta\bm{\nu}$ and $V_\Gamma=\mathbf{v}_\Gamma\cdot\bm{\nu}$, the above inequality reads
  \begin{align*}
    \int_0^T\theta\bigl(g[\mathbf{w}-|\mathbf{z}|^{p-2}\mathbf{z}],\mathbf{b}_{v,\zeta}-\mathbf{z}\bigr)_{L^2(\Gamma_t)}\,dt-\int_0^T\theta\bigl(g[\mathbf{w}\cdot\bm{\nu}+V_\Gamma v],\zeta\bigr)_{L^2(\Gamma_t)}dt \geq 0,
  \end{align*}
  and the second integral vanishes by Proposition \ref{P:Chw_Nor}.
  Therefore,
  \begin{align} \label{Pf_CwA:bz}
    \int_0^T\theta\bigl(g[\mathbf{w}-|\mathbf{z}|^{p-2}\mathbf{z}],\mathbf{b}_{v,\zeta}-\mathbf{z}\bigr)_{L^2(\Gamma_t)}\,dt \geq 0
  \end{align}
  for all $\mathbf{z}\in[L_{L^p}^p(S_T)]^n$ and all nonnegative $\theta\in C_c^\infty(0,T)$.
  Now, we set
  \begin{align*}
    \mathbf{z} = \mathbf{z}_h := \mathbf{b}_{v,\zeta}-h\bm{\eta} \in [L_{L^p}^p(S_T)]^n \quad\text{for any}\quad \bm{\eta}\in[C_{0,T}^\infty(\overline{S_T})]^n, \quad h\in(0,1)
  \end{align*}
  in \eqref{Pf_CwA:bz} and divide both sides by $h>0$.
  Then, we have
  \begin{align} \label{Pf_CwA:Ineq}
    \int_0^T\theta\bigl(g[\mathbf{w}-|\mathbf{z}_h|^{p-2}\mathbf{z}_h],\bm{\eta}\bigr)_{L^2(\Gamma_t)}\,dt \geq 0.
  \end{align}
  Moreover, we see by \eqref{E:pVec_Lip}, $|\mathbf{z}_h|\leq|\mathbf{b}_{v,\zeta}|+|\bm{\eta}|$, and Young's inequality that
  \begin{align*}
    \bigl|\,|\mathbf{z}_h|^{p-2}\mathbf{z}_h-|\mathbf{b}_{v,\zeta}|^{p-2}\mathbf{b}_{v,\zeta}\,\bigr| &\leq ch\Bigl(|\mathbf{b}_{v,\zeta}|^{p-2}+|\bm{\eta}|^{p-2}\Bigr)|\bm{\eta}| \\
    &\leq ch\Bigl(|\mathbf{b}_{v,\zeta}|^{p-1}+|\bm{\eta}|^{p-1}\Bigr)
  \end{align*}
  on $S_T$.
  From this inequality and $1/p+1/p'=1$, it follows that
  \begin{align*}
    \bigl\|\,|\mathbf{z}_h|^{p-2}\mathbf{z}_h-|\mathbf{b}_{v,\zeta}|^{p-2}\mathbf{b}_{v,\zeta}\,\bigr\|_{L_{L^{p'}}^{p'}(S_T)} \leq ch\Bigl(\|\mathbf{b}_{v,\zeta}\|_{L_{L^p}^p(S_T)}^{p-1}+\|\bm{\eta}\|_{L_{L^p}^p(S_T)}^{p-1}\Bigr) \to 0
  \end{align*}
  as $h\to0$.
  Thus, letting $h\to0$ in \eqref{Pf_CwA:Ineq}, we find that
  \begin{align*}
    \int_0^T\theta\bigl(g[\mathbf{w}-|\mathbf{b}_{v,\zeta}|^{p-2}\mathbf{b}_{v,\zeta}],\bm{\eta}\bigr)_{L^2(\Gamma_t)}\,dt \geq 0,
  \end{align*}
  and we also have the opposite inequality by replacing $\bm{\eta}$ with $-\bm{\eta}$.
  Therefore,
  \begin{align} \label{Pf_CwA:theta}
    \int_0^T\theta\bigl(g[\mathbf{w}-|\mathbf{b}_{v,\zeta}|^{p-2}\mathbf{b}_{v,\zeta}],\bm{\eta}\bigr)_{L^2(\Gamma_t)}\,dt = 0.
  \end{align}
  Since $\bm{\eta}\in[C_{0,T}^\infty(\overline{S_T})]^n$, we can take a small $t_0>0$ such that
  \begin{align*}
    \bm{\eta}(\cdot,t) = \mathbf{0}_n \quad\text{on}\quad \Gamma_t \quad\text{for all}\quad t\in[0,t_0]\cup[T-t_0,T],
  \end{align*}
  where $\mathbf{0}_n$ is the zero vector in $\mathbb{R}^n$.
  Then, taking a nonnegative function $\theta\in C_c^\infty(0,T)$ such that $\theta(t)=1$ for all $t\in[t_0,T-t_0]$ in \eqref{Pf_CwA:theta}, we find that
  \begin{align*}
    \int_0^T\bigl(g[\mathbf{w}-|\mathbf{b}_{v,\zeta}|^{p-2}\mathbf{b}_{v,\zeta}],\bm{\eta}\bigr)_{L^2(\Gamma_t)}\,dt = 0 \quad\text{for all}\quad \bm{\eta}\in[C_{0,T}^\infty(\overline{S_T})]^n.
  \end{align*}
  By this equality and \eqref{E:G_Bdd}, we obtain $\mathbf{w}=|\mathbf{b}_{v,\zeta}|^{p-2}\mathbf{b}_{v,\zeta}$ a.e. on $S_T$.
  Moreover,
  \begin{align*}
    \mathbf{b}_{v,\zeta} = \nabla_\Gamma v+\zeta\bm{\nu}, \quad |\mathbf{b}_{v,\zeta}| = (|\nabla_\Gamma v|^2+\zeta^2)^{1/2} \quad\text{on}\quad S_T
  \end{align*}
  by $\nabla_\Gamma v\cdot\bm{\nu}=0$ and $|\bm{\nu}|=1$ on $S_T$.
  Therefore, \eqref{E:Chw_All} follows.
\end{proof}

We can also characterize $\zeta$ in terms of $v$.

\begin{proposition} \label{P:Ch_zeta}
  We have
  \begin{align} \label{E:Ch_zeta}
    (|\nabla_\Gamma v|^2+\zeta^2)^{(p-2)/2}\zeta+V_\Gamma v = 0 \quad\text{a.e. on}\quad S_T.
  \end{align}
\end{proposition}

\begin{proof}
  It follows from \eqref{E:Chw_All}, $\nabla_\Gamma v\cdot\bm{\nu}=0$, and $|\bm{\nu}|=1$ on $S_T$ that
  \begin{align*}
    \mathbf{w}\cdot\bm{\nu} = (|\nabla_\Gamma v|^2+\zeta^2)^{(p-2)/2}\zeta \quad\text{a.e. on}\quad S_T.
  \end{align*}
  By this relation and Proposition \ref{P:Chw_Nor}, we obtain \eqref{E:Ch_zeta}.
\end{proof}

\begin{remark} \label{R:Ch_zeta}
  When $v$ is given, there exists a unique function $\zeta$ on $S_T$ that satisfies \eqref{E:Ch_zeta}.
  Indeed, for any given constants $a\geq0$ and $b\in\mathbb{R}$, let
  \begin{align*}
    \varphi(s) := (a+s^2)^{(p-2)/2}s+b, \quad s\in\mathbb{R}.
  \end{align*}
  Then, $\varphi$ is strictly increasing and $\varphi(s)\to\pm\infty$ as $s\to\pm\infty$ (double-sign corresponds) by $p>2$.
  Thus, the equation $\varphi(s)=0$ has a unique solution.
\end{remark}

\subsection{Weak solution to the limit problem} \label{SS:TF_Lim}
Suppose for a while that all functions in \eqref{E:LiWF_uw}, \eqref{E:Chw_All}, and \eqref{E:Ch_zeta} are smooth.
For a vector field $\mathbf{z}\colon S_T\to\mathbb{R}^n$, let $\mathbf{z}^\tau:=\mathbf{P}\mathbf{z}$ be its tangential component.
Then, since $\nabla_\Gamma\eta$ and $\mathbf{z}^\tau$ are tangential on $\Gamma_t$, we have
\begin{align*}
  \int_{\Gamma_t}\mathbf{z}\cdot\nabla_\Gamma\eta\,d\mathcal{H}^{n-1} = \int_{\Gamma_t}\mathbf{z}^\tau\cdot\nabla_\Gamma\eta\,d\mathcal{H}^{n-1} = -\int_{\Gamma_t}\eta\,\mathrm{div}_\Gamma\mathbf{z}^{\tau}\,d\mathcal{H}^{n-1}
\end{align*}
by \eqref{E:IbP_Sur}.
Using this with $\mathbf{z}=g\mathbf{w},gv\mathbf{v}_\Gamma$, we write \eqref{E:LiWF_uw} as
\begin{align*}
  \int_0^T\Bigl(\partial_\Gamma^\bullet(gv)+gv\,\mathrm{div}_\Gamma\mathbf{v}_\Gamma-\mathrm{div}_\Gamma\bigl[g(\mathbf{w}^\tau+v\mathbf{v}_\Gamma^\tau)\bigr],\eta\Bigr)_{L^2(\Gamma_t)}\,dt = \int_0^T(gf,\eta)_{L^2(\Gamma_t)}\,dt.
\end{align*}
Moreover, $\mathbf{w}^\tau$ and $\zeta$ are determined by $v$ through \eqref{E:Chw_All} and \eqref{E:Ch_zeta}, i.e.
\begin{align*}
  \mathbf{w}^\tau = (|\nabla_\Gamma v|^2+\zeta^2)^{(p-2)/2}\nabla_\Gamma v, \quad (|\nabla_\Gamma v|^2+\zeta^2)^{(p-2)/2}\zeta+V_\Gamma v = 0 \quad\text{on}\quad S_T.
\end{align*}
Hence, the limit problem \eqref{E:pLap_Lim} is obtained from \eqref{E:LiWF_uw}, \eqref{E:Chw_All}, and \eqref{E:Ch_zeta}.
This observation justifies to define a weak solution to \eqref{E:pLap_Lim} as follows.

\begin{definition} \label{D:WS_Lim}
  For given $v_0\in L^2(\Gamma_0)$ and $f\in L_{[W^{1,p}]^\ast}^{p'}(S_T)$, we say that a pair $(v,\zeta)$ is a weak solution to \eqref{E:pLap_Lim} if $(v,\zeta)\in\mathbb{W}^{p,p'}(S_T)\times L_{L^p}^p(S_T)$ and it satisfies \eqref{E:Ch_zeta} and
  \begin{align} \label{E:WeFo_Lim}
    \begin{aligned}
      &\int_0^T\langle\partial_\Gamma^\bullet(gv),\eta\rangle_{W^{1,p}(\Gamma_t)}\,dt+\int_0^T\bigl(g(|\nabla_\Gamma v|^2+\zeta^2)^{(p-2)/2}\nabla_\Gamma v,\nabla_\Gamma\eta\bigr)_{L^2(\Gamma_t)}\,dt \\
      &\qquad +\int_0^T(gv,\mathbf{v}_\Gamma\cdot\nabla_\Gamma\eta+\eta\,\mathrm{div}_\Gamma\mathbf{v}_\Gamma)_{L^2(\Gamma_t)}\,dt = \int_0^T\langle f,g\eta\rangle_{W^{1,p}(\Gamma_t)}\,dt
    \end{aligned}
  \end{align}
  for all $\eta\in L_{W^{1,p}}^p(S_T)$ and the initial condition $v|_{t=0}=v_0$ in $L^2(\Gamma_0)$.
\end{definition}

Note that the initial condition makes sense since $\mathbb{W}^{p,p'}(S_T)\subset C_{L^2}(S_T)$ by Lemma \ref{L:Trans}.
The equations \eqref{E:Ch_zeta} and \eqref{E:WeFo_Lim} can be combined into one weak form, which is crucial for the proof of the uniqueness of a weak solution to \eqref{E:pLap_Lim}.

\begin{lemma} \label{L:WFL_Equi}
  Let $f\in L_{[W^{1,p}]^\ast}^{p'}(S_T)$.
  A pair $(v,\zeta)\in\mathbb{W}^{p,p'}(S_T)\times L_{L^p}^p(S_T)$ satisfies \eqref{E:Ch_zeta} and \eqref{E:WeFo_Lim} if and only if
  \begin{align} \label{E:WFL_Equi}
    \begin{aligned}
      &\int_0^T\langle\partial_\Gamma^\bullet(gv),\eta\rangle_{W^{1,p}(\Gamma_t)}\,dt+\int_0^T(g|\mathbf{b}_{v,\zeta}|^{p-2}\mathbf{b}_{v,\zeta},\nabla_\Gamma\eta+\xi\bm{\nu})_{L^2(\Gamma_t)}\,dt \\
      &\qquad +\int_0^T(gv,\mathbf{v}_\Gamma\cdot[\nabla_\Gamma\eta+\xi\bm{\nu}]+\eta\,\mathrm{div}_\Gamma\mathbf{v}_\Gamma)_{L^2(\Gamma_t)}\,dt = \int_0^T\langle f,g\eta\rangle_{W^{1,p}(\Gamma_t)}\,dt
    \end{aligned}
  \end{align}
  for all $(\eta,\xi)\in L_{W^{1,p}}^p(S_T)\times L_{L^p}^p(S_T)$, where $\mathbf{b}_{v,\zeta}=\nabla_\Gamma v+\zeta\bm{\nu}$ on $S_T$.
\end{lemma}

\begin{proof}
  Suppose first that $(v,\zeta)$ satisfies \eqref{E:Ch_zeta} and \eqref{E:WeFo_Lim}.
  Then, since
  \begin{align*}
    (|\nabla_\Gamma v|+\zeta^2)^{1/2} = |\mathbf{b}_{v,\zeta}|, \quad |\bm{\nu}| = 1, \quad V_\Gamma = \mathbf{v}_\Gamma\cdot\bm{\nu} \quad\text{on}\quad S_T,
  \end{align*}
  it follows from \eqref{E:Ch_zeta} that
  \begin{align*}
    \int_0^T(g|\mathbf{b}_{v,\zeta}|^{p-2}\zeta\bm{\nu},\xi\bm{\nu})_{L^2(\Gamma_t)}\,dt+\int_0^T(gv\mathbf{v}_\Gamma,\xi\bm{\nu})_{L^2(\Gamma_t)}\,dt = 0
  \end{align*}
  for all $\xi\in L_{L^p}^p(S_T)$.
  Adding this to \eqref{E:WeFo_Lim}, we get \eqref{E:WFL_Equi}.
  Conversely, if \eqref{E:WFL_Equi} is satisfied, then we have \eqref{E:WeFo_Lim} by setting $\xi=0$, and also get \eqref{E:Ch_zeta} by setting $\eta=0$.
\end{proof}

The next result is also important for the uniqueness of a weak solution.

\begin{lemma} \label{L:WFL_Con}
  For $i=1,2$, let $(v_i,\zeta_i)\in L_{W^{1,p}}^p(S_T)\times L_{L^p}^p(S_T)$ satisfy \eqref{E:Ch_zeta}.
  Then,
  \begin{align} \label{E:WFL_Con}
    \nabla_\Gamma v_1 = \nabla_\Gamma v_2, \quad \zeta_1 = \zeta_2 \quad\text{a.e. on}\quad \{(y,t)\in S_T \mid v_1(y,t)=v_2(y,t)\}.
  \end{align}
\end{lemma}

\begin{proof}
 As in the case of flat domains (see \cite[Lemma 7.7]{GilTru01}), we can find that
 \begin{align*}
   \nabla_\Gamma v_1 = \nabla_\Gamma v_2 \quad\text{a.e. on}\quad \{(y,t)\in S_T \mid v_1(y,t)=v_2(y,t)\}.
 \end{align*}
 Also, the equation $(a+s^2)^{(p-2)/2}s+b=0$ has a unique solution $s\in\mathbb{R}$ for any given $a\geq0$ and $b\in\mathbb{R}$ as explained in Remark \ref{R:Ch_zeta}.
  This fact implies that
  \begin{align*}
    \zeta_1 = \zeta_2 \quad\text{on}\quad \{(y,t)\in S_T \mid v_1(y,t)=v_2(y,t), \, \nabla_\Gamma v_1(y,t) = \nabla_\Gamma v_2(y,t)\},
  \end{align*}
  since $(v_1,\zeta_1)$ and $(v_2,\zeta_2)$ satisfy \eqref{E:Ch_zeta}.
  Thus, we get \eqref{E:WFL_Con}.
\end{proof}

Let us prove the uniqueness of a weak solution to \eqref{E:pLap_Lim}.

\begin{theorem} \label{T:Lim_Uni}
  For each $v_0\in L^2(\Gamma_0)$ and $f\in L_{[W^{1,p}]^\ast}^{p'}(S_T)$, there exists at most one weak solution to \eqref{E:pLap_Lim}.
\end{theorem}

\begin{proof}
  We follow the idea of the proof of \cite[Theorem 2.9]{CaNoOr17} (see also \cite[Proposition 3.4]{Miu25pre_pLap}), but some modifications are required in order to deal with the functions $g$ and $\zeta$.

  Let $\rho$ be a standard mollifier on $\mathbb{R}$, i.e., for $z\in\mathbb{R}$,
  \begin{align*}
    \rho(z) :=
    \begin{cases}
      c_\rho\exp\left(\dfrac{1}{z^2-1}\right) &(|z|<1), \\
      0 &(|z|\geq1),
    \end{cases}
    \quad c_\rho := \left\{\int_{-1}^1\exp\left(\frac{1}{z^2-1}\right)\,dz\right\}^{-1}.
  \end{align*}
  For $\gamma\in(0,1)$, let $\Lambda_\gamma(z)$ be the mollification of $|z|$ given by
  \begin{align*}
    \Lambda_\gamma(z) := \frac{1}{\gamma}\int_{-\gamma}^\gamma\rho\left(\frac{s}{\gamma}\right)|z-s|\,ds = \int_{-1}^1\rho(s)|z-\gamma s|\,ds, \quad z\in\mathbb{R}.
  \end{align*}
  We see that $\Lambda_\gamma\in C^\infty(\mathbb{R})$.
  When $|z|\leq\gamma$, we have
  \begin{align*}
    \Lambda_\gamma(z) &= \int_{-1}^{z/\gamma}\rho(s)(z-\gamma s)\,ds+\int_{z/\gamma}^1\rho(s)(\gamma s-z)\,ds, \\
    \Lambda_\gamma'(z) &= \int_{-1}^{z/\gamma}\rho(s)\,ds-\int_{z/\gamma}^1\rho(s)\,ds, \quad \Lambda_\gamma''(z) = \frac{2}{\gamma}\rho\left(\frac{z}{\gamma}\right).
  \end{align*}
  Also, $\Lambda_\gamma(z)=|z|$ for $|z|\geq\gamma$.
  Thus, for all $z\in\mathbb{R}$, we easily get
  \begin{align} \label{Pf_LUn:abs}
    \begin{aligned}
      0 &\leq \Lambda_\gamma(z) \leq 1+|z|, \quad \lim_{\gamma\to0}\Lambda_\gamma(z) = |z|, \\
      -1 &\leq \Lambda_\gamma'(z) \leq 1, \quad \lim_{\gamma\to0}\Lambda_\gamma'(z) = \mathrm{sgn}\,z =
      \begin{cases}
        z/|z| &(z\neq0), \\
        0 &(z=0),
      \end{cases} \\
      0 &\leq \Lambda_\gamma''(z) \leq \frac{2c_\rho}{\gamma e} \quad (|z| \leq \gamma), \quad \Lambda_\gamma''(z) = 0 \quad (|z| \geq \gamma).
    \end{aligned}
  \end{align}
  Let $(v_1,\zeta_1)$ and $(v_2,\zeta_2)$ be weak solutions to \eqref{E:pLap_Lim} with same data $v_0$ and $f$, and let
  \begin{align*}
    v_{\mathrm{d}} := v_1-v_2 \in \mathbb{W}^{p,p'}(S_T) \subset C_{L^2}(S_T), \quad \zeta_{\mathrm{d}} := \zeta_1-\zeta_2 \in L_{L^p}^p(S_T).
  \end{align*}
  Also, let $\mathbf{b}_i:=\nabla_\Gamma v_i+\zeta_i\bm{\nu}$ on $S_T$ for $i=1,2$ and
  \begin{align} \label{Pf_LUn:bd}
    \mathbf{b}_{\mathrm{d}} := \mathbf{b}_1-\mathbf{b}_2 = \nabla_\Gamma v_{\mathrm{d}}+\zeta_{\mathrm{d}}\bm{\nu}, \quad \mathbf{z}_{\mathrm{d}} := |\mathbf{b}_1|^{p-2}\mathbf{b}_1-|\mathbf{b}_2|^{p-2}\mathbf{b}_2 \quad\text{on}\quad S_T.
  \end{align}
  Since $(v_i,\zeta_i)$ satisfies \eqref{E:WFL_Equi} with $\mathbf{b}_{v,\zeta}$ replaced by $\mathbf{b}_i$ for $i=1,2$ by Lemma \ref{L:WFL_Equi}, we take the difference of those weak forms with $T$ replaced by $t$ to get
  \begin{align} \label{Pf_LUn:WF_Df}
    \begin{aligned}
      &\int_0^t\langle\partial_\Gamma^\bullet(gv_{\mathrm{d}}),\eta\rangle_{W^{1,p}(\Gamma_s)}\,ds+\int_0^t(g\mathbf{z}_{\mathrm{d}},\nabla_\Gamma\eta+\xi\bm{\nu})_{L^2(\Gamma_s)}\,ds \\
      &\qquad +\int_0^t(gv_{\mathrm{d}},\mathbf{v}_\Gamma\cdot[\nabla_\Gamma\eta+\xi\bm{\nu}]+\eta\,\mathrm{div}_\Gamma\mathbf{v}_\Gamma)_{L^2(\Gamma_s)}\,ds = 0
    \end{aligned}
  \end{align}
  for all $(\eta,\xi)\in L_{W^{1,p}}^p(S_T)\times L_{L^p}^p(S_T)$ and $t\in[0,T]$.
  In this equation, we set
  \begin{align*}
    \eta = \eta_{\mathrm{d}} := \Lambda_\gamma'(gv_{\mathrm{d}}), \quad \xi = \xi_{\mathrm{d}} := g\zeta_{\mathrm{d}}\Lambda_\gamma''(gv_{\mathrm{d}}).
  \end{align*}
  Then, $(\eta_{\mathrm{d}},\xi_{\mathrm{d}})\in L_{W^{1,p}}^p(S_T)\times L_{L^p}^p(S_T)$ by \eqref{Pf_LUn:abs}.
  Moreover,
  \begin{align*}
    \nabla_\Gamma\eta_{\mathrm{d}}+\xi_{\mathrm{d}}\bm{\nu} &= \Lambda_\gamma''(gv_{\mathrm{d}})\nabla_\Gamma(gv_{\mathrm{d}})+g\zeta_{\mathrm{d}}\Lambda_\gamma''(gv_{\mathrm{d}})\bm{\nu} \\
    &= g\Lambda_\gamma''(gv_{\mathrm{d}})(\nabla_\Gamma v_{\mathrm{d}}+\zeta_{\mathrm{d}}\bm{\nu})+v_{\mathrm{d}}\Lambda_\gamma''(gv_{\mathrm{d}})\nabla_\Gamma g \\
    &= g\Lambda_\gamma''(gv_{\mathrm{d}})\mathbf{b}_{\mathrm{d}}+v_{\mathrm{d}}\Lambda_\gamma''(gv_{\mathrm{d}})\nabla_\Gamma g
  \end{align*}
  on $S_T$.
  Thus, substituting $(\eta_{\mathrm{d}},\xi_{\mathrm{d}})$ for \eqref{Pf_LUn:WF_Df}, we get $\sum_{k=1}^5\mathcal{I}_\gamma^k(t)=0$, where
  \begin{align*}
    \mathcal{I}_\gamma^1(t) &:= \int_0^t\langle\partial_\Gamma^\bullet(gv_{\mathrm{d}}),\Lambda_\Gamma'(gv_{\mathrm{d}})\rangle_{W^{1,p}(\Gamma_s)}\,ds, \\
    \mathcal{I}_\gamma^2(t) &:= \int_0^t\bigl(g\mathbf{z}_{\mathrm{d}},g\Lambda_\gamma''(gv_{\mathrm{d}})\mathbf{b}_{\mathrm{d}}\bigr)_{L^2(\Gamma_s)}\,ds, \\
    \mathcal{I}_\gamma^3(t) &:= \int_0^t\bigl(g\mathbf{z}_{\mathrm{d}},v_{\mathrm{d}}\Lambda_\gamma''(gv_{\mathrm{d}})\nabla_\Gamma g\bigr)_{L^2(\Gamma_s)}\,ds, \\
    \mathcal{I}_\gamma^4(t) &:= \int_0^t\bigl(gv_{\mathrm{d}},\Lambda_\gamma''(gv_{\mathrm{d}})\mathbf{v}_\Gamma\cdot[g\mathbf{b}_{\mathrm{d}}+v_{\mathrm{d}}\nabla_\Gamma g]\bigr)_{L^2(\Gamma_s)}\,ds, \\
    \mathcal{I}_\gamma^5(t) &:= \int_0^t\bigl(gv_{\mathrm{d}},\Lambda_\gamma'(gv_{\mathrm{d}})\,\mathrm{div}_\Gamma\mathbf{v}_\Gamma\bigr)_{L^2(\Gamma_s)}\,ds.
  \end{align*}
  Let us compute each $\mathcal{I}_\gamma^k(t)$.
  First, we see by \eqref{E:pVec_Coer}, \eqref{Pf_LUn:abs}, and \eqref{Pf_LUn:bd} that
  \begin{align} \label{Pf_LUn:I2}
    \mathcal{I}_\gamma^2(t) = \int_0^t\left(\int_{\Gamma_s}g^2\Lambda_\gamma''(gv_{\mathrm{d}})(\mathbf{z}_{\mathrm{d}}\cdot\mathbf{b}_{\mathrm{d}})\,d\mathcal{H}^{n-1}\right)\,ds \geq 0.
  \end{align}
  Next, we apply \eqref{E:Wpp_Comp} with $\varepsilon=0$ to $\mathcal{I}_\gamma^1(t)$ to observe that
  \begin{align*}
    \mathcal{I}_\gamma^1(t) &= \int_{\Gamma_t}[\Lambda_\gamma(gv_{\mathrm{d}})](t)\,d\mathcal{H}^{n-1}-\int_{\Gamma_0}[\Lambda_\gamma(gv_{\mathrm{d}})](0)\,d\mathcal{H}^{n-1} \\
    &\qquad -\int_0^t\left(\int_{\Gamma_s}\Lambda_\gamma(gv_{\mathrm{d}})\,\mathrm{div}_\Gamma\mathbf{v}_\Gamma\,d\mathcal{H}^{n-1}\right)\,ds.
  \end{align*}
  For each $s\in[0,T]$, it follows from \eqref{Pf_LUn:abs} and $g\in C^\infty(\overline{S_T})$ that
  \begin{align*}
    0 \leq [\Lambda_\gamma(gv_{\mathrm{d}})](s) \leq 1+c|v_{\mathrm{d}}(s)|, \quad \lim_{\gamma\to0}[\Lambda_\gamma(gv_{\mathrm{d}})](s) = |[gv_{\mathrm{d}}](s)| \quad\text{a.e. on}\quad \Gamma_s.
  \end{align*}
  Moreover, $1+c|v_{\mathrm{d}}(s)|\in L^1(\Gamma_s)$ by \eqref{E:Sur_Area} and $v_{\mathrm{d}}\in C_{L^2}(S_T)\subset C_{L^1}(S_T)$.
  Thus,
  \begin{align*}
    \lim_{\gamma\to0}\int_{\Gamma_s}[\Lambda_\gamma(gv_{\mathrm{d}})](s)\,d\mathcal{H}^{n-1} = \int_{\Gamma_s}|[gv_{\mathrm{d}}](s)|\,d\mathcal{H}^{n-1} \quad\text{for all}\quad s\in[0,T]
  \end{align*}
  by the dominated convergence theorem.
  Similarly, we can observe that
  \begin{align*}
    \lim_{\gamma\to0}\,\bigl\|\,\Lambda_\gamma(gv_{\mathrm{d}})-|gv_{\mathrm{d}}|\,\bigr\|_{L_{L^1}^1(S_T)} = 0.
  \end{align*}
  By these results and the smoothness of $\mathbf{v}_\Gamma$ on $\overline{S_T}$, we find that
  \begin{align} \label{Pf_LUn:I1}
    \begin{aligned}
      \lim_{\gamma\to0}\mathcal{I}_\gamma^1(t) &= \int_{\Gamma_t}|[gv_{\mathrm{d}}](t)|\,d\mathcal{H}^{n-1}-\int_{\Gamma_0}|[gv_{\mathrm{d}}](0)|\,d\mathcal{H}^{n-1} \\
      &\qquad -\int_0^t\left(\int_{\Gamma_s}|gv_{\mathrm{d}}|\,\mathrm{div}_\Gamma\mathbf{v}_\Gamma\,d\mathcal{H}^{n-1}\right)\,ds.
    \end{aligned}
  \end{align}
  Let us show $\mathcal{I}_\gamma^3(t),\mathcal{I}_\gamma^4(t)\to0$ as $\gamma\to0$.
  To this end, we define
  \begin{align*}
    \mathcal{A}_\gamma &:= \{(y,t)\in S_T \mid |[gv_{\mathrm{d}}](y,t)| \leq \gamma\}, \quad \gamma\in(0,1), \\
    \mathcal{A}_0 &:= \{(y,t)\in S_T \mid [gv_{\mathrm{d}}](y,t) = 0\},
  \end{align*}
  and let $\chi_\gamma$ and $\chi_0$ be the characteristic functions of $\mathcal{A}_\gamma$ and $\mathcal{A}_0$, respectively.
  Then,
  \begin{align*}
    \mathcal{A}_0 = \{(y,t)\in S_T \mid v_1(y,t) = v_2(y,t)\}, \quad \quad \lim_{\gamma\to0}\chi_\gamma = \chi_0 \quad\text{on}\quad S_T,
  \end{align*}
  where the first relation follows from \eqref{E:G_Bdd} and $v_{\mathrm{d}}=v_1-v_2$.
  Since
  \begin{align*}
    0 \leq \Lambda_\gamma''(gv_{\mathrm{d}}) \leq \frac{2c_\rho}{\gamma e}\chi_\gamma, \quad |v_{\mathrm{d}}|\chi_\gamma = \frac{1}{g}|gv_{\mathrm{d}}|\chi_\gamma \leq c\gamma\chi_\gamma \quad\text{on}\quad S_T
  \end{align*}
  by \eqref{Pf_LUn:abs}, the definition of $\mathcal{A}_\gamma$, and \eqref{E:G_Bdd}, and since $g$ and $\mathbf{v}_\Gamma$ are smooth on $\overline{S_T}$,
  \begin{align*}
    |\mathcal{I}_\gamma^3(t)| &\leq c\int_0^T\left(\int_{\Gamma_s}|\mathbf{z}_{\mathrm{d}}|\chi_\gamma\,d\mathcal{H}^{n-1}\right)\,ds, \\
    |\mathcal{I}_\gamma^4(t)| &\leq c\int_0^T\left(\int_{\Gamma_s}(|\mathbf{b}_{\mathrm{d}}|+\gamma)\chi_\gamma\,d\mathcal{H}^{n-1}\right)\,ds.
  \end{align*}
  In the right-hand side, we observe by \eqref{E:WFL_Con} and \eqref{Pf_LUn:bd} that
  \begin{align*}
    \lim_{\gamma\to0}|\mathbf{z}_{\mathrm{d}}|\chi_\gamma = |\mathbf{z}_{\mathrm{d}}|\chi_0 = 0, \quad \lim_{\gamma\to0}(|\mathbf{b}_{\mathrm{d}}|+\gamma)\chi_\gamma = |\mathbf{b}_{\mathrm{d}}|\chi_0 = 0 \quad\text{a.e. on}\quad S_T.
  \end{align*}
  Moreover, $0\leq|\mathbf{z}_{\mathrm{d}}|\chi_\gamma\leq|\mathbf{z}_{\mathrm{d}}|$ and $0\leq(|\mathbf{b}_{\mathrm{d}}|+\gamma)\chi_\gamma\leq|\mathbf{b}_{\mathrm{d}}|+1$ on $S_T$, and
  \begin{align*}
    |\mathbf{z}_{\mathrm{d}}| \in L_{L^{p'}}^{p'}(S_T) \subset L_{L^1}^1(S_T), \quad |\mathbf{b}_{\mathrm{d}}|+1 \in L_{L^p}^p(S_T) \subset L_{L^1}^1(S_T)
  \end{align*}
  by \eqref{Pf_LUn:bd}, $v_{\mathrm{d}}\in L_{W^{1,p}}^p(S_T)$, $\zeta_{\mathrm{d}}\in L_{L^p}^p(S_T)$.
  Thus,
  \begin{align*}
    \int_0^T\left(\int_{\Gamma_s}|\mathbf{z}_{\mathrm{d}}|\chi_\gamma\,d\mathcal{H}^{n-1}\right)\,ds, \, \int_0^T\left(\int_{\Gamma_s}(|\mathbf{b}_{\mathrm{d}}|+\gamma)\chi_\gamma\,d\mathcal{H}^{n-1}\right)\,ds \to 0
  \end{align*}
  as $\gamma\to0$ by the dominated convergence theorem and we obtain
  \begin{align} \label{Pf_LUn:I34}
    \lim_{\gamma\to0}\mathcal{I}_\gamma^3(t) = \lim_{\gamma\to0}\mathcal{I}_\gamma^4(t) = 0.
  \end{align}
  For $\mathcal{I}_\gamma^5(t)$, we see by \eqref{Pf_LUn:abs} and $g\in C^\infty(\overline{S_T})$ that
  \begin{align*}
    \bigl|\,gv_{\mathrm{d}}\Lambda_\gamma'(gv_{\mathrm{d}})-|gv_{\mathrm{d}}|\,\bigr| \leq c|v_{\mathrm{d}}|, \quad \lim_{\gamma\to0}gv_{\mathrm{d}}\Lambda_\gamma'(gv_{\mathrm{d}}) = |gv_{\mathrm{d}}| \quad\text{a.e. on} \quad S_T.
  \end{align*}
  Thus, the dominated convergence theorem implies that
  \begin{align*}
    \lim_{\gamma\to0}\,\bigl\|\,gv_{\mathrm{d}}\Lambda_\gamma'(gv_{\mathrm{d}})-|gv_{\mathrm{d}}|\,\bigr\|_{L_{L^1}^1(S_T)} = 0,
  \end{align*}
  and it follows from this result and the smoothness of $\mathbf{v}_\Gamma$ on $\overline{S_T}$ that
  \begin{align} \label{Pf_LUn:I5}
    \begin{aligned}
      \lim_{\gamma\to0}\mathcal{I}_\gamma^5(t) &= \lim_{\gamma\to0}\int_0^t\left(\int_{\Gamma_s}gv_{\mathrm{d}}\Lambda_\gamma'(gv_{\mathrm{d}})\,\mathrm{div}_\Gamma\mathbf{v}_\Gamma\,d\mathcal{H}^{n-1}\right)\,ds \\
      &= \int_0^t\left(\int_{\Gamma_s}|gv_{\mathrm{d}}|\,\mathrm{div}_\Gamma\mathbf{v}_\Gamma\,d\mathcal{H}^{n-1}\right)\,ds.
    \end{aligned}
  \end{align}
  Now, we send $\gamma\to0$ in $\sum_{k=1}^5\mathcal{I}_\gamma^k(t)=0$ and use \eqref{Pf_LUn:I2}--\eqref{Pf_LUn:I5} to get
  \begin{align*}
    \int_{\Gamma_t}|[gv_{\mathrm{d}}](t)|\,d\mathcal{H}^{n-1}-\int_{\Gamma_0}|[gv_{\mathrm{d}}](0)|\,d\mathcal{H}^{n-1} \leq 0.
  \end{align*}
  Moreover, since $v_{\mathrm{d}}(0)=v_1(0)-v_2(0)=0$ on $\Gamma_0$ and \eqref{E:G_Bdd} holds, we have
  \begin{align*}
    \int_{\Gamma_t}|v_{\mathrm{d}}(t)|\,d\mathcal{H}^{n-1} \leq c\int_{\Gamma_t}|[gv_{\mathrm{d}}](t)|\,d\mathcal{H}^{n-1} \leq 0
  \end{align*}
  and thus $v_{\mathrm{d}}(t)=0$ a.e. on $\Gamma_t$ for all $t\in[0,T]$.
  Hence, $v_1=v_2$ a.e. on $S_T$, and it follows from \eqref{E:WFL_Con} that $\zeta_1=\zeta_2$ a.e. on $S_T$.
\end{proof}

Now, we return to the setting of Theorem \ref{T:TFL} and complete its proof.

\begin{proof}[Proof of Theorem \ref{T:TFL}]
  Let $v^\varepsilon$ and $\zeta^\varepsilon$ be given by \eqref{E:Def_veps}.
  We see that subsequences of $v^\varepsilon$ and $\zeta^\varepsilon$ converge weakly in the sense of \eqref{E:veps_WeCo} by Proposition \ref{P:veps_Conv}, and that the pair $(v,\zeta)$ of the weak limits is a unique weak solution to \eqref{E:pLap_Lim} by Propositions \ref{P:LiWF_uw}--\ref{P:Ch_zeta} and Theorem \ref{T:Lim_Uni}.
  By the uniqueness, we also find that the weak convergence \eqref{E:veps_WeCo} actually holds for the whole sequences of $v^\varepsilon$ and $\zeta^\varepsilon$ as $\varepsilon\to0$.
  The proof is complete.
\end{proof}

We also observe that Theorem \ref{T:Lim_UnEx} is an immediate consequence of Theorem \ref{T:TFL}.

\begin{proof}[Proof of Theorem \ref{T:Lim_UnEx}]
  For $\varepsilon>0$, let $f^\varepsilon:=\bar{f}$ in $L_{[W^{1,p}]^\ast}^{p'}(Q_T^\varepsilon)$ and
  \begin{align*}
    u_0^\varepsilon(X) := \frac{\bar{v}_0(X)}{J(\pi(X),0,d(X))} = \frac{v_0(Y)}{J(Y,0,r)}, \quad X = Y+r\bm{\nu}(Y,0) \in \Omega_0^\varepsilon.
  \end{align*}
  Then, we have \eqref{E:Data_MTD} by \eqref{E:J_Bdd}, \eqref{E:CE_Lq}, and \eqref{E:CE_Func}.
  Also, \eqref{E:u0Av_WeCo} and \eqref{E:ExF_StCo} hold by $\mathcal{M}_\varepsilon u_0^\varepsilon=v_0$ on $\Gamma_0$ and the definition of $f^\varepsilon$.
  Thus, we can apply Theorem \ref{T:TFL} (and Theorem \ref{T:Lim_Uni}) to get the existence and uniqueness of a weak solution to \eqref{E:pLap_Lim}.
\end{proof}

\begin{remark} \label{R:Lim_UnEx}
  The Galerkin method may be also applicable to construct a weak solution to \eqref{E:pLap_Lim} as in \cite{Miu25pre_pLap}.
  However, it seems that some careful considerations will be required when one uses a theory of ordinary differential equations to get a finite-dimensional approximate solution $(v_N,\zeta_N)$.
  Indeed, since $\zeta_N$ is determined by $v_N$ through the nonlinear relation \eqref{E:Ch_zeta}, the continuity (in a suitable sense) of the term
  \begin{align*}
    (|\nabla_\Gamma v_N|^2+\zeta_N^2)^{(p-2)/2}\nabla_\Gamma v_N = \{|\nabla_\Gamma v_N|^2+[\zeta_N(v_N)]^2\}^{(p-2)/2}\nabla_\Gamma v_N
  \end{align*}
  with respect to the coefficients of $v_N$ is not obvious.
  We leave this issue open.
\end{remark}

%%% Section 8 %%%
\section{Proofs of auxiliary results} \label{S:Pf_Aux}
This sections provides the proofs of auxiliary results in Section \ref{S:Prelim}.

\subsection{Results on a moving surface} \label{SS:PA_Sur}
Let us show Lemmas \ref{L:MS_ST}--\ref{L:CoSur_Lq}.
For $t\in[0,T]$ and $y\in\Gamma_t$, let $T_y\Gamma_t$ be the tangent plane of $\Gamma_t$ at $y$.
We consider $T_y\Gamma_t$ as a subspace of $\mathbb{R}^n$.
Let $\Phi_{\pm(\cdot)}^0$ be the mappings given in Assumption \ref{A:Flow_Sur}.
We denote by
\begin{align*}
  [D\Phi_t^0]_Y\colon T_Y\Gamma_0 \to T_{\Phi_t^0(Y)}\Gamma_t, \quad Y\in\Gamma_0, \quad t\in[0,T]
\end{align*}
the differential of $\Phi_t^0$ at $Y$.
It is a linear isomorphism by Assumption \ref{A:Flow_Sur}.

\begin{proof}[Proof of Lemma \ref{L:MS_ST}]
  First, we see that $\overline{S_T}$ is compact in $\mathbb{R}^{n+1}$ since it is the image of the compact set $\Gamma_0\times[0,T]$ under the smooth mapping
  \begin{align*}
    \Gamma_0\times[0,T] \ni (Y,t) \mapsto (\Phi_t^0(Y),t) \in \mathbb{R}^{n+1}.
  \end{align*}
  Let us show that $\overline{S_T}$ is a smooth hypersurface in $\mathbb{R}^{n+1}$ (the statement for $S_T$ is proved in the same way).
  Fix any $(y_0,t_0)\in\overline{S_T}$ and let $Y_0:=\Phi_{-t_0}^0(y_0)\in\Gamma_0$.
  Since $\Gamma_0$ is a smooth hypersurface in $\mathbb{R}^n$, we can take an open subset $U$ of $\mathbb{R}^{n-1}$ with $0\in U$ and a smooth local parametrization $Y\colon U\to\Gamma_0$ such that $Y_0=Y(0)$ and the vectors
  \begin{align*}
    \partial_{\xi_i}Y(0) \in T_{Y_0}\Gamma_0, \quad i=1,\dots,n-1
  \end{align*}
  are linearly independent.
  Let
  \begin{align*}
    y(\xi,t) := [\Phi_t^0\circ Y](\xi) = \Phi_t^0\bigl(Y(\xi)\bigr) \in \Gamma_t, \quad (\xi,t)\in U\times[0,T]
  \end{align*}
  so that $y_0=\Phi_{t_0}^0(Y_0)=y(0,t_0)$.
  Then, $y$ is smooth on $U\times[0,T]$ and the vectors
  \begin{align*}
    \partial_{\xi_i}y(0,t_0) = [D\Phi_{t_0}^0]_{Y_0}\partial_{\xi_i}Y(0) \in T_{y_0}\Gamma_{t_0}, \quad i=1,\dots,n-1
  \end{align*}
  are linearly independent by Assumption \ref{A:Flow_Sur}.
  Thus, writing
  \begin{align*}
    y(\xi,t) = \bigl(y'(\xi,t),y_n(\xi,t)\bigr) = \bigl(y_1(\xi,t),\dots,y_{n-1}(\xi,t),y_n(\xi,t)\bigr)
  \end{align*}
  and changing the orders of $y_1,\dots,y_n$ if necessary, we may assume that the vectors
  \begin{align*}
    \partial_{\xi_i}y'(0,t_0) \in \mathbb{R}^{n-1}, \quad i=1,\dots,n-1
  \end{align*}
  are linearly independent.
  Now, we define a mapping $\psi\colon U\times[0,T]\to\mathbb{R}^n$ by
  \begin{align*}
    \psi(\xi,t) := (y'(\xi,t),t), \quad (\xi,t)\in U\times[0,T].
  \end{align*}
  Then, $\psi$ is smooth on $U\times[0,T]$ and $\psi(0,t_0)=(y_0',t_0)$ with
  \begin{align*}
    y_0 = (y_0',y_{0,n}) = (y_{0,1},\dots,y_{0,n-1},y_{0,n}).
  \end{align*}
  Moreover, the vectors
  \begin{align*}
    \partial_{\xi_1}\psi(0,t_0),\dots,\partial_{\xi_{n-1}}\psi(0,t_0),\partial_t\psi(0,t_0) \in \mathbb{R}^n
  \end{align*}
  are linearly independent, since
  \begin{align*}
    \partial_{\xi_i}\psi(0,t_0) =
    \begin{pmatrix}
      \partial_{\xi_i}y'(0,t_0) \\
      0
    \end{pmatrix},
    \quad \partial_t\psi(0,t_0) =
    \begin{pmatrix}
      \partial_ty'(0,t_0) \\ 1
    \end{pmatrix}
  \end{align*}
  and the vectors $\partial_{\xi_i}y'(0,t_0)$ are linearly independent.
  Thus, by the inverse function theorem, we can take a small number $\sigma_0>0$ such that, if we set
  \begin{align} \label{Pf_ST:V}
    \begin{gathered}
      B_{\sigma_0}(y_0') := \{y'\in\mathbb{R}^{n-1} \mid |y'-y_0'|<\sigma_0\}, \\
      W := B_{\sigma_0}(y_0')\times\bigl((t_0-\sigma_0,t_0+\sigma_0)\cap[0,T]\bigr), \quad V := \psi^{-1}(W),
    \end{gathered}
  \end{align}
  then $\psi\colon V\to W$ is a diffeomorphism.
  By the definition of $\psi$, we can write
  \begin{align*}
    \psi^{-1}(y',t) = (\xi(y',t),t), \quad (y',t) \in W
  \end{align*}
  with a smooth mapping $\xi(y',t)$ on $W$.
  Then, since $\psi(0,t_0)=(y_0',t_0)$,
  \begin{align*}
    \xi(y_0',t_0) = 0, \quad y_n(\xi(y_0',t_0),t_0) = y_n(0,t_0) = y_{0,n}.
  \end{align*}
  Moreover, we have $y'(\xi(z',t),t)=z'$ for $(z',t)\in W$ by
  \begin{align*}
    (y'(\xi(z',t),t),t) = \psi(\xi(z',t),t) = \psi\bigl(\psi^{-1}(z',t)\bigr) = (z',t).
  \end{align*}
  Let $z_n(z',t):=y_n(\xi(z',t),t)$ for $(z',t)\in W$.
  Then, $z_n$ is smooth on $W$ and
  \begin{align} \label{Pf_ST:ysz}
    y(\xi(z',t),t) = \bigl(y'(\xi(z',t),t),y_n(\xi(z',t),t)\bigr) = \bigl(z',z_n(z',t)\bigr), \quad (z',t) \in W.
  \end{align}
  Also, $(y_0',t_0)\in W$ and $y_{0,n}=z_n(y_0',t_0)$.
  Thus, $\overline{S_T}$ is locally expressed as a graph
  \begin{align} \label{Pf_ST:loc}
    \overline{S_T} = \{(z',z_n(z',t),t) \mid (z',t)\in W\}
  \end{align}
  near $(y_0,t_0)$.
  This shows that $\overline{S_T}$ is a smooth hypersurface in $\mathbb{R}^{n+1}$.
\end{proof}

\begin{proof}[Proof of Lemma \ref{L:MS_Reg}]
  Fix any $(y_0,t_0)\in\overline{S_T}$.
  We use the notations in the proof of Lemma \ref{L:MS_ST} above.
  Hence, $\overline{S_T}$ is locally expressed as \eqref{Pf_ST:loc} near $(y_0,t_0)$.
  We have
  \begin{align*}
    [\Phi_{-(\cdot)}^0](y(\xi,t),t) = \Phi_{-t}^0\bigl(y(\xi,t)\bigr) = Y(\xi), \quad (\xi,t)\in U\times[0,T]
  \end{align*}
  by $y(\xi,t)=\Phi_t^0\bigl(Y(\xi)\bigr)$.
  Thus, we set $\xi=\xi(z',t)$ and use \eqref{Pf_ST:ysz} to get
  \begin{align*}
    [\Phi_{-(\cdot)}^0](z',z_n(z',t),t) = Y\bigl(\xi(z',t)\bigr), \quad (z',t)\in W.
  \end{align*}
  Since the right-hand side is smooth on $W$, we find that $\Phi_{-(\cdot)}^0$ is smooth near $(y_0,t_0)$.
  Next, we see by \eqref{Pf_ST:loc} that $\bm{\nu}$ is locally expressed as
  \begin{align} \label{Pf_Re:nulc}
    \bm{\nu}(z',z_n(z',t),t) = \pm\frac{1}{\sqrt{1+|\nabla_{z'}z_n(z',t)|^2}}
    \begin{pmatrix}
      -\nabla_{z'}z_n(z',t) \\
      1
    \end{pmatrix},
    \quad (z',t) \in W.
  \end{align}
  On the other hand, by Assumption \ref{A:Flow_Sur} (ii), we have
  \begin{align*}
    \bm{\nu}(z',z_n(z',t),t) = \tilde{\bm{\nu}}(z',z_n(z',t),t), \quad (z',t) \in W,
  \end{align*}
  and the right-hand side is continuous on $W$ by the continuity of $\tilde{\bm{\nu}}$ and $z_n$.
  Hence, $\bm{\nu}$ takes only one sign in the local expression \eqref{Pf_Re:nulc}.
  By this fact and the smoothness of $z_n$, we find that $\bm{\nu}$ is smooth near $(y_0,t_0)$.
  Therefore, $\Phi_{-(\cdot)}^0$ and $\bm{\nu}$ are smooth on $\overline{S_T}$.
\end{proof}

To prove Lemma \ref{L:MS_Tubu}, we prepare an auxiliary result.

\begin{lemma} \label{L:MS_ILC}
  Let $t\in[0,T]$.
  We define a mapping
  \begin{align*}
    x_t\colon\Gamma_t\times\mathbb{R}\to\mathbb{R}^n, \quad x_t(y,r) := y+r\bm{\nu}(y,t).
  \end{align*}
  Then, there exist constants $c_0,\rho_0>0$ independent of $t$ such that
  \begin{align} \label{E:MS_ILC}
    |y_1-y_2|+|r_1-r_2| \leq c_0|x_t(y_1,r_1)-x_t(y_2,r_2)|
  \end{align}
  for all $y_1,y_2\in\Gamma_t$ and $r_1,r_2\in[-\rho_0,\rho_0]$.
\end{lemma}

\begin{proof}
  We follow the proof in the case of a fixed surface \cite[Theorem III.3.5]{BoyFab13}.

  Assume to the contrary that there exist sequences
  \begin{align*}
    t_k\in[0,T], \quad y_{1,k},y_{2,k}\in\Gamma_{t_k}, \quad r_{1,k},r_{2,k} \in \left[-\frac{1}{k},\frac{1}{k}\right], \quad k\in\mathbb{N}
  \end{align*}
  and $x_{j,k}:=x_{t_k}(y_{j,k},r_{j,k})=y_{j,k}+r_{j,k}\bm{\nu}(y_{j,k},t_k)$ for $j=1,2$ such that
  \begin{align} \label{Pf_MC:Cont}
    |y_{1,k}-y_{2,k}|+|r_{1,k}-r_{2,k}| > k|x_{1,k}-x_{2,k}|.
  \end{align}
  Let $Y_{j,k}:=\Phi_{-t_k}^0(y_{j,k})\in\Gamma_0$ for $j=1,2$.
  Since $[0,T]$ and $\Gamma_0$ are compact, we have
  \begin{align*}
    \lim_{k\to\infty}t_k = t_0 \in [0,T], \quad \lim_{k\to\infty}Y_{j,k} = Y_j \in \Gamma_0, \quad j=1,2
  \end{align*}
  up to a subsequence.
  Then, since $\Phi_{(\cdot)}^0$ is smooth on $\Gamma_0\times[0,T]$,
  \begin{align*}
    \lim_{k\to\infty}y_{j,k} &= \lim_{k\to\infty}\Phi_{t_k}^0(Y_{j,k}) = \Phi_{t_0}^0(Y_j) =: y_j \in \Gamma_{t_0}, \quad j=1,2.
  \end{align*}
  Moreover, $x_{j,k}\to y_j$ as $k\to\infty$ by $|\bm{\nu}(y_{j,k},t_k)|=1$ and $r_{j,k}\to0$.
  Since
  \begin{align*}
    |x_{1,k}-x_{2,k}| < \frac{1}{k}|y_{1,k}-y_{2,k}|+\frac{1}{k}|r_{1,k}-r_{2,k}| \to 0
  \end{align*}
  as $k\to\infty$ by \eqref{Pf_MC:Cont}, $y_{j,k}\to y_j$, and $r_{j,k}\to0$, we find that $y_1=y_2$.

  In what follows, we set $y_0:=y_1\in\Gamma_{t_0}$ and use the notations in the proof of Lemma \ref{L:MS_ST}.
  Thus, $\overline{S_T}$ is locally expressed as \eqref{Pf_ST:loc} near $(y_0,t_0)$, where $W$ is given by \eqref{Pf_ST:V} and $z_n$ is a smooth function on $W$.
  Let
  \begin{align*}
    B' := B_{\sigma_0/2}(y_0'), \quad I' := \left(t_0-\frac{\sigma_0}{2},t_0+\frac{\sigma_0}{2}\right)\cap[0,T], \quad W' := B'\times I'.
  \end{align*}
  Since $(y_{j,k},t_k)\to(y_0,t_0)$ as $k\to\infty$, we have
  \begin{align*}
    (y_{j,k},t_k) = (y_{j,k}',z_n(y_{j,k}',t_k),t_k), \quad (y_{j,k}',t_k) \in \overline{W'}, \quad j=1,2
  \end{align*}
  when $k$ is sufficiently large.
  For $(z',t) \in W$, we define
  \begin{align*}
    \varphi_0(z',t) := \bigl(z',z_n(z',t)\bigr) \in \Gamma_t, \quad \bm{\nu}_0(z',t) := \bm{\nu}(\varphi_0(z',t),t) \in \mathbb{R}^n.
  \end{align*}
  Since $\bm{\nu}_0$ is of the form \eqref{Pf_Re:nulc}, $\overline{B'}$ is convex in $\mathbb{R}^{n-1}$, and $\overline{W'}$ is compact in $W$,
  \begin{align} \label{Pf_MC:phi0}
    |\nabla_{z'}\varphi_0(z',t)| \leq c, \quad |\bm{\nu}_0(z',t)-\bm{\nu}_0(w',t)| \leq c|z'-w'|
  \end{align}
  for $z',w'\in\overline{B'}$ and $t\in\overline{I'}$.
  Also, for $z_1',z_2'\in\overline{B'}$ and $t\in\overline{I'}$, we define curves
  \begin{align*}
    z'(s) := sz_1'+(1-s)z_2' \in \overline{B'}, \quad \gamma(s) := \varphi_0(z'(s),t)\in\Gamma_t, \quad s\in[0,1].
  \end{align*}
  Then, since the vector $\frac{d\gamma}{ds}(s)$ is tangent to $\Gamma_t$ at $\gamma(s)$ for each $s\in[0,1]$,
  \begin{align*}
    \{\varphi_0(z_1',t)-\varphi_0(z_2',t)\}\cdot\bm{\nu}_0(z_1',t) &= \int_0^1\frac{d\gamma}{ds}(s)\cdot\bm{\nu}_0(z_1',t)\,ds \\
    &= \int_0^1\frac{d\gamma}{ds}(s)\cdot\{\bm{\nu}_0(z_1',t)-\bm{\nu}_0(z'(s),t)\}\,ds.
  \end{align*}
  Thus, using the chain rule for $\gamma$ and applying \eqref{Pf_MC:phi0}, we get
  \begin{align} \label{Pf_MC:inner}
    |\{\varphi_0(z_1',t)-\varphi_0(z_2',t)\}\cdot\bm{\nu}_0(z_1',t)| \leq c|z_1'-z_2'|^2.
  \end{align}
  Now, we apply \eqref{Pf_MC:phi0} and \eqref{Pf_MC:inner} to $y_{j,k}=\varphi_0(y_{j,k}',t_k)$ to get
  \begin{align} \label{Pf_MC:nudif}
    \begin{aligned}
      |\bm{\nu}(y_{1,k},t_k)-\bm{\nu}(y_{2,k},t_k)| &\leq c|y_{1,k}'-y_{2,k}'|, \\
      |(y_{1,k}-y_{2,k})\cdot\bm{\nu}(y_{1,k},t_k)| &\leq c|y_{1,k}'-y_{2,k}'|^2.
    \end{aligned}
  \end{align}
  By $x_{j,k}=y_{j,k}+r_{j,k}\bm{\nu}(y_{j,k},t_k)$, we observe that
  \begin{align} \label{Pf_MC:rnu}
    \begin{aligned}
      &(r_{1,k}-r_{2,k})\bm{\nu}(y_{1,k},t_k) \\
      &\qquad = (x_{1,k}-x_{2,k})-(y_{1,k}-y_{2,k})-r_{2,k}\{\bm{\nu}(y_{1,k},t_k)-\bm{\nu}(y_{2,k},t_k)\}.
    \end{aligned}
  \end{align}
  Thus, we take the inner product of both sides with $\bm{\nu}(y_{1,k},t_k)$ to get
  \begin{align*}
    |r_{1,k}-r_{2,k}| \leq |x_{1,k}-x_{2,k}|+|(y_{1,k}-y_{2,k})\cdot\bm{\nu}(y_{1,k},t_k)|+\frac{1}{k}|\bm{\nu}(y_{1,k},t_k)-\bm{\nu}(y_{2,k},t_k)|
  \end{align*}
  by $|\bm{\nu}(y_{1,k},t_k)|=1$ and $|r_{2,k}|\leq1/k$.
  We further use \eqref{Pf_MC:Cont}, \eqref{Pf_MC:nudif}, and
  \begin{align} \label{Pf_MC:aux}
    |y_{1,k}'-y_{2,k}'| \leq |y_{1,k}-y_{2,k}|,
  \end{align}
  since we write $y=(y',y_n)$ for $y\in\mathbb{R}^n$.
  Then, we have
  \begin{align*}
    |r_{1,k}-r_{2,k}| \leq \frac{1}{k}|r_{1,k}-r_{2,k}|+\frac{c}{k}|y_{1,k}-y_{2,k}|+c|y_{1,k}-y_{2,k}|^2
  \end{align*}
  and thus, when $k$ is sufficiently large,
  \begin{align} \label{Pf_MC:rdif}
    |r_{1,k}-r_{2,k}| \leq \frac{c}{k}|y_{1,k}-y_{2,k}|+c|y_{1,k}-y_{2,k}|^2.
  \end{align}
  Also, we see by \eqref{Pf_MC:rnu}, $|\bm{\nu}(y_{1,k},t_k)|=1$, and $|r_{2,k}|\leq1/k$ that
  \begin{align*}
    |y_{1,k}-y_{2,k}| &\leq |r_{1,k}-r_{2,k}|+|x_{1,k}-x_{2,k}|+\frac{1}{k}|\bm{\nu}(y_{1,k},t_k)-\bm{\nu}(y_{2,k},t_k)|.
  \end{align*}
  Thus, noting that
  \begin{align*}
    |x_{1,k}-x_{2,k}| \leq \frac{1}{k}|y_{1,k}-y_{2,k}|+|r_{1,k}-r_{2,k}|
  \end{align*}
  by \eqref{Pf_MC:Cont} and $1/k\leq1$, we use \eqref{Pf_MC:nudif}, \eqref{Pf_MC:aux}, and \eqref{Pf_MC:rdif} to get
  \begin{align*}
    |y_{1,k}-y_{2,k}| \leq \frac{c}{k}|y_{1,k}-y_{2,k}|+c|y_{1,k}-y_{2,k}|^2 = c\left(\frac{1}{k}+|y_{1,k}-y_{2,k}|\right)|y_{1,k}-y_{2,k}|.
  \end{align*}
  Since $|y_{1,k}-y_{2,k}|\to|y_0-y_0|=0$ as $k\to\infty$, this inequality gives
  \begin{align*}
    |y_{1,k}-y_{2,k}| \leq \frac{1}{2}|y_{1,k}-y_{2,k}|
  \end{align*}
  when $k$ is sufficiently large.
  Thus, $y_{1,k}=y_{2,k}$.
  Then, we also get $r_{1,k}=r_{2,k}$ by \eqref{Pf_MC:rdif}, but this contradicts the strict inequality of \eqref{Pf_MC:Cont}.
  Hence, the statement holds.
\end{proof}

\begin{proof}[Proof of Lemma \ref{L:MS_Tubu}]
  Let $x_t$ and $\rho_0$ be given in Lemma \ref{L:MS_ILC}.
  We set
  \begin{align*}
    \overline{\mathcal{N}_t^0} := \{x\in\mathbb{R}^n \mid -\rho_0 \leq d(x,t) \leq \rho_0\}, \quad t\in[0,T].
  \end{align*}
  Let $x_1\in\overline{\mathcal{N}_t^0}$.
  Since $\Gamma_t$ is compact, there exists at least one $y_1\in\Gamma_t$ such that
  \begin{align} \label{Pf_MT:dist}
    r_1 := d(x_1,t) =
    \begin{cases}
      -|x_1-y_1| &\text{if}\quad x\in\Omega_t, \\
      |x_1-y_1| &\text{if}\quad x\in\mathbb{R}^n\setminus\Omega_t,
    \end{cases}
  \end{align}
  where $\Omega_t$ is a bounded domain in $\mathbb{R}^n$ with boundary $\Gamma_t$ (see Section \ref{SS:Pre_Sur}).
  Then,
  \begin{align*}
    (x_1-y_1)\cdot\frac{d\gamma}{ds}(0) = 0 \quad\text{for any curve}\quad \gamma\colon\mathbb{R}\to\Gamma_t \quad\text{with}\quad \gamma(0)=y_1,
  \end{align*}
  since $|x_1-\gamma(s)|^2$ takes a local minimum at $s=0$.
  Thus, $x_1-y_1$ is orthogonal to $\Gamma_t$ at $y_1$.
  By this fact and \eqref{Pf_MT:dist}, we find that $x_1-y_1=r_1\bm{\nu}(y_1,t)$, i.e.
  \begin{align*}
    x_1 = y_1+r_1\bm{\nu}(y_1,t) = x_t(y_1,r_1), \quad y_1\in\Gamma_t, \quad r_1 = d(x_1,t) \in [-\rho_0,\rho_0].
  \end{align*}
  Moreover, if $x_1=x_t(y_2,r_2)$ with some $y_2\in\Gamma_t$ and $r_2\in[-\rho_0,\rho_0]$, then $y_1=y_2$ and $r_1=r_2$ by \eqref{E:MS_ILC}.
  Hence, we can define $\pi(x_1,t):=y_1$ to get the one-to-one correspondence
  \begin{align} \label{Pf_MT:Corr}
    x = x_t(y,r) \in \overline{\mathcal{N}_t^0} \quad\Leftrightarrow\quad (y,r) = \bigl(\pi(x,t),d(x,t)\bigr) \in \Gamma_t\times[-\rho_0,\rho_0].
  \end{align}
  This gives the first relation of \eqref{E:Fermi}.

  The second relation of \eqref{E:Fermi} is shown as in the case of a fixed surface \cite[Section III.3.2]{BoyFab13}.
  For the reader's convenience, we give the proof by assuming that $d$ is smooth (we confirm it later in some space-time domain).
  Let
  \begin{align*}
    t_0\in[0,T], \quad x_0\in\overline{\mathcal{N}_{t_0}^0}, \quad y_0 = \pi(x_0,t_0)\in\Gamma_{t_0}, \quad r_0 = d(x_0,t_0) \in[-\rho_0,\rho_0].
  \end{align*}
  We take an orthonormal basis $\{\mathbf{u}_1,\dots,\mathbf{u}_n\}$ of $\mathbb{R}^n$ such that $\mathbf{u}_1=\bm{\nu}(y_0,t_0)$ and write
  \begin{align*}
    \nabla d(x_0,t_0) = \alpha_1\mathbf{u}_1+\dots+\alpha_n\mathbf{u}_n, \quad \alpha_k = \nabla d(x_0,t_0)\cdot\mathbf{u}_k, \quad k=1,\dots,n.
  \end{align*}
  From the triangle inequality
  \begin{align*}
    |x_1-y| \leq |x_1-x_2|+|x_2-y|, \quad x_1,x_2\in\mathbb{R}^n, \quad y\in\Gamma_{t_0},
  \end{align*}
  we can deduce that $|d(x,t_0)-d(x_0,t_0)|\leq|x-x_0|$ when $x$ is near $x_0$.
  Thus,
  \begin{align} \label{Pf_MT:grd}
    |\nabla d(x_0,t_0)| = (\alpha_1^2+\cdots+\alpha_n^2)^{1/2} \leq 1.
  \end{align}
  On the other hand, when $h\in\mathbb{R}\setminus\{0\}$ is sufficiently small, we have
  \begin{align*}
    x_0+h\mathbf{u}_1 = y_0+(r_0+h)\bm{\nu}(y_0,t_0) = x_{t_0}(y_0,r_0+h)
  \end{align*}
  and thus $d(x_0+h\mathbf{u}_1,t_0)=r_0+h=d(x_0,t_0)+h$ by \eqref{Pf_MT:Corr}.
  Hence,
  \begin{align*}
    \alpha_1 = \nabla d(x_0,t_0)\cdot\mathbf{u}_1 = \lim_{h\to0}\frac{d(x_0+h\mathbf{u}_1,t_0)-d(x_0,t_0)}{h} = 1.
  \end{align*}
  By this result and \eqref{Pf_MT:grd}, we find that $\alpha_k=0$ for $k=2,\dots,n$ and
  \begin{align*}
    \nabla d(x_0,t_0) = \mathbf{u}_1 = \bm{\nu}(y_0,t_0).
  \end{align*}
  Since $y_0=\pi(x_0,t_0)$, it follows that the second relation of \eqref{E:Fermi} is valid.

  Let us show the smoothness of $\pi$ and $d$.
  Let $(y_0,t_0)\in\overline{S_T}$ and $Y_0=\Phi_{-t_0}^0(y_0)\in\Gamma_0$.
  As in the proof of Lemma \ref{L:MS_ST}, we take a smooth local parametrization $Y\colon U\to\Gamma_0$ such that $Y(0)=Y_0$ and the vectors $\partial_{\xi_i}Y(0)$ are linearly independent.
  We set
  \begin{align*}
    \Psi_0\colon U\times[-\rho_0,\rho_0]\times[0,T] \to \mathbb{R}^{n+1}, \quad \Psi_0(\xi,r,t) := (y(\xi,t)+r\bm{\nu}(y(\xi,t),t),t),
  \end{align*}
  where $y(\xi,t)=\Phi_t^0\bigl(Y(\xi)\bigr)$.
  Then, $\Psi_0$ is smooth on $U\times[-\rho_0,\rho_0]\times[0,T]$ and
  \begin{align} \label{Pf_MT:psi}
    \pi\bigl(\Psi_0(\xi,r,t)\bigr) = y(\xi,t), \quad d\bigl(\Psi_0(\xi,r,t)\bigr) = r
  \end{align}
  by \eqref{Pf_MT:Corr}.
  Moreover, $\Psi_0(0,0,t_0)=(y_0,t_0)$ and the vectors
  \begin{align*}
    \partial_{\xi_1}\Psi_0(0,0,t_0),\dots,\partial_{\xi_{n-1}}\Psi_0(0,0,t_0),\partial_r\Psi_0(0,0,t_0),\partial_t\Psi_0(0,0,t_0) \in \mathbb{R}^{n+1}
  \end{align*}
  are linearly independent, since (we suppress the argument $(0,0,t_0)$ below)
  \begin{align*}
    \partial_{\xi_i}\Psi_0 =
    \begin{pmatrix}
      \partial_{\xi_i}y(0,t_0) \\
      0
    \end{pmatrix},
    \quad \partial_r\Psi_0 =
    \begin{pmatrix}
      \bm{\nu}(y_0,t_0) \\
      0
    \end{pmatrix},
    \quad \partial_t\Psi_0 =
    \begin{pmatrix}
      \partial_ty(0,t_0) \\
      1
    \end{pmatrix}
  \end{align*}
  and the vectors $\partial_{\xi_i}y(0,t_0)\in T_{y_0}\Gamma_{t_0}$ are linearly independent and $\bm{\nu}(y_0,t_0)$ is orthogonal to $T_{y_0}\Gamma_{t_0}$.
  Thus, by the inverse function theorem, we can take small numbers
  \begin{align} \label{Pf_MT:sigma}
    \sigma_0 = \sigma_0(y_0,t_0) > 0, \quad \delta_0 = \delta_0(y_0,t_0) \in (0,\rho_0)
  \end{align}
  such that, if we define the sets $A_0=A_0(y_0,t_0)$ with $A=B,I,V,W$ by
  \begin{align} \label{Pf_MT:U0}
    \begin{gathered}
      B_0 := \{\xi\in\mathbb{R}^{n-1} \mid |\xi|<\sigma_0\}, \quad I_0 := (t_0-\sigma_0,t_0+\sigma_0)\cap[0,T], \\
      V_0 := B_0\times[-\delta_0,\delta_0]\times I_0, \quad W_0 := \Psi_0(V_0),
    \end{gathered}
  \end{align}
  then $\Psi_0\colon V_0\to W_0$ is a diffeomorphism.
  By the definition of $\Psi_0$, we can write
  \begin{align*}
    \Psi_0^{-1}(x,t) = (\xi(x,t),r(x,t),t) \in V_0, \quad (x,t)\in W_0,
  \end{align*}
  where $\xi(x,t)$ and $r(x,t)$ are smooth on $W_0$.
  Substituting this for \eqref{Pf_MT:psi}, we get
  \begin{align*}
    \pi(x,t) = y(\xi(x,t),t), \quad d(x,t) = r(x,t), \quad (x,t) \in W_0.
  \end{align*}
  Thus, $\pi$ and $d$ are smooth on $W_0$.
  Now, since $\overline{S_T}$ is compact and
  \begin{align*}
    \overline{S_T} \subset \textstyle\bigcup_{(y_0,t_0)\in\overline{S_T}}W_0 = \textstyle\bigcup_{(y_0,t_0)\in\overline{S_T}}W_0(y_0,t_0)
  \end{align*}
  by $(y_0,t_0)=\Psi_0(0,0,t_0)\in W_0$, we can take a finite covering
  \begin{align*}
    \overline{S_T} \subset \overline{\mathcal{W}_T} := \textstyle\bigcup_{\ell=1}^LW_\ell = \textstyle\bigcup_{\ell=1}^L\Psi_\ell(V_\ell),
  \end{align*}
  and $\pi$ and $d$ are smooth on $\overline{\mathcal{W}_T}$.
  Here, we use the notations \eqref{Pf_MT:sigma}--\eqref{Pf_MT:U0} with subscript $0$ replaced by $\ell$.
  Since $\Psi_\ell(\xi,r,t)\in\overline{S_T}$ implies $r=0$ by \eqref{Pf_MT:psi}, we further get
  \begin{align} \label{Pf_MT:cover}
    \overline{S_T} \subset \textstyle\bigcup_{\ell=1}^L\Psi_\ell(B_\ell\times\{0\}\times I_\ell).
  \end{align}
  Let $\delta:=\min_{1\leq\ell\leq L}\delta_\ell$.
  Then, $0<\delta<\rho_0$ and $\delta$ is independent of $t$.
  Also, let
  \begin{align*}
    \overline{\mathcal{N}_t} := \{x\in\mathbb{R}^n \mid -\delta \leq d(x,t) \leq \delta\} \subset \overline{\mathcal{N}_t^0}, \quad \overline{\mathcal{U}_T} := \textstyle\bigcup_{t\in[0,T]}\overline{\mathcal{N}_t}\times\{t\}.
  \end{align*}
  For each $(x,t)\in\overline{\mathcal{U}_T}$, the first relation of \eqref{E:Fermi} holds by \eqref{Pf_MT:Corr}.
  Thus, we can write
  \begin{align*}
    x = y+r\bm{\nu}(y,t), \quad y = \pi(x,t) \in \Gamma_t, \quad r = d(x,t) \in [-\delta,\delta].
  \end{align*}
  Then, since $(y,t)\in\overline{S_T}$, it follows from \eqref{Pf_MT:cover} that
  \begin{align*}
    t \in I_\ell, \quad (y,t) = \Psi_\ell(\xi,0,t) \quad\text{with some}\quad \ell\in\{1,\dots,L\}, \quad \xi\in B_\ell.
  \end{align*}
  By this fact, $r\in[-\delta,\delta]$, and $\delta\leq\delta_\ell$, we find that
  \begin{align*}
    (x,t) = \Psi_\ell(\xi,r,t) \in \Psi_\ell(V_\ell) = W_\ell \subset \overline{\mathcal{W}_T}.
  \end{align*}
  Thus, $\overline{\mathcal{U}_T}\subset\overline{\mathcal{W}_T}$.
  Since $\pi$ and $d$ are smooth on $\overline{\mathcal{W}_T}$, they are also smooth on $\overline{\mathcal{U}_T}$.
\end{proof}

Let us prove Lemma \ref{L:CoSur_Lq}.
To this end, we give a change of variables formula.

\begin{lemma} \label{L:CoV_Sur}
  Let $\mathbf{v}_\Gamma$ be the velocity of $\Gamma_t$ given by \eqref{E:Def_vSur} and
  \begin{align} \label{E:Def_JSur}
    \mathcal{J}_t^0(Y) := \exp\left(\int_0^t\mathrm{div}_\Gamma\mathbf{v}_\Gamma(\Phi_\tau^0(Y),\tau)\,d\tau\right), \quad (Y,t)\in\Gamma_0\times[0,T].
  \end{align}
  Then, for all $t\in[0,T]$ and $v\in L^1(\Gamma_t)$, we have
  \begin{align} \label{E:CoV_Sur}
    \int_{\Gamma_t}v(y)\,d\mathcal{H}^{n-1}(y) = \int_{\Gamma_0}v\bigl(\Phi_t^0(Y)\bigr)\mathcal{J}_t^0(Y)\,d\mathcal{H}^{n-1}(Y).
  \end{align}
\end{lemma}

\begin{proof}
  We first recall that the Leibniz formula
  \begin{align} \label{Pf_CVS:LF}
    \frac{d}{ds}\int_{\Gamma_s}f\,d\mathcal{H}^{n-1} = \int_{\Gamma_s}(\partial_\Gamma^\bullet f+f\,\mathrm{div}_\Gamma\mathbf{v}_\Gamma)\,d\mathcal{H}^{n-1}, \quad s\in[0,T]
  \end{align}
  holds for a $C^1$ function $f$ on $\overline{S_T}$ (see \cite[Lemma 2.2]{DziEll07}).

  Fix $t\in[0,T]$ and let $v\in C^1(\Gamma_t)$ first.
  We set $V(Y):=v\bigl(\Phi_t(Y)\bigr)$ and
  \begin{align*}
    F(Y,s) := V(Y)\exp\left(-\int_t^s\mathrm{div}_\Gamma\mathbf{v}_\Gamma(\Phi_\tau^0(Y),\tau)\,d\tau\right), \quad (Y,s) \in \Gamma_0\times[0,T].
  \end{align*}
  Note that $V$ is independent of $s$.
  Also, $F$ is $C^1$ on $\Gamma_0\times[0,T]$ and
  \begin{align*}
    \partial_sF(Y,s) = -F(Y,s)\mathrm{div}_\Gamma\mathbf{v}_\Gamma(\Phi_s^0(Y),s).
  \end{align*}
  Let $f(y,s):=F(\Phi_{-s}^0(y),s)$ for $(y,s)\in\overline{S_T}$.
  Then, by \eqref{E:Def_MtSur},
  \begin{align*}
    \partial_\Gamma^\bullet f(y,s) &= \frac{\partial}{\partial s}\Bigl(f(\Phi_s^0(Y),s)\Bigr)\Big|_{Y=\Phi_{-s}^0(y)} = \partial_sF(\Phi_{-s}^0(y),s) \\
    &= -f(y,s)\mathrm{div}_\Gamma\mathbf{v}_\Gamma(y,s).
  \end{align*}
  Applying this to \eqref{Pf_CVS:LF}, we get $\frac{d}{ds}\int_{\Gamma_s}f\,d\mathcal{H}^{n-1}=0$.
  Therefore,
  \begin{align*}
    \int_{\Gamma_t}f(y,t)\,d\mathcal{H}^{n-1}(y) = \int_{\Gamma_0}f(Y,0)\,d\mathcal{H}^{n-1}(Y).
  \end{align*}
  Moreover, since $V(Y)=v\bigl(\Phi_t^0(Y)\bigr)$, $\int_t^t=0$, and $-\int_t^0=\int_0^t$, we have
  \begin{align*}
    f(y,t) &= F(\Phi_{-t}^0(y),t) = V\bigl(\Phi_{-t}^0(y)\bigr) = v(y), \quad y\in\Gamma_t, \\
    f(Y,0) &= F(Y,0) = V(Y)\mathcal{J}_t^0(Y) = v\bigl(\Phi_t^0(Y)\bigr)\mathcal{J}_t^0(Y), \quad Y\in\Gamma_0.
  \end{align*}
  Thus, the formula \eqref{E:CoV_Sur} is valid when $v\in C^1(\Gamma_t)$.
  Also, we see that $C^1(\Gamma_t)$ is dense in $L^1(\Gamma_t)$ by standard localization and mollification arguments.
  Hence, we also have \eqref{E:CoV_Sur} for $v\in L^1(\Gamma_t)$ by a density argument.
\end{proof}

\begin{proof}[Proof of Lemma \ref{L:CoSur_Lq}]
  For the function $\mathcal{J}_{(\cdot)}^0$ given by \eqref{E:Def_JSur}, we see that
  \begin{align} \label{Pf_CSL:JSBd}
    c^{-1} \leq \mathcal{J}_t^0(Y) \leq c, \quad (Y,t) \in \Gamma_0\times[0,T],
  \end{align}
  since $\Phi_{(\cdot)}^0$ and $\mathbf{v}_\Gamma$ are smooth on the compact sets $\Gamma_0\times[0,T]$ and $\overline{S_T}$, respectively.
  Thus, the inequality \eqref{E:CoSur_Lq} follows from \eqref{E:CoV_Sur} and \eqref{Pf_CSL:JSBd}.
  Also, if we have
  \begin{align} \label{Pf_CSL:TGr}
    c^{-1}\bigl|\nabla_\Gamma V(Y)\bigr| \leq \bigl|\nabla_\Gamma v\bigl(\Phi_t^0(Y)\bigr)\bigr| \leq c|\nabla_\Gamma V(Y)|, \quad Y\in\Gamma_0,
  \end{align}
  then we get \eqref{E:CoSur_W1q} by \eqref{E:CoV_Sur}, \eqref{Pf_CSL:JSBd}, and \eqref{Pf_CSL:TGr}.

  Let us show \eqref{Pf_CSL:TGr}.
  Since $\Gamma_0$ is compact, we can take a finite number of smooth local parametrizations $Y_\ell\colon\overline{B}\to\Gamma_0$ such that $\Gamma_0=\bigcup_{\ell=1}^LY_\ell(\overline{B})$, where $\overline{B}$ is the unit closed ball in $\mathbb{R}^{n-1}$.
  Thus, it is sufficient to show \eqref{Pf_CSL:TGr} on each $Y_\ell(\overline{B})$.
  Now, we fix and suppress the subscript $\ell$ so that $Y\colon\overline{B}\to\Gamma_0$ is a smooth local parametrization.
  Let
  \begin{align*}
    y(\xi,s) := \Phi_s^0\bigl(Y(\xi)\bigr) \in \Gamma_s, \quad \theta_{ij}(\xi,s) := \partial_{\xi_i}y(\xi,s)\cdot\partial_{\xi_j}y(\xi,s)
  \end{align*}
  for $(\xi,s)\in\overline{B}\times[0,T]$ and $i,j=1,\dots,n-1$.
  Also, let $\theta:=(\theta_{ij})_{i,j}$.
  Then,
  \begin{align} \label{Pf_CSL:psdf}
    \mathbf{a}\cdot[\theta(\xi,s)\mathbf{a}] = |\mathbf{a}_{\xi,s}|^2 \geq 0, \quad (\mathbf{a},\xi,s) \in \mathbb{R}^{n-1}\times\overline{B}\times[0,T],
  \end{align}
  where $\mathbf{a}_{\xi,s}:=\sum_{i=1}^{n-1}a_i\partial_{\xi_i}y(\xi,s)$ with $\mathbf{a}=(a_1,\dots,a_{n-1})^{\mathrm{T}}$.
  Moreover, the vectors
  \begin{align*}
    \partial_{\xi_i}y(\xi,s) = [D\Phi_s^0]_{Y(\xi)}\partial_{x_i}Y(\xi) \in T_{y(\xi,s)}\Gamma_s, \quad i = 1,\dots,n-1
  \end{align*}
  are linearly independent for all $(\xi,s)\in\overline{B}\times[0,T]$, since $Y$ is a local parametrization of $\Gamma_0$ and the differential $[D\Phi_s^0]_{Y(\xi)}$ is a linear isomorphism by Assumption \ref{A:Flow_Sur}.
  Thus, the equality holds in \eqref{Pf_CSL:psdf} if and only if $\mathbf{a}=\mathbf{0}_{n-1}$ is the zero vector.
  This shows that $\theta$ is invertible on $\overline{B}\times[0,T]$.
  Let $\theta^{-1}=(\theta^{ij})_{i,j}$ be the inverse matrix of $\theta$.
  Then, $\theta^{-1}$ is smooth on $\overline{B}\times[0,T]$ since $\theta$ is so.
  Moreover, setting $\mathbf{a}=\theta^{-1}\mathbf{b}$ in \eqref{Pf_CSL:psdf}, we have
  \begin{align*}
    [\theta^{-1}(\xi,s)\mathbf{b}]\cdot\mathbf{b} > 0, \quad (\mathbf{b},\xi,s) \in \mathbb{R}^{n-1}\times\overline{B}\times[0,T], \quad \mathbf{b} \neq \mathbf{0}_{n-1}.
  \end{align*}
  Thus, denoting by $\mathbb{S}^{n-2}$ the unit sphere in $\mathbb{R}^{n-1}$, we see that
  \begin{align*}
    c^{-1} \leq [\theta^{-1}(\xi,s)\mathbf{b}]\cdot\mathbf{b} \leq c, \quad (\mathbf{b},\xi,s) \in \mathbb{S}^{n-2}\times\overline{B}\times[0,T]
  \end{align*}
  with a constant $c>0$ by the compactness of $\mathbb{S}^{n-2}\times\overline{B}\times[0,T]$, and this gives
  \begin{align} \label{Pf_CSL:equi}
    c^{-1}|\mathbf{b}|^2 \leq [\theta^{-1}(\xi,s)\mathbf{b}]\cdot\mathbf{b} \leq c|\mathbf{b}|^2, \quad (\mathbf{b},\xi,s) \in \mathbb{R}^{n-1}\times\overline{B}\times[0,T].
  \end{align}
  Now, we define and observe that
  \begin{align*}
    V^\sharp(\xi) := V\bigl(Y(\xi)\bigr) = v\Bigl(\Phi_t^0\bigl(Y(\xi)\bigr)\Bigr) = v\bigl(y(\xi,t)\bigr), \quad \xi\in\overline{B}.
  \end{align*}
  Thus, the tangential gradients of $V$ on $\Gamma_0$ and of $v$ on $\Gamma_t$ are expressed as
  \begin{align*}
    \nabla_\Gamma V\bigl(Y(\xi)\bigr) &= \sum_{i,j=1}^{n-1}\theta^{ij}(\xi,0)\partial_{\xi_i}V^\sharp(\xi)\partial_{\xi_j}y(\xi,0), \\
    \nabla_\Gamma v\bigl(y(\xi,t)\bigr) &= \sum_{i,j=1}^{n-1}\theta^{ij}(\xi,t)\partial_{\xi_i}V^\sharp(\xi)\partial_{\xi_j}y(\xi,t)
  \end{align*}
  for $\xi\in\overline{B}$, and writing $\nabla_\xi V^\sharp=(\partial_{\xi_1}V^\sharp,\dots,\partial_{\xi_{n-1}}V^\sharp)^{\mathrm{T}}$, we can deduce that
  \begin{align*}
    \bigl|\nabla_\Gamma V\bigl(Y(\xi)\bigr)\bigr|^2 &= [\theta^{-1}(\xi,0)\nabla_\xi V^\sharp(\xi)]\cdot\nabla_\xi V^\sharp(\xi), \\
    \bigl|\nabla_\Gamma v\bigl(y(\xi,t)\bigr)\bigr|^2 &= [\theta^{-1}(\xi,t)\nabla_\xi V^\sharp(\xi)]\cdot\nabla_\xi V^\sharp(\xi).
  \end{align*}
  Hence, applying \eqref{Pf_CSL:equi} to the right-hand sides, we find that \eqref{Pf_CSL:TGr} holds on $Y(\overline{B})$.
\end{proof}

\subsection{Uniform Poincar\'{e} inequality} \label{SS:PA_UP}
The goal of this subsection is to prove Lemma \ref{L:UP_MTD}.
To this end, we derive a uniform Poincar\'{e} inequality on $\Gamma_t$ under a weighted constraint.
Recall that we write $c$ for a general positive constant independent of $t$ and $\varepsilon$.

\begin{lemma} \label{L:UP_Sur}
  Let $q\in[1,\infty)$.
  There exists a constant $c>0$ such that
  \begin{align} \label{E:UP_Sur}
    \|v\|_{L^q(\Gamma_t)} \leq c\|\nabla_\Gamma v\|_{L^q(\Gamma_t)}
  \end{align}
  for all $t\in[0,T]$ and $v\in L^q(\Gamma_t)$ satisfying $\int_{\Gamma_t}g(y,t)v(y)\,d\mathcal{H}^{n-1}(y)=0$.
\end{lemma}

\begin{proof}
  Assume to the contrary that there exist $t_k\in[0,T]$ and $v_k\in L^q(\Gamma_{t_k})$ such that
  \begin{align*}
    \|v_k\|_{L^q(\Gamma_{t_k})} > k\|\nabla_\Gamma v_k\|_{L^q(\Gamma_{t_k})}, \quad \int_{\Gamma_{t_k}}g(y,t_k)v_k(y)\,d\mathcal{H}^{n-1}(y) = 0, \quad k\in\mathbb{N}.
  \end{align*}
  Let $V_k(Y):=v_k\bigl(\Phi_{t_k}^0(Y)\bigr)$ for $Y\in\Gamma_0$.
  Then, by \eqref{E:CoSur_Lq}, \eqref{E:CoSur_W1q}, and \eqref{E:CoV_Sur},
  \begin{align} \label{Pf_UPS:Cont}
    \|V_k\|_{L^q(\Gamma_0)} > ck\|\nabla_\Gamma V_k\|_{L^q(\Gamma_0)}, \quad \int_{\Gamma_0}g(\Phi_{t_k}^0(Y),t_k)V_k(Y)\mathcal{J}_{t_k}^0(Y)\,d\mathcal{H}^{n-1}(Y) = 0.
  \end{align}
  By $\|V_k\|_{L^q(\Gamma_0)}>0$, we may replace $V_k$ by $V_k/\|V_k\|_{L^q(\Gamma_0)}$ to assume that
  \begin{align*}
    \|V_k\|_{L^q(\Gamma_0)} = 1, \quad \|\nabla_\Gamma V_k\|_{L^q(\Gamma_0)} < \frac{1}{ck} \quad\text{and thus}\quad \|V_k\|_{W^{1,q}(\Gamma_0)} \leq c.
  \end{align*}
  Since $[0,T]$ is compact, we have $t_k\to t_0\in[0,T]$ as $k\to\infty$ up to a subsequence.
  Moreover, since the embedding $W^{1,q}(\Gamma_0)\hookrightarrow L^q(\Gamma_0)$ is compact (see e.g. \cite[Lemma 3.7]{Miu20_03}),
  \begin{align*}
    \lim_{k\to\infty}V_k = V_0 \quad\text{weakly in $W^{1,q}(\Gamma_0)$ and strongly in $L^q(\Gamma_0)$}
  \end{align*}
  with some $V_0\in W^{1,q}(\Gamma_0)$ up to a subsequence.
  Hence,
  \begin{align*}
    \|V_0\|_{L^q(\Gamma_0)} = 1, \quad \|\nabla_\Gamma V_0\|_{L^q(\Gamma_0)} = 0, \quad\text{i.e.}\quad \nabla_\Gamma V_0 = 0 \quad\text{a.e. on}\quad \Gamma_0,
  \end{align*}
  and it follows that $V_0$ is constant on $\Gamma_0$ (see e.g. \cite[Lemma 3.6]{Miu20_03}).
  On the other hand, since $\Phi_{(\cdot)}^0$ and $J_{(\cdot)}^0$ are smooth on $\Gamma_0\times[0,T]$, and since $g$ is smooth on $\overline{S_T}$, we send $k\to\infty$ in the second equality of \eqref{Pf_UPS:Cont} to get
  \begin{align*}
    V_0\int_{\Gamma_0}g(\Phi_{t_0}^0(Y),t_0)\mathcal{J}_{t_0}^0(Y)\,d\mathcal{H}^{n-1}(Y) = 0.
  \end{align*}
  By this equality, \eqref{E:G_Bdd}, and \eqref{Pf_CSL:JSBd}, we have $V_0=0$, but this contradicts $\|V_0\|_{L^q(\Gamma_0)}=1$.
  Hence, the inequality \eqref{E:UP_Sur} is valid.
\end{proof}

\begin{proof}[Proof of Lemma \ref{L:UP_MTD}]
  Let $u\in W^{1,q}(\Omega_t^\varepsilon)$ satisfy $(u,1)_{L^2(\Omega_t^\varepsilon)}=0$.
  Then, since
  \begin{align*}
    \mathcal{M}_\varepsilon u \in W^{1,q}(\Gamma_t), \quad \int_{\Gamma_t}g(y,t)\mathcal{M}_\varepsilon u(y)\,d\mathcal{H}^{n-1}(y) = \frac{1}{\varepsilon}\int_{\Omega_t^\varepsilon}u(x)\,dx = 0
  \end{align*}
  by Lemma \ref{L:Ave_W1q} and \eqref{E:Ave_Pair}, we can apply \eqref{E:UP_Sur} to $\mathcal{M}_\varepsilon u$ and use \eqref{E:Ave_GrLq} to get
  \begin{align*}
    \|\mathcal{M}_\varepsilon u\|_{L^q(\Gamma_t)} \leq c\|\nabla_\Gamma\mathcal{M}_\varepsilon u\|_{L^q(\Gamma_t)} \leq c\Bigl(\varepsilon^{1-1/q}\|u\|_{L^q(\Omega_t^\varepsilon)}+\varepsilon^{-1/q}\|\nabla u\|_{L^q(\Omega_t^\varepsilon)}\Bigr).
  \end{align*}
  Also, it follows from \eqref{E:CE_Lq} and \eqref{E:AvDf_Lq} that
  \begin{align*}
    \|u\|_{L^q(\Omega_t^\varepsilon)} &\leq \Bigl\|u-\overline{\mathcal{M}_\varepsilon u}\Bigr\|_{L^q(\Omega_t^\varepsilon)}+\Bigl\|\overline{\mathcal{M}_\varepsilon u}\Bigr\|_{L^q(\Omega_t^\varepsilon)} \\
    &\leq c\Bigl(\varepsilon\|u\|_{W^{1,q}(\Omega_t^\varepsilon)}+\varepsilon^{1/q}\|\mathcal{M}_\varepsilon u\|_{L^q(\Gamma_t)}\Bigr).
  \end{align*}
  Thus, combining the above inequalities and using $0<\varepsilon<1$, we obtain
  \begin{align*}
    \|u\|_{L^q(\Omega_t^\varepsilon)} \leq c_1\varepsilon\|u\|_{L^q(\Omega_t^\varepsilon)}+c_2\|\nabla u\|_{L^q(\Omega_t)}
  \end{align*}
  with some constants $c_1,c_2>0$ independent of $t$ and $\varepsilon$.
  Now, let
  \begin{align*}
    \varepsilon_1 := \min\left\{\frac{1}{2c_1},\varepsilon_0\right\} \in (0,\varepsilon_0],
  \end{align*}
  where $\varepsilon_0\in(0,1)$ is the constant given at the beginning of Section \ref{SS:Pre_Thin}.
  Then,
  \begin{align*}
    \|u\|_{L^q(\Omega_t^\varepsilon)} \leq \frac{1}{2}\|u\|_{L^q(\Omega_t^\varepsilon)}+c_2\|\nabla u\|_{L^q(\Omega_t)}
  \end{align*}
  for $0<\varepsilon\leq\varepsilon_1$ by $c_1\varepsilon_1\leq1/2$.
  Hence, the inequality \eqref{E:UP_MTD} follows.
\end{proof}

\subsection{Velocity of the moving thin domain} \label{SS:PA_VMTD}
Let $\Phi_{\pm(\cdot)}^\varepsilon$ be the mappings given in Lemma \ref{L:Flow_MTD}.
We prove Lemmas \ref{L:Vls_MTD} and \ref{L:Dist_Mat}.

\begin{proof}[Proof of Lemma \ref{L:Vls_MTD}]
  Let $X\in\overline{\Omega_0^\varepsilon}$.
  By \eqref{E:Def_FloMTD} and \eqref{E:Def_vMTD}, we have
  \begin{align*}
    \Phi_t^\varepsilon(X) = \Phi_t^0(Y)+r^\varepsilon\bm{\nu}(\Phi_t^0(Y),t), \quad \mathbf{v}^\varepsilon(\Phi_t^\varepsilon(X),t) = \partial_t\Phi_t^\varepsilon(X)
  \end{align*}
  for $t\in[0,T]$, where $Y=\pi(X,0)\in\Gamma_0$ and
  \begin{align} \label{Pf_VMTD:re}
    r^\varepsilon = r^\varepsilon(X,t) = \frac{g(\Phi_t^0(Y),t)}{g(Y,0)}\{d(X,0)-\varepsilon g_0(Y,0)\}+\varepsilon g_0(\Phi_t^0(Y),t).
  \end{align}
  Using \eqref{E:Def_vSur} and \eqref{E:Def_MtSur}, we compute $\partial_t\Phi_t^\varepsilon(X)$ to get
  \begin{align*}
    \mathbf{v}^\varepsilon(\Phi_t^\varepsilon(X),t) &= \mathbf{v}_\Gamma(\Phi_t^0(Y),t)+\frac{\partial r^\varepsilon}{\partial t}\bm{\nu}(\Phi_t^0(Y),t)+r^\varepsilon\partial_\Gamma^\bullet\bm{\nu}(\Phi_t^0(Y),t), \\
    \frac{\partial r^\varepsilon}{\partial t} &= \frac{\partial_\Gamma^\bullet g(\Phi_t^0(Y),t)}{g(Y,0)}\{d(X,0)-\varepsilon g_0(Y,0)\}+\varepsilon\partial_\Gamma^\bullet g_0(\Phi_t^0(Y),t).
  \end{align*}
  Next, we define and observe by \eqref{E:Fermi} and \eqref{E:Def_FloMTD} that
  \begin{align} \label{Pf_VMTD:xy}
    x := \Phi_t^\varepsilon(X) \in \overline{\Omega_t^\varepsilon}, \quad y := \Phi_t^0(Y) = \pi(\Phi_t^\varepsilon(X),t) = \pi(x,t) \in \Gamma_t.
  \end{align}
  Then, the above expressions of $\mathbf{v}^\varepsilon$ and $\partial r^\varepsilon/\partial t$ read
  \begin{align} \label{Pf_VMTD:vtr}
    \begin{aligned}
      \mathbf{v}^\varepsilon(x,t) &= \mathbf{v}_\Gamma(y,t)+\frac{\partial r^\varepsilon}{\partial t}\bm{\nu}(y,t)+r^\varepsilon\partial_\Gamma^\bullet\bm{\nu}(y,t), \\
      \frac{\partial r^\varepsilon}{\partial t} &= \frac{\partial_\Gamma^\bullet g(y,t)}{g(Y,0)}\{d(X,0)-\varepsilon g_0(Y,0)\}+\varepsilon\partial_\Gamma^\bullet g_0(y,t).
    \end{aligned}
  \end{align}
  Moreover, since $r^\varepsilon=d(\Phi_t^\varepsilon(X),t)$ by \eqref{E:Fermi} and \eqref{E:Def_FloMTD}, we have
  \begin{align*}
    d(x,t) = r^\varepsilon = \frac{g(y,t)}{g(Y,0)}\{d(X,0)-\varepsilon g_0(Y,0)\}+\varepsilon g_0(y,t)
  \end{align*}
  by \eqref{Pf_VMTD:re} and \eqref{Pf_VMTD:xy}.
  From this relation, we deduce that
  \begin{align*}
    \frac{1}{g(Y,0)}\{d(X,0)-\varepsilon g_0(Y,0)\} = \frac{1}{g(y,t)}\{d(x,t)-\varepsilon g_0(y,t)\}.
  \end{align*}
  Hence, we can write $\partial r^\varepsilon/\partial t$ in \eqref{Pf_VMTD:vtr} as
  \begin{align} \label{Pf_VMTD:dtdr}
    \frac{\partial r^\varepsilon}{\partial t} = \frac{\partial_\Gamma^\bullet g(y,t)}{g(y,t)}\{d(x,t)-\varepsilon g_0(y,t)\}+\varepsilon\partial_\Gamma^\bullet g_0(y,t),
  \end{align}
  and we apply this relation and $r^\varepsilon=d(x,t)$ to $\mathbf{v}^\varepsilon$ in \eqref{Pf_VMTD:vtr} to find that
  \begin{align*}
    \mathbf{v}^\varepsilon(x,t) = \mathbf{v}_\Gamma(y,t)+d(x,t)\mathbf{v}_1(y,t)+\varepsilon\mathbf{v}_2(y,t),
  \end{align*}
  where $\mathbf{v}_1,\mathbf{v}_2\colon\overline{S_T}\to\mathbb{R}^n$ are smooth vector fields given by
  \begin{align*}
    \mathbf{v}_1 := \frac{\partial_\Gamma^\bullet g}{g}\,\bm{\nu}+\partial_\Gamma^\bullet\bm{\nu}, \quad \mathbf{v}_2 := \left(\partial_\Gamma^\bullet g_0-\frac{\partial_\Gamma^\bullet g}{g}\,g_0\right)\bm{\nu} \quad\text{on}\quad \overline{S_T}.
  \end{align*}
  Now, since $y=\pi(x,t)$, the above expression can be written as
  \begin{align*}
    \mathbf{v}^\varepsilon(x,t) = \bar{\mathbf{v}}_\Gamma(x,t)+d(x,t)\bar{\mathbf{v}}_1(x,t)+\varepsilon\bar{\mathbf{v}}_2(x,t), \quad (x,t)\in\overline{Q_T^\varepsilon}.
  \end{align*}
  Thus, we get the first inequality of \eqref{E:Vls_MTD} by $|d|\leq c\varepsilon$ in $\overline{Q_T^\varepsilon}$.
  Next, we set
  \begin{align*}
    \mathbf{G}_{\Gamma,g} := \nabla_\Gamma\mathbf{v}_\Gamma+\bm{\nu}\otimes\mathbf{v}_1 = \nabla_\Gamma\mathbf{v}_\Gamma+\bm{\nu}\otimes\left[\frac{\partial_\Gamma^\bullet g}{g}\bm{\nu}+\partial_\Gamma^\bullet\bm{\nu}\right] \quad\text{on}\quad \overline{S_T}.
  \end{align*}
  Since $\nabla(d\bar{\mathbf{v}}_1)=(\nabla d)\otimes\bar{\mathbf{v}}_1+d\nabla\bar{\mathbf{v}}_1=\bar{\bm{\nu}}\otimes\bar{\mathbf{v}}_1+d\nabla\bar{\mathbf{v}}_1$ in $\overline{Q_T^\varepsilon}$ by \eqref{E:Fermi}, we have
  \begin{align*}
    \nabla\mathbf{v}^\varepsilon = \nabla\bar{\mathbf{v}}_\Gamma+\nabla(d\bar{\mathbf{v}}_1)+\varepsilon\nabla\bar{\mathbf{v}}_2 = \overline{\mathbf{G}}_{\Gamma,g}+\Bigl(\nabla\bar{\mathbf{v}}_\Gamma-\overline{\nabla_\Gamma\mathbf{v}_\Gamma}\Bigr)+d\nabla\bar{\mathbf{v}}_1+\varepsilon\nabla\bar{\mathbf{v}}_2
  \end{align*}
  in $\overline{Q_T^\varepsilon}$.
  Thus, the second inequality of \eqref{E:Vls_MTD} holds by \eqref{E:CEGr_NB} and $|d|\leq c\varepsilon$ in $\overline{Q_T^\varepsilon}$.
  Also,
  \begin{align*}
    \mathrm{tr}[\mathbf{G}_{\Gamma,g}] = \mathrm{div}_\Gamma\mathbf{v}_\Gamma+\frac{\partial_\Gamma^\bullet g}{g}|\bm{\nu}|^2+\bm{\nu}\cdot\partial_\Gamma^\bullet\bm{\nu} = \mathrm{div}_\Gamma\mathbf{v}_\Gamma+\frac{\partial_\Gamma^\bullet g}{g} = \sigma_{\Gamma,g} \quad\text{on}\quad \overline{S_T}
  \end{align*}
  by $|\bm{\nu}|=1$ and $\bm{\nu}\cdot\partial_\Gamma^\bullet\bm{\nu}=\partial_\Gamma^\bullet(|\bm{\nu}|^2/2)=0$ on $\overline{S_T}$.
  By this relation and the second inequality of \eqref{E:Vls_MTD}, we obtain the last inequality of \eqref{E:Vls_MTD}.
\end{proof}

\begin{proof}[Proof of Lemma \ref{L:Dist_Mat}]
  For $X\in\overline{\Omega_0^\varepsilon}$, let $Y=\pi(X,0)\in\Gamma_0$ and $r^\varepsilon$ be given by \eqref{Pf_VMTD:re}.
  Also, let $x$ and $y$ be given by \eqref{Pf_VMTD:xy}, and let $\eta$ be a function on $\overline{S_T}$.
  Then,
  \begin{align*}
    d(\Phi_t^\varepsilon(X),t) = r^\varepsilon, \quad \bar{\eta}(\Phi_t^\varepsilon(X),t) = \bar{\eta}(x,t) = \eta(y,t) = \eta(\Phi_t^0(Y),t)
  \end{align*}
  by \eqref{E:Fermi}, \eqref{E:Def_FloMTD}, and $\bar{\eta}(\cdot,t)=\eta(\pi(\cdot,t),t)$.
  Thus, we see by \eqref{E:Def_MtMTD} that
  \begin{align*}
    \partial_\varepsilon^\bullet d(x,t) &= \frac{\partial}{\partial t}\Bigl(d(\Phi_t^\varepsilon(X),t)\Bigr) = \frac{\partial r^\varepsilon}{\partial t}, \\
    \partial_\varepsilon^\bullet\bar{\eta}(x,t) &= \frac{\partial}{\partial t}\Bigl(\bar{\eta}(\Phi_t^\varepsilon(X),t)\Bigr) = \frac{\partial}{\partial t}\Bigl(\eta(\Phi_t^0(Y),t)\Bigr),
  \end{align*}
  and we use \eqref{E:Def_MtSur} and \eqref{Pf_VMTD:dtdr} to the right-hand sides to get \eqref{E:Dist_Mat} and \eqref{E:CE_Mat}.
\end{proof}

%%% Section 9 %%%
\section{Conclusion} \label{S:Concl}
We studied the parabolic $p$-Laplace equation \eqref{E:pLap_MTD} with $p>2$ in the moving thin domain $\Omega_t^\varepsilon$ of the form \eqref{E:Def_MTD}.
When $\varepsilon\to0$ and $\Omega_t^\varepsilon$ shrinks to the given closed moving hypersurface $\Gamma_t$, we showed that the average in the thin direction of a weak solution to \eqref{E:pLap_MTD} converges weakly in an appropriate function space on $\Gamma_t$.
Moreover, we derived the limit problem \eqref{E:pLap_Lim} on $\Gamma_t$ rigorously by characterizing the limit function as a unique weak solution.
The proof relies on the abstract framework of evolving Bochner spaces developed in \cite{AlCaDjEl23,AlElSt15_PM}, the properties of the weighted average operator, and a monotonicity argument.
In particular, we determined the weak limit of the averaged nonlinear gradient term by choosing test functions appropriately and making use of the strong convergence of the averaged weak solution due to the Aubin-Lions lemma on evolving Bochner spaces.
Also, we combined the two equations of the limit problem \eqref{E:pLap_Lim} into one weak form and used it to establish the uniqueness of a weak solution to \eqref{E:pLap_Lim}.

As explained in Section \ref{S:Intro}, the limit problem \eqref{E:pLap_Lim} obtained here can be seen as a local mass conservation law on the moving hypersurface $\Gamma_t$ in which the rate of mass moving on $\Gamma_t$ is balanced with a normal flux.
It extends the local mass conservation law described in \cite{DziEll07,DziEll13_AN} which only takes into account a tangential flux.
This paper gives the first result on a rigorous thin-film limit of nonlinear equations in moving thin domains around given moving hypersurfaces, and also proposes a new kind of conservation law.

%%% Acknowledgments %%%
\section*{Acknowledgments}
The work of the author was supported by JSPS KAKENHI Grant Number 23K12993.

%%% References %%%
\bibliographystyle{abbrv}
\bibliography{pLap_MoThin_Ref}

\end{document}